\documentclass[12pt, a4paper, reqno]{amsart}

\usepackage[utf8]{inputenc} 
\usepackage[T1]{fontenc}  
\usepackage{newtxtext}   
\usepackage{eulervm}    
\usepackage{amsmath,amssymb}

\usepackage[a4paper, margin=1in]{geometry}
\usepackage{graphicx}
\usepackage{tikz}
\usepackage{overpic}
\usepackage{shadow}

\usepackage{booktabs}
\usepackage{threeparttable}
\usepackage{makecell}
\usepackage{multirow}

\usepackage{subcaption} 
\usepackage{mathrsfs}  
\usepackage{csquotes}  

\usepackage{natbib}   

\usepackage{hyperref}
\hypersetup{
  plainpages = True,
  pdfstartview= FitH,
  bookmarksopen=true,
  pdfpagemode = none,
  colorlinks = true,
  linkcolor  = blue,
  citecolor  = blue,
  urlcolor  = blue
}

\def\le{\leqslant}
\def\ge{\geqslant}
\def\tr#1{\left\lfloor #1\right\rfloor}

\def\cl#1{\lceil #1\rceil}

\def\lpa#1{\bigl({#1}\bigr)}
\def\Lpa#1{\Bigl({#1}\Bigr)}

\def\llpa#1{\biggl({#1}\biggr)}
\def\dd#1{{\,\rm d}#1}
\def\ve{\varepsilon}
\def\eqtext#1{\quad\text{#1}\quad}
\def\lbb{\lambda_{\text{bb}}}
\def\lbe{\lambda_{\text{be}}}
\def\leb{\lambda_{\text{eb}}}
\def\lee{\lambda_{\text{ee}}}

\DeclareRobustCommand{\Stirling}{\genfrac\{\}{0pt}{}}
\DeclareRobustCommand{\stirling}{\genfrac[]{0pt}{}}

\newtheorem{thm}{Theorem}[section]
\newtheorem{remark}{Remark}[section]
\newtheorem{lemma}{Lemma}[section]
\newtheorem{cor}{Corollary}[section]
\newtheorem{prop}{Proposition}[section]

\graphicspath{{./ele-stir2-figs/}}

\def\and{\mbox{\quad and \quad}}
\begin{document}

\title[Elementary asymptotics of Stirling numbers]{Elementary asymptotics for the Stirling numbers of the second kind:\\ The central range}

\let\origmaketitle\maketitle
\def\maketitle{
 \begingroup
 \def\uppercasenonmath##1{} 
 \let\MakeUppercase\relax 
 \origmaketitle
 \endgroup
}

\thanks{The research of the first author was partially supported by Taiwan Ministry of Science and Technology Grant MOST 108-2118-M-001-005-MY3. The second author was partially supported by the NSTC Grant 114-2118-M-031-002, and conducted part of this research during a post-doctoral appointment at the Institute of Statistical Science, Academia Sinica. The third author was supported by NSTC Grant 112-2811-M-001-002 during his 2023 appointment as Visiting Associate Professor at the Institute of Statistical Science, Academia Sinica, and subsequently by the Institute during an extended research appointment there. He thanks the Institute for its hospitality and support.}

\author[]{Hsien-Kuei Hwang}
\address[Hsien-Kuei Hwang]{Institute of Statistical Science, Academia Sinica, Taipei, 115, Taiwan}
\email{\textbf{hkhwang@stat.sinica.edu.tw}}

\author[]{Chong-Yi Li}
\address[Chong-Yi Li]{Department of Mathematics, Soochow University, Taipei, 111, Taiwan}
\email{\textbf{chongyili356@gmail.com}}

\author[]{Vytas Zacharovas}
\address[Vytas Zacharovas]{Institute of Computer Science, Vilnius University, Naugarduko 24, LT-03225 Vilnius, Lithuania}
\email{\textbf{vytas.zacharovas@mif.vu.lt}}

\date{\today}

\maketitle

\begin{abstract}

We derive the local and central limit theorems for the Stirling numbers of the second kind by elementary means, obtaining as corollaries effective asymptotic estimates for the Bell numbers and for the moments of the distribution. We also develop asymptotic expansions along several directions, all based on a novel finite-differencing approach---the first self-contained elementary justification of such expansions.

\end{abstract}
\tableofcontents

\bigskip


\section{Introduction}

The Stirling numbers of the second kind (also known as Stirling partition numbers; see \href{https://oeis.org/A008277}{OEIS A008277}), named after James Stirling (1692–1770) by Nielsen \cite{Nielsen1906}, are defined by the sieve formula
\begin{align}\label{E:Snk-sum}
	\Stirling{n}{k}
	= \frac{1}{k!} \sum_{0 \le j \le k}
	\binom{k}{j} (-1)^j (k - j)^n 
	\qquad (n \ge 0,\ 0 \le k \le n),
\end{align}
which counts the number of ways to partition the set $\{1,2,\dots,n\}$ into $k$ non-empty blocks. This identity follows from the standard inclusion-exclusion principle; see \cite[\S 5.1]{Comtet1974}. Since their formal introduction by Stirling in 1730 \cite{Stirling1764}, these numbers have been extensively studied from a variety of perspectives; see, for example, \cite{Charalambides2005, Duran2013, Goldstine2012, Knuth1998, Mezo2019, Simovici2021}. In this paper, we analyze the inclusion–exclusion alternating sum in \eqref{E:Snk-sum} directly and develop systematic techniques for deriving its asymptotic expansion. This approach is of methodological interest in its own right, as alternating sums of this form frequently arise in applications---particularly in the context of inclusion-exclusion. A detailed historical and technical survey of the asymptotics of $\Stirling{n}{k}$ since the 1780s will be provided in a future companion paper.

\textbf{Notation.} Throughout this paper, $W(x)$ denotes the principal branch of the Lambert $W$-function, i.e., the solution to the equation $w e^w = x$ that is positive when $x > 0$. As $x \to \infty$, it is known that
\[
	W(x) = \log \frac{x}{\log x} + \frac{\log \log x}{\log x}
	+ \frac{(\log \log x)(\log \log x - 2)}{2(\log x)^2}
	+ O\left(\frac{(\log \log x)^3}{(\log x)^3}\right);
\]
see \cite{Corless1996} for further details. 

We use the abbreviations CLT and LLT for the \emph{central} and \emph{local limit theorems}, respectively. For convenience, we also write $\rho := \frac nk$, and
\[
	\lambda = \lambda(n) := ke^{-\rho} = ke^{-\frac nk}, \eqtext{or}
	e^{-\rho} = e^{-\frac nk} = \frac{\lambda}k;
\]
thus $k= \frac{n}{W(\frac{n}{\lambda})}$. In this paper, we work mainly on the range $0\le k\le \frac{2n}{\log n}$, which means that $0\le \lambda\le\frac{2\sqrt{n}}{\log n}$.

\section{A simple elementary approach to the LLT}

Assume a uniform distribution on the set of all partitions of $n$ elements; let $X_n$ denote the number of blocks in a randomly chosen set partition. Then
\[
  \mathbb{P}(X_n=k)
	= \frac{1}{B_n}\Stirling{n}{k},\qquad
	(n\ge1; 1\le k\le n),
\]
where $B_n := \sum_{1\le k\le n}\Stirling{n}{k}$ are the Bell numbers (see \href{https://oeis.org/A000110}{OEIS A000110}). Let
\[
	\Phi(x) := \frac1{\sqrt{2\pi}}
	\int_{-\infty}^x e^{-\frac12t^2}\dd t
\]
denote the standard normal distribution function, and define
\begin{align}\label{E:mu-var}
 	\mu_n := \frac{n}{W(n)},
 	\eqtext{and} 
 	\sigma_n^2 := \frac{n}{W(n)(W(n)+1)}.
 \end{align}
\begin{thm} \label{T:llt}
The Stirling partition numbers satisfy the CLT:
\begin{align}\label{E:clt}
	\sup_{x\in\mathbb{R}}
	\biggl|\mathbb{P}\Lpa{\frac{X_n-\mu_n}{\sigma_n}\le x}
	-\Phi(x) \biggr|
	= O\lpa{n^{-\frac12}\log n},
\end{align}
and the LLT:
\begin{align}\label{E:llt}
	\mathbb{P}(X_n=\tr{\mu_n +x\sigma_n})
	=\frac{e^{-\frac12x^2}}{\sqrt{2\pi}\,\sigma_n}
	\Lpa{1+O\Lpa{\frac{1+|x|^3}{\sigma_n}}},
\end{align}
uniformly for $x=o(\sigma_n^{\frac13})$.
\end{thm}

The CLT (without rate) was first established by Harper \cite{Harper1967}, where he also mentions the LLT without proof; see also Bender \cite{Bender1973} for an approach to obtain general CLTs and LLTs, and Canfield \cite{Canfield1977} for a modification of the sufficient conditions. Menon \cite{Menon1981} also derived the LLT by an elementary approach developed earlier in \cite{Menon1973} and similar to ours, but his proof is incomplete. Unlike Menon's argument, we do not rely on Bonferroni inequality, which may not be available in more general situations. For a more detailed comparative discussion, see Section~\ref{S:comparison}.

Since our proof of Theorem~\ref{T:llt} is unexpectedly simple, we present it first in the end of this section before addressing the implications of the theorem, followed by further refinements (in Section~\ref{S:PC}), and a comparison of our results with known ones in the literature (in Section~\ref{S:comparison}). 

\subsection{Asymptotic nature of the sieve formula, I} \label{S:sieve}
Spelling out the first few terms of \eqref{E:Snk-sum}, we see that 
\begin{align*}
  \Stirling{n}{k}
  = \frac{k^n}{k!}\left(1
  -k\Lpa{1-\frac1k}^n
  +\binom{k}{2} \Lpa{1-\frac2k}^n
  +\cdots\right),
\end{align*}
which is itself an asymptotic expansion for $1\le k\le k_0$, where $k_0$ is chosen such that each term in the parentheses on the right-hand side is of a smaller order than its previous one, or when
\begin{align*}
  k_0e^{-\frac n{k_0}}=o(1),
  \eqtext{namely}
  k_0:=\frac{n}{W(\xi_n n)}
  \quad\text{with}\quad \xi_n\to\infty.
\end{align*}
This simple argument covers already the range of $k$ up to 
\begin{align*}
  1\le k\le k_0=\frac{n}{W(\xi_n n)}
  =\frac{n}{\log \frac{n}{\log n}+\log \xi_n},
\end{align*}
in which $\Stirling{n}{k}$ satisfies uniformly the asymptotic approximation
\[
  \Stirling{n}{k}
  = \frac{k^n}{k!}\Lpa{1+O\lpa{ke^{-\frac nk}}}.
\]
This estimate (without error term) was first derived by Jordan for the case $k=O(1)$ \cite{Jordan1933}; see also \cite[\S 59]{Jordan1947}. The broader range $k\le k_0$ was later established by Korshunov \cite{Korshunov1983}, who refined an earlier estimate by Bernstein (originally appearing in his 1934 monograph \cite{Bernstein1934}). Korshunov's range $k\le k_0$ can be expressed as
\begin{align}\label{E:k00}
	\frac{n}{k}-W(n)\to\infty \eqtext{or}
	\frac{n}{k}-\log \frac{n}{\log n}\to\infty,
\end{align}
which falls slightly short of the region where the mean (and the mode) of the distribution is concentrated, namely around $\frac{n}{W(n)} + O(1)$, or when $\frac{n}{k} - W(n) = O(1)$.

\subsection{From $\frac{n}{\log n-\log\log n}$ to $\frac{n}{\log n-2\log\log n}$}
We now show that a simple extension of the above argument will provide the estimate required to establish the LLT in a wider range $k=\frac{n}{W(n)}+O(n^{\frac12+\ve})$, or $\frac nk-W(n) = O(n^{-\frac12+\ve} (\log n)^2)$.

Our elementary approach relies on the following inequality.

\begin{lemma}[{\cite[Lemma 5]{Bai2001}}]
For $n\ge1$ and $0\le t\le 1$,
\begin{align}\label{E:elem-ineq}
	|(1-t)^n - e^{-nt}| \le nt^2 e^{-nt}.
\end{align}
\end{lemma}
\begin{proof}
We have	
\begin{align*}
	0\ge(1-t)^n - e^{-nt}
	= e^{-nt}\lpa{\lpa{e^t(1-t)}^n-1}
	\ge e^{-nt}\lpa{(1-t^2)^n-1}\ge -nt^2 e^{-nt},
\end{align*}
by the inequalities $e^t\ge 1+t$ and (Bernoulli's inequality) 
$(1-y)^n-1\ge -ny$ for $n\ge1$ and $0\le y\le 1$. 	
\end{proof}

Recall that $\lambda =\lambda(n,k) := ke^{-\frac nk}$ is increasing in  $k$ for fixed $n$; representative values are listed below.
\begin{center}
\begin{tabular}{c|cccccccc}
$k$ & $1$ & $\frac{n}{\log n}$ & $\frac{n}{W(n)}$
& $(1+c)\frac{n}{\log n}$ 
& $\ve n$ & $n$ \\ \hline
$\lambda$ &	$e^{-n}$ & $\frac1{\log n}$ & $1$
& $(1+c)\frac{n^{c/(1+c)}}{\log n}$ 
& $\ve e^{-1/\ve}n$ & $e^{-1}n$
\end{tabular}	
\end{center}
Here and below, generic constants such as $c$ and $\ve$ are positive and may vary from one occurrence to another, unless explicitly fixed. For convenience, define, \emph{throughout this paper},
\begin{align}\label{E:Snk-def}
	S(n,k) := \Stirling{n}{k}\frac{k!}{k^n},
\end{align} 
and 
\begin{align}\label{E:k1-0}
	k_1 := \frac{n}{W\lpa{\frac{n}{\lambda^+}}}
	\eqtext{with} 
	\lambda^+ = W\Lpa{\frac{\sqrt{n}}{(\log n)^{1+\ve}}}.
\end{align}

\begin{prop} \label{P:R1} Uniformly for $1\le k\le k_1$, 
\begin{align}\label{E:Snk-central}
	S(n,k)
	= \Lpa{1-\frac{\lambda}{k}}^k\lpa{1+E_{n,k}},
\end{align}	
where 
\begin{align}\label{E:Enk-est}
	E_{n,k} = O\Lpa{\frac{n}{k^2}\, 
	\lambda(\lambda+1)e^{2\lambda}}.
\end{align}
\end{prop}
Note that 
\begin{align*}
	W\Lpa{\frac{n}{\lambda^+}}
	&= \log n -2\log\log n +\log 2
	+ \frac{2(3+\ve)\log \log n +O(1)}{\log n}.
\end{align*}
Comparing this bound with $k_0$, we observe that the range is extended only slightly from 
\begin{align}\label{E:k0}
  k_0 = \frac{n}{\log n-\log\log n + \xi_n}
	\eqtext{for any}\xi_n\to\infty,
\end{align}
to 
\begin{align}\label{E:k1}
	k_1 = \frac{n}{\log n-2\log\log n
	+\log 2+\frac{2(3+\ve)\log \log n +O(1)}{\log n}},
\end{align}
yielding a net difference in the denominator that is asymptotic to 
\[
	\log \log n+\xi_n -\log2-\frac{2(3+\ve)\log \log n }{\log n}.
\]
This quantity, when neglecting the unspecified term $\xi_n$, is nevertheless negative for $n\le74\,316$ when \(\varepsilon=1\). Despite the marginal gain in range, this refinement is sufficient for our LLT application in \eqref{E:llt}.

\begin{proof}[Proof of Proposition~\ref{P:R1}]
By applying \eqref{E:elem-ineq} to the factor $(1-\frac jk)^n$ in \eqref{E:Snk-sum}, we obtain
\begin{align}
  S(n,k)
  &= \sum_{0\le j\le k}
	\binom{k}{j}(-1)^j e^{-\frac nk j}
	\Lpa{1+O\Lpa{\frac{j^2 n}{k^2}}} \nonumber \\
	&= \lpa{1-e^{-\frac nk}}^k\lpa{1 + O(E_{n,k})},
	\label{E:appr1}
\end{align}
where
\begin{align}\label{E:Enk-1}
	\lpa{1-e^{-\frac nk}}^kE_{n,k}
	&= \frac{n}{k^2}\sum_{0\le j\le k}
	\binom{k}{j}j^2 e^{-\frac nkj}
	= \frac nk\lpa{1+e^{-\frac nk}}^{k-2}
	e^{-\frac nk}\lpa{ke^{-\frac nk}+1}.
\end{align}
This error bound is not sharp for $\lambda\to\infty$, primarily due to the absence of the alternating factor $(-1)^j$ in the $O$-term. Nevertheless, it suffices for the purposes of our limit theorems. By the inequality 
\[
	\frac{1 + x}{1 - x} 
	= 1 + \frac{2x}{1 - x} \le e^{\frac{2x}{1-x}}
	\eqtext{for } 0 < x < 1,
\]
we obtain
\begin{align}\label{E:Enk-est0}
	E_{n,k} = O\Lpa{\frac{n}{k^2}\, 
	\lambda(\lambda+1)e^{\frac{2\lambda}{1-\lambda/k}}},
\end{align}
which then yields \eqref{E:Enk-est}. Since $k=\frac{n} {W(\frac{n}{\lambda})}$, we see that the dominant term in \eqref{E:Snk-central} is of order
\[
	\lpa{1-e^{-\frac nk}}^k=\Lpa{1-\frac{\lambda}k}^k
	= \exp\Lpa{-\lambda
	+O\Lpa{\frac{\lambda^2W(\frac n{\lambda})}{n}}},
\]
while the error term is bounded above by 
\begin{align*}
	E_{n,k}
	&= O\llpa{\frac{W(\frac n{\lambda})^2}{n}\, 
	\lambda(\lambda+1)e^{2\lambda}}.
\end{align*}
If $\lambda=O(1)$, then $E_{n,k}=o(1)$. On the other hand, if $\lambda\to\infty$ and satisfies $\lambda \le \lambda^+$ (defined in \eqref{E:k1-0}), then $e^{\lambda} = O(\sqrt{n} (\log n)^{-2-\ve})$, and
\[
	E_{n,k}
	= O\Lpa{\frac{\sqrt{n}}{(\log n)^{2+\ve}}
	\cdot \frac{(\log n)^{2-\ve}}{\sqrt{n}}}
	= O\lpa{(\log n)^{-2\ve}} =o(1),
\] 
uniformly for $\lambda\le \lambda^+$. Thus \eqref{E:appr1} is an asymptotic approximation for $1\le k\le k_1$.
\end{proof}

\begin{remark}
While uniformly valid for $1\le k\le k_1$, \eqref{E:Snk-central} is more useful when $\frac{n}{k^2}\to0$ or when $\frac{k}{\sqrt{n}} \to\infty$ because when $\lambda\to0$,
\[
	S(n,k)=\Lpa{1-\frac{\lambda}k}^k\left(1+E_{n,k}\right)
	= 1+O\lpa{\lambda + nk^{-2}\lambda} = 1+o(1).
\]
\end{remark}

\begin{remark}
When $ke^{-\frac{2n}k} =\frac{\lambda^2}k\to0$ or when $k\le 2k_0$ (see \eqref{E:k0}), then 
\[
	\Lpa{1-\frac{\lambda}k}^k = e^{-\lambda}(1+o(1)).
\]
The exponential form on the right-hand side (in approximating $S(n,k)$) appeared first in Laplace's 1783 memoir \cite[p.~337]{Laplace1786} (where his $i$ is our $n$ and his $n$ is our $k$), and later rederived by Cayley in \cite{Cayley1888} by a different formal approach; see also \cite{David1962, Laplace1812, Laplace1820} and Section~\ref{S:comparison} for more details.
\end{remark}

\subsection{Asymptotics of the Bell numbers $B_n$}
In this section, we derive an asymptotic approximation to $B_n$, beginning with the following uniform estimate.
\begin{lemma} Uniformly for $1\le k\le n$
\begin{align}\label{E:Snk-ua}
	\Stirling{n}{k}\le \frac{k^n}{k!}.
\end{align}
\end{lemma}
\begin{proof}
We have
\begin{align*}
	\Stirling{n}{k}
	= \frac{n!}{k!}[z^n](e^z-1)^k
	\le \frac{n!}{k!}[z^n]e^{kz}
	= \frac{k^n}{k!},
\end{align*}
where the symbol $[z^n]f(z)$ represents the coefficient of $z^n$ in the Taylor expansion of $f$. A proof by a sieve argument (or by Bonferroni inequality) is also straightforward.
\end{proof}

For convenience, we use, throughout this paper, the abbreviation $\omega_n=W(n)$.
\begin{lemma} \label{L:k-central}
Uniformly for $k=\mu_n+x\sigma_n$ with $x=o(n^{\frac16})$, where $\mu_n$ and $\sigma_n$ are defined in \eqref{E:mu-var}, 
\begin{align*}
	\begin{split}
	\frac{k^n}{k!}
	&= \frac{\sqrt{\omega_n}\,e^{(\omega_n-1+\frac1{\omega_n})n 
	-\frac12x^2}}{\sqrt{2\pi n}}
	\Bigl(1+\frac{(2\omega_n+1)x^3-3(\omega_n+1)x}
	{6\sigma_n(\omega_n+1)^2}
	+O\Lpa{\frac{1+x^6+\omega_n}{\sigma_n^2\omega_n^2}}\Bigr),\\
	\lpa{1-e^{-\frac nk}}^k
	&= e^{-1}\Lpa{1-\frac{x}{\sigma_n}
	+O\Lpa{\frac{1+x^2}{\sigma_n^2}}}.
	\end{split}
\end{align*}
\end{lemma}
\begin{proof}
By Stirling's formula,
\[
	n\log k -\log k!
	= \Lpa{n-k-\frac12}\log k +k -\frac12\log 2\pi
	-\frac1{12k}+O\lpa{k^{-2}}.
\]
Substituting $k=\mu_n+x\sigma_n$ and using the relation $\log W(n) = \log n-W(n)$, we obtain the above asymptotic expansion for $\frac{k^n}{k!}$ after routine expansions and simplifications. The proof of the second expansion proceeds similarly.
\end{proof}

\begin{prop}[\cite{deBruijn1981,Moser1955,Szekeres1957}] 
\label{P:Bell-asympt}
For large $n$ (with the convention $\omega_n = W(n)$)
\begin{align}\label{E:Bn-asympt}
  B_n = \frac{e^{(\omega_n-1+\frac1{\omega_n})n-1}}
	{\sqrt{\omega_n+1}}\lpa{1+O\lpa{n^{-1}(\log n)^2}}.
\end{align}	
\end{prop}
Applying the saddle-point method to Cauchy's integral representation yields the stronger expansion:
\begin{align}\label{E:Bn}
  B_n = \frac{e^{(\omega_n-1+\omega_n^{-1})n-1}}
	{\sqrt{\omega_n+1}}\Lpa{1-\frac{\omega_n^2(2\omega_n^2 
	+ 7\omega_n + 10)}
	{24(\omega_n + 1)^3n}
	+O\lpa{n^{-2}(\log n)^2}};
\end{align}
see \cite[\S 6.2]{deBruijn1981}, \cite{Dou2022, Moser1955, Szekeres1957}. This indicates that the error term in \eqref{E:Bn-asympt}, which includes an extra $\log n$ factor, is suboptimal due to the crudeness of our analysis. For an alternative application of the saddle-point method based on Dobi\'nski's formula, see \cite[\S 6.3]{deBruijn1981}.

\begin{proof}
Let $k_\pm := \mu_n \pm \sigma_n^{\frac54}$, where $(\mu_n, \sigma_n^2)$ are given in \eqref{E:mu-var}. Then we split the sum over $k$ into three parts: 
\begin{align*}
	B_n &= \sum_{1\le k\le n}\Stirling{n}{k}
	= \llpa{\sum_{1\le k\le k_-}
	+\sum_{k_-<k<k_+}
	+\sum_{k_+\le k\le n}}\Stirling{n}{k}.
\end{align*}
Observe that $k\mapsto \frac{k^n}{k!}$ is unimodal for fixed $n$ with a unique peak at $k=\frac{n}{W(n)}(1+o(1))$. Thus, by \eqref{E:Snk-ua} and Lemma~\ref{L:k-central} with $x=\sigma_n^{\frac14}$,
\begin{align}
	\llpa{\sum_{1\le k\le k_-}
	+ \sum_{k_+\le k\le n}}\Stirling{n}{k}
	&\le \llpa{\sum_{1\le k\le k_-}
	+ \sum_{k_+\le k\le n}}\frac{k^n}{k!}\nonumber\\
	&\le n\max\biggl\{\frac{k_-^n}{k_-!},
	\frac{k_+^n}{k_+!}\biggr\}\nonumber \\
	&= O\lpa{\sqrt{n\omega_n}\, 
	e^{(\omega_n-1+\omega_n^{-1})n
	-\frac12\sqrt{\sigma_n}}},
	\label{E:k-out}
\end{align}
which is asymptotically negligible compared to the $O$-term on the right-hand side of \eqref{E:Bn-asympt}.

For the middle range $k_-<k<k_+$, the crucial observation is that
\emph{the interval $[k_-,k_+]$ is contained within $[1,k_1]$ (see
\eqref{E:k1}) for sufficiently large $n$} because
\begin{align*}
	k_+ 
	&=\frac{n}{\omega_n}\Lpa{1+\frac{\omega_n^{\frac38}}
	{n^{\frac38}(\omega_n+1)^{\frac58}}}
    =\frac{n}{\log n -\log \log n
	+O\left(\frac{\log \log n}{\log n}\right)}\le k_1,
\end{align*}
for large $n$. We can thus apply \eqref{E:Snk-central}. This use, together with $k=\mu_n+x\sigma_n$ and Lemma~\ref{L:k-central}, gives
\begin{align}\label{E:Snk-middle}
	\Stirling{n}{k}
	&= \frac{\sqrt{\omega_n}\,e^{(\omega_n-1+\omega_n^{-1})n-1}}
	{\sqrt{n}}\cdot\frac{e^{-\frac12 x^2}}{\sqrt{2\pi}}
	\llpa{1+\frac{p_1(\omega_n,x)}{\sigma_n}
	+O\Lpa{\frac{1+x^6}{\sigma_n^2}}},
\end{align}
uniformly for $x=o(\sigma_n^{\frac13})$, where 
\[
  p_1(\omega_n,x) 
	:= \frac{x((2\omega_n+1)x^2-3(\omega_n+1)(2\omega_n+3))}
	{6(\omega_n+1)^2}.
\] 
Here the exact form of $p_1$ is immaterial; what matters is that it is an odd polynomial in $x$. From \eqref{E:Enk-est} and the expression of $k$, we have $E_{n,k} = O(\sigma_n^{-2})$. Summing over $k$ in the range $k_-<k<k_+$, approximating the sum by an integral and extending the integral limits to infinity (introducing only asymptotically negligible errors), we obtain an extra factor of $\sigma_n$, leading to \eqref{E:Bn-asympt}. For similar arguments, see \cite[\S~5.1]{Odlyzko1995}.
\end{proof}

\subsection{Asymptotic approximations to the mean and the variance}

Applying the same analysis, we can derive asymptotic approximations for the mean and variance of the number of blocks in a random set partition (where all $B_n$ partitions of $n$ elements are equally likely).

\begin{thm}
\label{T:mean-var} For large $n$
\begin{equation}\label{E:mu-var2}
	\begin{split}
		\mathbb{E}(X_n) 
		&= \mu_n + O(\omega_n),\\
		\mathbb{V}(X_n) 
		&= \sigma_n^2 + O(\omega_n^2),
	\end{split}
\end{equation}
where $(\mu_n, \sigma_n^2)$ are given in \eqref{E:mu-var}.
\end{thm}
Finer approximations by other approaches are provided in Appendix~\ref{S:App1}.
\begin{proof}
The proof follows closely \emph{mutatis mutandis} that of \eqref{E:Bn-asympt}: defining \(x_{n,k}=(k-\mu_n)/\sigma_n\) then \(k=\mu_n+x_{n,k}\sigma_n\), therefore
\begin{align*}
	\mathbb{E}(X_n) 
	&= \frac1{B_n}\sum_{1\le k\le n}k\Stirling{n}{k}\\
	&= \frac1{B_n}\sum_{k_-< k< k_+}(\mu_n+x_{n,k}\sigma_n)
	\Stirling{n}{k} + O\lpa{n^{\frac32} 
	e^{-\frac12\sqrt{\sigma_n}}},
\end{align*}
which then yields the approximation for the mean in \eqref{E:mu-var2}. Similarly, 
\begin{align*}
	\mathbb{V}(X_n) 
	&= \mathbb{E}(X_n-\mu_n)^2 - (\mathbb{E}(X_n-\mu_n))^2\\
	&= \frac1{B_n}\sum_{k_-< k< k_+}
	(k-\mu_n)^2\Stirling{n}{k} + O(\omega_n^2)\\
	&= \frac1{B_n}\sum_{k_-< k< k_+}(x_{n,k}\sigma_n)^2
	\Stirling{n}{k} + O\lpa{n^{\frac52} 
	e^{-\frac12\sqrt{\sigma_n}}},
\end{align*}
from which we deduce the approximation for the variance in \eqref{E:mu-var2}. 
\end{proof}

It is also straightforward to extend the same calculations to all central moments $\mathbb{E}(X_n-\mu_n)^m$:
\[
	\mathbb{E}(X_n-\mu_n)^m
	= \begin{cases}
		O(\sigma_n^{m-1}),&\text{if $m$ is odd};\\
		\frac{m!}{2^{m/2}(m/2)!}\, \sigma_n^m
		(1+o(1)), &\text{if $m$ is even},
	\end{cases}
\]
implying the convergence of all moments of $\frac{X_n-\mu_n} {\sigma_n}$ to those of the standard normal, which in turn also leads to a proof of the CLT \eqref{E:clt} by the method of moments. The key difference is that this approach does not lead to a convergence rate for the asymptotic normality \eqref{E:clt}.

\subsection{Proof of the LLT}

We now prove Theorem~\ref{T:llt}, first for the LLT~\eqref{E:llt} for $\Stirling{n}{k}$. By the uniform bound \eqref{E:Snk-ua}, \eqref{E:k-out} and Lemma~\ref{L:k-central}, we have
\[
	\sup_{|k-\mu_n|\ge \sigma_n^{\frac43}}
    \biggl|\mathbb{P}(X_n = k)
	-\frac{e^{-\frac{(k-\mu_n)^2}
	{2\sigma_n^2}}}{\sqrt{2\pi}\,\sigma_n}\biggr|
	= O\lpa{\sigma_n^{-1}e^{-\frac12\sqrt{\sigma_n}}}
	= o(\sigma_n^{-2}).
\]
On the other hand, by \eqref{E:Bn-asympt}, \eqref{E:Snk-middle} and Lemma~\ref{L:k-central}, we see that when $k=\tr{\mu_n+x\sigma_n}$
\begin{align}
	\mathbb{P}(X_n = k)
	= \frac1{B_n}\Stirling{n}{k}
	= \frac{e^{-\frac12x^2}}{\sqrt{2\pi}\,\sigma_n}
	\llpa{1+O\llpa{\frac{|x|+|x|^3}{\sigma_n}}}
	\Lpa{1+O\Lpa{\frac{(\log n)^2}{n}}},
\end{align}
uniformly for $x=o(\sigma_n^{\frac13})$, and, particularly, for $k_-<k<k_+$. This proves Theorem~\ref{T:llt}. The proof for the CLT \eqref{E:clt} is similar.

\section{Asymptotic nature of the sieve formula, II}

The preceding analysis is simple but limited by the \emph{range of uniformity} in $k$ and the \emph{degree of precision}. Before extending along these two directions, we examine more closely the alternating nature of the sum \eqref{E:Snk-sum} in this section. 

Since the range $k\le k_0$, equivalently $\lambda=o(1)$ (see \eqref{E:k0}) in the earlier setting, has already been addressed, we next focus on the transition and larger range; in particular $\lambda>1$, which corresponds to $k\ge n/W(n)$. For this range of $k$, the \emph{exponential cancellation effect} resulting from the alternating factor $(-1)^j$ becomes more pronounced: \emph{the largest binomial terms grow exponentially, while their alternating sum is of much smaller order}.

\subsection{Numerical instability}
\label{sec:numerical}

For convenience, we consider the normalized sum
\begin{align}\label{E:Snk-sum0}
	S(n,k) 
	:= \sum_{0\le j\le k}(-1)^jb_{n,k}(j),
	\eqtext{where}
	b_{n,k}(j) := \binom{k}{j}\Lpa{1-\frac jk}^n.
\end{align}
We first look at the numerics of $S(n,k)$ with $n=20$ and $k=11$; see Figure~\ref{F:bj}.
\begin{figure}[!ht]
\centering
\begin{tikzpicture} 
	\node[anchor=south west, inner sep=0] (image) at (0,0)
	{\includegraphics[width=0.25\textwidth]{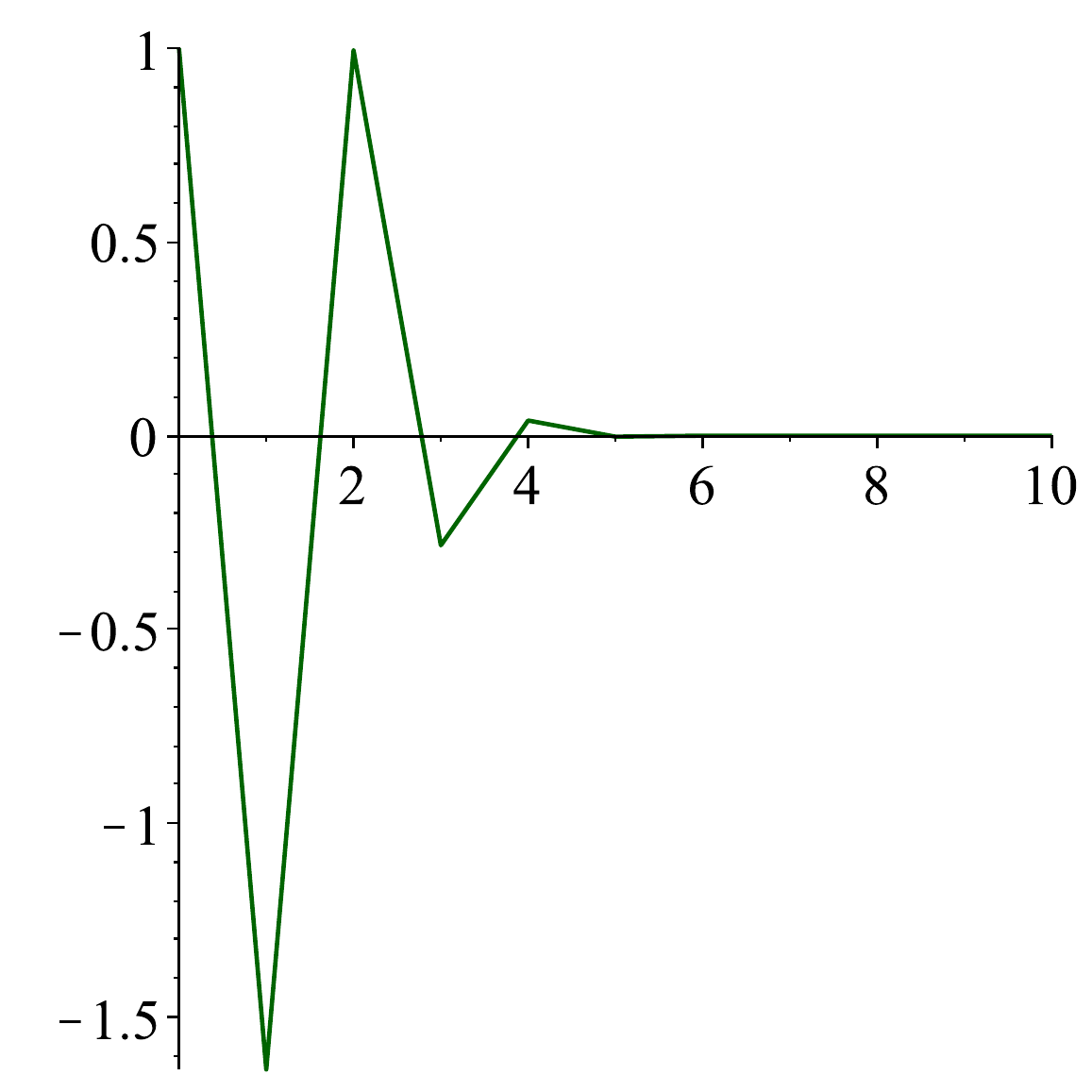}\qquad
	\includegraphics[width=0.25\textwidth]{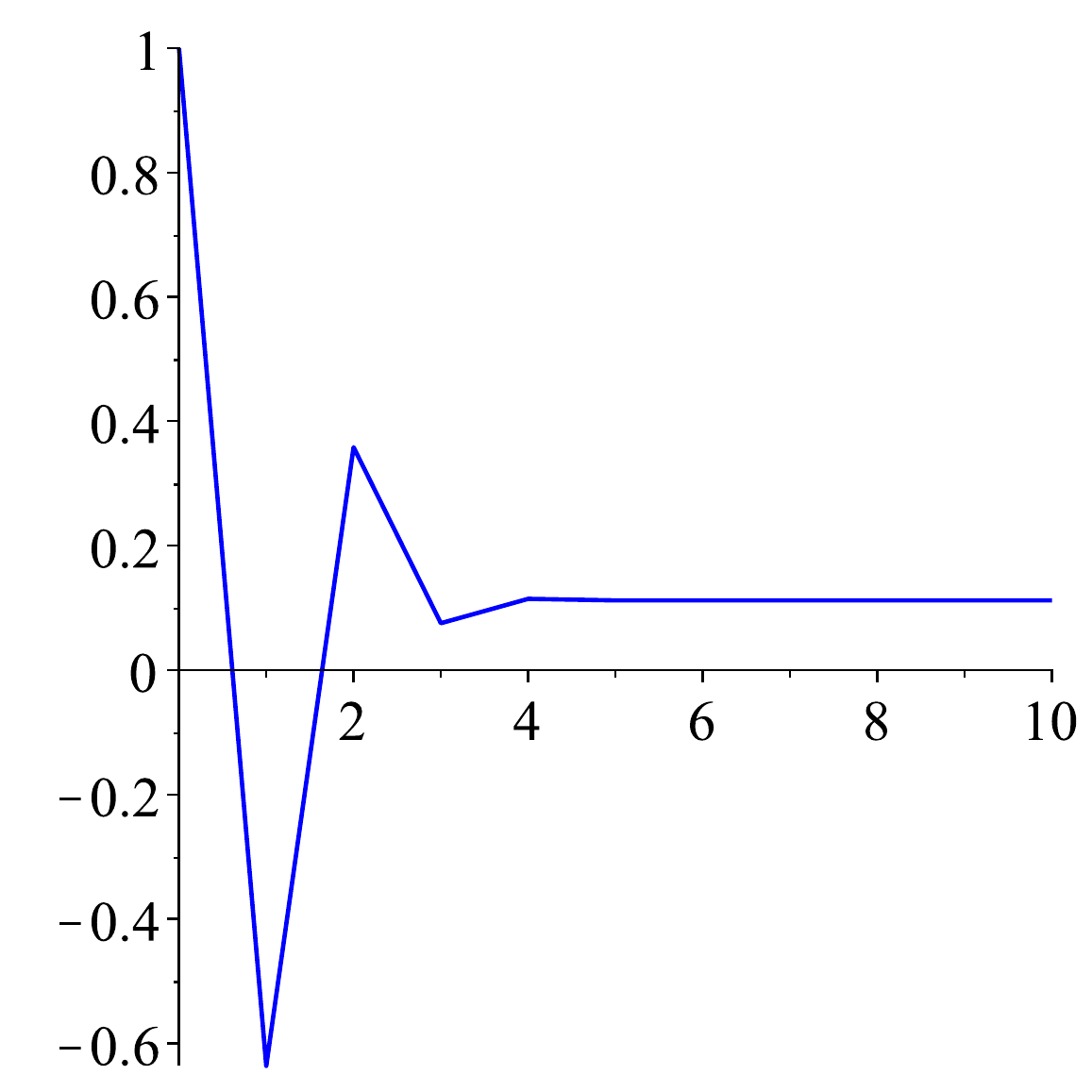}};		
	\begin{scope}[x={(image.south east)}, y={(image.north west)}]
	
	\node at (0.3,0.75) {\small$(-1)^jb_{n,k}(j)$};
	\node at (0.81,0.75) {\small$\sum\limits_{0\le i\le j}
	(-1)^ib_{n,k}(i)$};
	\node at (0.83,0.2) {\small$\frac{k!}{k^n}
	\Stirling{n}{k}\approx 0.1127$};
	\end{scope}
\end{tikzpicture}
\begin{footnotesize}
\begin{tabular}{cccccccc}
$j$ & $0$ & $1$ & $2$ & $3$ & $4$ & $\cdots$
& $\sum_{0\le j\le k}(-1)^j b_{n,k}(j)$ \\ \hline
$b_{n,k}(j)$ & $1$ & $1.6351$ & $0.99394$ 
& $0.28277$ & $0.039140$ & $\cdots$ & $0.1127$
\end{tabular}
\end{footnotesize}
\caption{Fluctuations of $(-1)^jb_{n,k}(j) = \binom{k}{j}(-1)^j(1 - \frac{j}{k})^n$ for $n=20$, $k=11$, and the corresponding partial sums.}\label{F:bj}
\end{figure}

While $S(n,k)\approx0.1127$, the first 4 terms in the sum \eqref{E:Snk-sum0} are all larger than the resulting sum. This phenomenon becomes even more noticeable for larger values of $n$ and $k$. For instance, when $n=200$ and $k=70$, we have $S(n,k)\approx0.01149$ while $b_{n,k}(j)>0.01149$ for all $0\le j\le 10$. The largest term, $b_{n,k}(3)\approx 8.5833$, is over 746 times greater than the value of $S(n,k)$.

\subsection{Bonferroni inequality and unimodality}

The above type of \emph{numerical instability} can be further examined through the use of the Bonferroni inequality, a consequence of the inclusion-exclusion principle, which states that
\begin{equation}\label{E:bonf-ineq}
	\left|S(n,k)
	-\sum_{0\le l<j}(-1)^lb_{n,k}(l)\right|
	\le b_{n,k}(j),
\end{equation}
for $j=0,1,\dots$. For example, for $(n,k)=(20,11)$, at least four terms ($j=4$) are required for the error to fall below the 
value of the resulting sum. In the case $(n,k)=(200,70)$, using $j=11$ in \eqref{E:bonf-ineq} results in an absolute error less than $b_{n,k}(11)\approx0.003062$.

These numerical observations can be further analyzed more precisely from an analytic viewpoint. Let $\lambda_0 := \frac{b_{n,k}(1)}{b_{n,k}(0)} =k\lpa{1-\frac1k}^n$.

\begin{prop} If $\lambda_0<1$, then $b_{n,k}(0)>b_{n,k}(1) >\cdots>b_{n,k}(k)$ and \eqref{E:Snk-sum} is an asymptotic expansion. If $\lambda_0>1$, then $\{b_{n,k}(j)\}$ is unimodal; let $j_\star$ be the (unique) index where b attains its maximum (equivalently where $\frac{b(j+1)}{b(j)}$ crosses 1), then $j_\star\sim \lambda$ if \(\lambda \to \infty\) and \(k=o(n)\). Moreover, for $k\le \frac{(2-\ve)n}{\log n}$ and \(\lambda\to \infty\),
\begin{align}\label{E:max-bnk-j}
	\max_{0\le j\le k}b_{n,k}(j)
	= \frac{e^\lambda}{\sqrt{2 \pi \lambda}}
	\left(1+o\left(1\right)\right).
\end{align}	
\end{prop}
While the largest term is of order $\lambda^{-\frac12}e^\lambda$, the resulting alternating sum (see \eqref{E:Snk-central}) is of a much smaller order, namely, $e^{-\lambda}$, at least within the range $k\le k_1$. This highlights that any elementary approach relying on the alternating sum \eqref{E:Snk-sum} must carefully account for the intricate exponential cancellations involved. Such complications seem overlooked in \cite{Menon1973}. 

\begin{proof}
We begin with the ratios of consecutive terms in \eqref{E:Snk-sum}:
\[
	\frac{b_{n,k}(j+1)}{b_{n,k}(j)}
	=\frac{k-j}{j+1}\Lpa{1-\frac{1}{k-j}}^n,
\]
which behave like $\frac{\lambda}{j+1}=\frac{k}{j+1}e^{-\frac nk}$ for bounded values of $j$. Since \(\frac{k-j}{j+1}\) and \((1-\frac1{k-j})^n\) decrease in j, so does their product 
\begin{align}\label{E:bnkj-down}
	\frac{b_{n,k}(1)}{b_{n,k}(0)}> \frac{b_{n,k}(2)}{b_{n,k}(1)}
	>\cdots> \frac{b_{n,k}(k)}{b_{n,k}(k-1)}.
\end{align}
Thus, if $\lambda_0=\frac{b_{n,k}(1)}{b_{n,k}(0)}\le 1$, then by monotonicity \(b_{n,k}(j)\le \lambda^j\), yielding, by \eqref{E:bonf-ineq}, the asymptotic expansion 
\[
	\Bigl|S(n,k)-\sum_{0\le l<j}(-1)^l b_{n,k}(l)
	\Bigr|\le b_{n,k}(j)\le \lambda^j.
\]
This holds true whenever \(\lambda_0\le 1\) and uniformly for \(0\le j\le k\).

Assume now \(\lambda_0>1\). Then \eqref{E:bnkj-down} implies that \(b_{n,k}(j)\) is unimodal: they first increase, reaching the maximum at say, \(j=j_\star\), and then steadily decrease. By the simple inequalities \(\binom{k}{j} \le \frac{k^j}{j!}\) and \(1-\frac jk\le e^{-\frac jk}\), we see that
\[
	b_{n,k}(j)
	=\binom{k}{j}\left( 1-\frac{j}{k}\right)^{n} 
	<\frac{k^j}{j!}e^{-\frac{nj}k}
	=\frac{\lambda^j}{j!}
\]
for all \(0\le j\le k\). Thus
\[
	\max_{0\le j \le k} b_{n,k}(j)
	\le \max_{0\le j \le k} \frac{\lambda^j}{j!}
	=\frac{\lambda^{j_0}}{j_0!}
\]
where \(j_0:=\tr{\lambda}\), the largest integer less than or equal to $\lambda$. For a lower bound, note that
\[
\begin{split}
	b_{n,k}(j)
	\ge \frac{k^j}{j!}\left( 1-\frac{j}{k}\right)^{n+j}
	\ge \frac{k^j}{j!}e^{-\frac{j(n+j)}{k-j}},
\end{split}
\]
where we applied the elementary inequality \((1-x)\ge e^{-\frac x{1-x}}\). Consequently, with \(j=j_0\), we obtain the lower bound
\[
\begin{split}
	\max_{0\le j \le k} b_{n,k}(j)
	\geq b_{n,k}(j_0)
	\ge \frac{\lambda^{j_0}}{j_0!}e^{-\frac{j_0^2(n+k)}{k(k-j_0)}}
	\ge\frac{\lambda^{j_0}}{j_0!}
	e^{-\frac{2\lambda^2n}{k(k-\lambda-1)}}
	\ge\frac{\lambda^{j_0}}{j_0!}
	e^{-\frac{2ne^{-2n/k}}{(1-e^{-n/k}-1/k)}}.
\end{split}
\]
Here, while evaluating the term under the exponent, we used the inequalities \(\lambda-1\le j_0\le \lambda\). Combining the lower and upper bound inequalities
\begin{equation}\label{eq:lu_ineq_max}   
    \frac{\lambda^{j_0}}{j_0!}
	\exp\left(-\frac{2ne^{-\frac{2n}k}}{1-1/e-1/2}\right)
	\le\max_{0\le j \le k} b_{n,k}(j)
	\le\frac{\lambda^{j_0}}{j_0!},
\end{equation}
we see that for $2\le k\le \frac{(2-\ve)n}{\log n}$ with \(\ve>0\) the term under the exponent in the above estimate is \(o(1)\) and, as a consequence, by applying Stirling's formula to evaluate \(j_0!=\tr{\lambda}!\),  we complete the proof of the estimate \eqref{E:max-bnk-j}.
\end{proof}

In particular, if $k=\lfloor\frac{cn}{\log n}\rfloor$ with \(1<c<2\), then \(b_{n,k}(j)\) attains its maximum value 
\[
	b_{n,k}(j_\star)=O\lpa{(\log n)^{\frac12}
	n^{-\frac12+\frac1{2c}}\,e^{\frac{c}{\log n}\,
	n^{1-\frac1c}}},
\]
at a point $j_\star$ close to \(\frac{cn^{1-\frac{1}{c}}}{\log n}\).

\section{A finite difference approach}

While the sum \eqref{E:Snk-sum} serves as an asymptotic expansion when $\lambda=o(1)$, it breaks down when $\lambda\asymp 1$: the elementary approach proposed above does enlarge the uniformity range in $k$ from $k_0$ (see \eqref{E:k0}) to $k_1$ (see \eqref{E:k1}), yet it is not clear how to extend that range further. In this section, we derive a new asymptotic expansion for $S(n,k)$, which is useful for $n^{1-\ve}\le k\le (2-\ve)\frac{n}{\log n}$. Our approach relies on Leibniz’s rule for finite differences of a product of two functions (or sequences). 

\subsection{A finite difference expansion}

Define the backward difference operator \(\nabla\) by 
\[
	\nabla_{\!x} f(x) = f(x)-f(x-1).
\]

\begin{lemma}
For \(k,n\ge1\),
\begin{align}\label{E:stir-leib}
	S(n,k)
	=(-1)^k\nabla_{\!x}^{k}
	\Lpa{1-\frac{x}{k}}^n\Bigl|_{x=k}.
\end{align}
\end{lemma}

\begin{proof}
Iterating the difference operator \(k\) times yields
\begin{equation}\label{E:finite-diff-k}
    \nabla_{\!x}^{k}f(x) 
	=\sum_{0\le j\le k}\binom{k}{j}(-1)^jf(x-j).
\end{equation}
Then
\[
	\nabla_{\!x}^{k}\Lpa{1-\frac{x}{k}}^n
	=\sum_{0\le j\le k}\binom{k}{j}(-1)^j
	\Lpa{1-\frac{x-j}{k}}^n,
\]
and \eqref{E:stir-leib} follows from comparing this expression with the right-hand side of \eqref{E:Snk-sum} for Stirling numbers.
\end{proof}

For large \(k\), we can approximate $\left(1-\frac{x}{k}\right)^n$ by 
$e^{-\frac{nx}k}$; it is then natural to write
\begin{align}\label{E:fg}
	\Lpa{1-\frac{x}{k}}^n
	=f(x)g(x),\eqtext{with}
	f(x):=e^{-x\frac{n}{k}} \eqtext{and}
	g(x):=e^{\frac{nx}k}
	\Lpa{1-\frac{x}{k}}^n.
\end{align}

\begin{lemma}[Leibniz's formula for finite differences]
\begin{equation}\label{E:Leibniz}
	\nabla_{\!x}^{k}\bigl(f(x) g(x) \bigr)
	=\sum_{0\le j\le k}\binom{k}{j} 
	\lpa{\nabla_{\!x}^jf(x)}
	\lpa{\nabla_{\!x}^{k-j}g(x-j)}.
\end{equation}
\end{lemma}

\begin{proof}
The difference operator with respect to the product of two functions can be expressed as
\[
\begin{aligned}
	\nabla_{\!x}\bigl(f(x)g(x)\bigr) 
	&=f(x) g(x) -f(x-1) g(x-1)\\
	&=f(x) g(x) -f(x) g(x-1) 
	    +f(x) g(x-1) -f(x-1) g(x-1) \\
	&=f(x) \nabla_{\!x} g(x) 
	    + g(x-1) \nabla_{\!x} f(x).
\end{aligned}
\]
A direct iteration then gives \eqref{E:Leibniz}.
\end{proof}

For convenience, define the ratio
\begin{align}\label{E:Lambda}
	\Lambda := \frac{e^{-\frac{n}{k}}}{1-e^{-\frac{n}{k}}}
	= \frac{\lambda}{k-\lambda},
\end{align}
and
\begin{align}\label{E:Dnkj}
	D_{n,k}(j)
	:= \nabla_{\!x}^jg(x)\bigl|_{x=j} 
	=\sum_{0\le l\le j}\binom{j}{l}(-1)^{j-l}
	\left(e^{\frac {l}k}\left(1-\frac{l}{k}\right)\right)^n .
\end{align}
Leibniz's formula \eqref{E:Leibniz} then gives the following exact identity.

\begin{prop}[Identity]\label{P:st2-lbnz-prop}
The normalized Stirling partition numbers satisfy the identity
\begin{align}\label{E:st2-lbnz}
	S(n,k)
	= \left(1-\frac{\lambda}k\right)^{k} 
	\sum_{0\le j\le k}\binom{k}{j}(-\Lambda)^j
	D_{n,k}(j).
\end{align}
\end{prop}
\begin{proof}
Substituting \eqref{E:fg} into \eqref{E:Leibniz} and then replacing \(j\) by \(k-j\), we obtain
\[
	S(n,k)
	=(-1)^k\sum_{0\le j\le k}\binom{k}{j}
	\left(\nabla_{\!x}^{k-j} 
	e^{-\frac{nx}{k}}\Bigl|_{x=k} \right)
	D_{n,k}(j).
\]
The finite differences of \(e^{-nx/k}\) are explicit:
\[
	\nabla_{\!x}^{k-j}e^{-\frac{nx}{k}} 
	=\sum_{0\le l\le k-j}\binom{k-j}{l}
	(-1)^{l}e^{-(x-l)\frac{n}{k}}
	=e^{-\frac{nx}{k}}\bigl(1-e^{\frac{n}{k}}\bigr)^{k-j}.
\]
Evaluating at \(x=k\), we get
\[
\begin{split}
	\nabla_{\!x}^{k-j}e^{-\frac{nx}{k}} \Bigl|_{x=k}
	&=e^{-n}\bigl(1-e^{\frac{n}{k}}\bigr)^{k-j}\\
	&=(-1)^{k-j}e^{-j\frac{n}{k}}
	\bigl(1-e^{-\frac{n}{k}}\bigr)^{k-j}\\
	&=(-1)^{k-j}
	\left(1-\frac{\lambda}{k}\right)^k\Lambda^j.
\end{split}
\]
Substituting this into the previous display proves \eqref{E:st2-lbnz}.
\end{proof}

When $\frac{k}{\sqrt{n}}\to \infty$, the function $g(x)$ behaves asymptotically like a constant
\[
	g(x)\approx \exp\Lpa{-\frac{x^2n}{2k^2}}\to 1,
\]
so its finite differences are negligibly small. Consequently, the main contribution in Leibniz's formula \eqref{E:Leibniz} is expected to arise from the term with $j=0$ .

Although the terms on the right-hand side of \eqref{E:st2-lbnz} are individually more complex than those on the left-hand side, the utility of this identity lies in its asymptotic character: truncating the expansion at any fixed number of terms yields a rigorously effective approximation, entirely circumventing the exponential cancellation inherent in the defining sum \eqref{E:Snk-sum}. More precisely, we establish in this section that \eqref{E:st2-lbnz} furnishes an asymptotic expansion in $k$ throughout the range stated in the following theorem.

\begin{thm}\label{T:finite-diff-exp}
If
\begin{align}\label{E:kk2}
	\frac{n}{\log n}\le k\le \frac{2n}{\log n + 6},
	\eqtext{or}
	\frac1{\log n}\le \lambda \le \frac{2}{e^3}\cdot
	\frac{\sqrt{n}}{\log n+6},
\end{align}
then
\begin{equation}\label{E:FD}
	S(n,k)
	= \Lpa{1-\frac{\lambda}{k}}^{k} 
	\llpa{1+\sum_{1\le j<s}\binom{k}{j}
	(-\Lambda)^jD_{n,k}(j)
	+O\Lpa{\Lpa{\frac{\lambda\log n}{\sqrt{n}}}^{s}}},
\end{equation}
for any fixed \(s\ge 1\), where $\Lambda$ and $D_{n,k}(j)$ are defined in \eqref{E:Lambda} and \eqref{E:Dnkj}, respectively.
\end{thm}

The constant $6$ in the upper limit of $k$ in \eqref{E:kk2} can be made smaller if needed.

\subsection{Lemmas}

We first derive two upper bounds for \(|D_{n,k}(j)|\) that will be useful for large and small values of \(j\), respectively.

\begin{lemma}\label{L:Dnkj-bd1}
The inequality 
\[
	\left|D_{n,k}(j)\right|\le 2^j
\]
holds for each $k\ge1$, $n\ge0$, and $0\le j\le k$.
\end{lemma}
\begin{proof}
First, by the inequality $1+y\le e^{y}$ for real $y$, we have
\[
    0\le e^{\frac xk}\left( 1-\frac{x}{k}\right)\le 1,
    \qquad(0\le x\le k).
\]
Consequently, by \eqref{E:Dnkj},
\[
	|D_{n,k}(j)|
	\le \sum_{0\le l\le j}\binom{j}{l}= 2^j.
	\qedhere
\]
\end{proof}

\begin{lemma}\label{finite_diff_deriv}
Assume that a function \(u(y)\) is defined and \(j\) times continuously differentiable on \(\mathbb{R}\). Then
\[
	\bigl|\nabla_{\!x}^ju(x)\bigr|
	\le\max_{y\in[x-j,x]}\bigl|u^{(j)}(y)\bigr|.
\]
\end{lemma}
\begin{proof}
This follows from the integral representation
\[
	\nabla_y^m u(y)
	= \int_{[0,1]^m}u^{(m)}(y-t_1-\cdots-t_m)
	\,\dd{t_m}\cdots \dd{t_1}.
	\qedhere
\]
\end{proof}

\begin{lemma}\label{E:Dnkj-bd2}
Uniformly for \(0\le j < \frac k{\sqrt{n}}\)
\[
	\left|D_{n,k}(j)\right|
	\le \frac{j!}{\lpa{1-j\frac{\sqrt{n}}{k}}^{j+1}}
	\Lpa{\frac{\sqrt{n}}{k}}^j.
\]
\end{lemma}
\begin{proof}
Since
\[
    g(x)=\exp\!\left(-\sum_{m\ge 2}\frac{n}{m}
    \left(\frac{x}{k}\right)^m\right)
\]
and \(\frac{n}{mk^m}\le\frac{1}{m}\bigl(\frac{\sqrt{n}}{k}\bigr)^m\) for \(m\ge 2\), comparing coefficients termwise gives, for any \(s\ge 0\),
\begin{equation}\label{E:coeff-g}
    \bigl|[x^s]g(x)\bigr|
    \le [x^s]\exp\!\left(\sum_{m\ge 2}\frac{1}{m}
    \left(\frac{\sqrt{n}}{k}\,x\right)^m\right)
    \le [x^s]\frac{1}{1-\frac{\sqrt{n}}{k}\,x}
    =\left(\frac{\sqrt{n}}{k}\right)^s.
\end{equation}
Summing over \(s\ge j\) with the appropriate factorial weights, this lifts to derivatives: for \(0\le x<k/\sqrt{n}\),
\[
    \bigl|g^{(j)}(x)\bigr|
    \le \left(\frac{1}{1-\frac{\sqrt{n}}{k}\,x}\right)^{(j)}
    =\frac{j!}{\left(1-\frac{\sqrt{n}}{k}\,x\right)^{j+1}}
    \left(\frac{\sqrt{n}}{k}\right)^j.
\]
It follows, by Lemma~\ref{finite_diff_deriv}, that 
\[
    |D_{n,k}(j)|
    =\left|\nabla_{\!x}^j g(x)\Big|_{x=j}\right|
    \le \max_{x\in[0,j]}|g^{(j)}(x)|
    \le \frac{j!}{\left(1-\frac{j\sqrt{n}}{k}\right)^{j+1}}
    \left(\frac{\sqrt{n}}{k}\right)^j.\qedhere
\]
\end{proof}

\subsection{Asymptotic nature of the expansion \eqref{E:st2-lbnz}}

\begin{proof}[Proof of Theorem~\ref{T:finite-diff-exp}]
Since \(D_{n,k}(0)=1\), it suffices to estimate the remainder
\[
	R_s:=
	\frac{S(n,k)}{\left(1-\frac{\lambda}{k}\right)^k}
	-\sum_{0\le j<s}\binom{k}{j}(-\Lambda)^jD_{n,k}(j).
\]
By \eqref{E:st2-lbnz},
\begin{equation}\label{E:st2-lbnz2}
	R_s
	=\sum_{s\le j\le k}\binom{k}{j}(-\Lambda)^jD_{n,k}(j).
\end{equation}
Define 
\[
	c_1:=\frac{2e}{1+2e},
	\qquad
	c_2:=1+2e,
	\qquad
	j_1:=c_1\frac{k}{\sqrt n}.
\]
Split the sum in \eqref{E:st2-lbnz2} as
\[
	R_s=
	\left(\sum_{s\le j<j_1}+\sum_{j\ge \max\{s,j_1\}}\right)
	\binom{k}{j}(-\Lambda)^jD_{n,k}(j)
	=:S_1+S_2.
\]

For \(S_2\), Lemma~\ref{L:Dnkj-bd1} and the bound \(j!\ge j^je^{-j}\) give
\[
\begin{split}
	|S_2|
	&\le \sum_{j\ge \max\{s,j_1\}}\binom{k}{j}(2\Lambda)^j
	\le \sum_{j\ge \max\{s,j_1\}}\frac{k^j}{j!}(2\Lambda)^j\\
	&\le \sum_{j\ge \max\{s,j_1\}}
	\left(\frac{2ek\Lambda}{j}\right)^j
	\le \sum_{j\ge \max\{s,j_1\}}
	\left(\frac{2e\sqrt n\,\Lambda}{c_1}\right)^j.
\end{split}
\]

For \(S_1\), Lemma~\ref{E:Dnkj-bd2} yields
\[
\begin{split}
	|S_1|
	&\le\sum_{s\le j<j_1}\binom{k}{j}\Lambda^j
	\frac{j!}{\left(1-j\frac{\sqrt n}{k}\right)^{j+1}}
	\left(\frac{\sqrt n}{k}\right)^j\\
	&\le\sum_{s\le j<j_1}\frac{(k\Lambda)^j}{j!}
	\frac{j!}{\left(1-j\frac{\sqrt n}{k}\right)^{j+1}}
	\left(\frac{\sqrt n}{k}\right)^j\\
	&\le \sum_{s\le j<j_1}
	\frac{(\sqrt n\,\Lambda)^j}{(1-c_1)^{j+1}}.
\end{split}
\]
Since $\frac1{1-c_1}=\frac{2e}{c_1}=c_2$, the preceding bounds combine to give
\[
	|R_s|
	\le c_2\sum_{j\ge s}(c_2\sqrt n\,\Lambda)^j.
\]
Now assume that \(k\) satisfies \eqref{E:kk2}. Then $\frac{n}{k}\ge \frac{\log n+6}{2}$, and $e^{-\frac nk}\le e^{-3}n^{-1/2}$. Therefore
\[
	c_2\sqrt n\,\Lambda
	=\frac{c_2\sqrt n\,e^{-\frac nk}}{1-e^{-\frac nk}}
	\le \frac{c_2e^{-3}}{1-e^{-3}}
	<\frac12.
\]
Hence the geometric series is uniformly convergent, and
\[
	|R_s|
	\le 2c_2(c_2\sqrt n\,\Lambda)^s
	=O\lpa{(\sqrt n\,\Lambda)^s}
	=O\Lpa{\Lpa{\frac{\lambda\sqrt n} {k-\lambda}}^s}.
\]
Since \(k\ge \frac{n}{\log n}\) and \(k-\lambda\ge (1-e^{-3})k\), it follows that
\[
	\frac{\lambda\sqrt n} {k-\lambda}
	\le \frac{\lambda\sqrt n}{(1-e^{-3})k}
	\le \frac{\lambda\log n}{(1-e^{-3})\sqrt n}.
\]
Thus
\[
	R_s
	=O\left(\left(\frac{\lambda\log n}{\sqrt n}\right)^s\right),
\]
which is exactly \eqref{E:FD}.
\end{proof}

\begin{remark}
The choice of $f(t) = e^{-nt}$ as a convenient approximation to $(1-t)^n$ in our application of Leibniz's formula is not unique. One may consider alternative functions $f$ that admit the local expansion $f(t) = 1-nt+\cdots$ for small $t$, while also possessing desirable features such as simple or easily computable finite differences. However, the systematic construction of such functions is far from straightforward. 
\end{remark}

\begin{remark}[Coupon collector problem]\label{R:ticket-fd}
The classical \emph{coupon collector's problem}---asks for the probability that, after $n$ independent uniform random draws from $k$ coupon types, every type has been observed at least once. A natural generalization allows $s$ distinct coupons to be drawn simultaneously at each stage, with stages remaining independent. The probability that all $k$ types have appeared after $n$ such stages is then given by inclusion–exclusion as (see \cite[Book~II, \S~4]{Laplace1820})
\[
	M_s(n,k):=\sum_{0\le j\le k}\binom{k}{j}(-1)^j\Pi_s(j),
	\qquad (n,k\ge1,\ 1\le s\le k),
\]
where (an empty product being equal to $1$)
\[
    \Pi_s(x) := \prod_{0\le m<s}\Lpa{1-\frac{x}{k-m}}^{n}.
\]
In particular, $s=1$ recovers the classical case and satisfies $M_1(n,k)=S(n,k)$. The differencing arguments developed above extend naturally to $M_s$; since this generalization plays no role in what follows, we omit the details.
\end{remark}

\subsection{Asymptotic nature of the expansion \eqref{E:st2-lbnz}, II}

We already established the asymptotic nature of the identity \eqref{E:st2-lbnz} in Theorem~\ref{T:finite-diff-exp}, yet the terms in the expansion remain less obvious as far as their asymptotic smallness is concerned. In this subsection, we look more closely (not crude bounds) at the terms in the expansion \eqref{E:FD}. 

\begin{itemize}

\item If $1\le k\le \frac{n}{\log n}$, then $\lambda=o(1)$, and we obtain, by Lemma~\ref{E:Dnkj-bd2},
\begin{align}\label{E:ksmall}
	\binom{k}{j} (-\Lambda)^j D_{n,k}(j)
	=O\left( \lambda^j \Lpa{\frac{\sqrt{n}}k}^j\right)
	=O\lpa{n^{-\frac j2}},
\end{align}
for $j=1,2,\dots$, showing that \eqref{E:FD} is also an asymptotic expansion for smaller values of $k$.

\item If $\frac{n}{\log n}\le k\le \frac{2n}{\log n}$, then 
\[
	e^{\frac nkl}\Lpa{1-\frac lk}^n
	= \exp\Lpa{-(1+o(1))\frac{nl^2}{2k^2}}
	= \exp\Lpa{-(1+o(1))\frac{\rho l^2}{2k}}.
\]
We then get, for each $j$,
\[
	D_{n,k}(j) = \sum_{0\le l\le j}\binom{j}{l}
	(-1)^{j-l}\exp\Lpa{-\frac{\rho l^2}{2k}(1+o(1))}.
\]
Now for any $\ve>0$ 
\begin{align*}
	\sum_{0\le l\le j}\binom{j}{l}
	(-1)^{j-l}e^{-\frac12\ve l^2}
	&= \frac1{\sqrt{2\pi}}
	\int_{-\infty}^\infty e^{-\frac12u^2}
	\lpa{e^{\sqrt{\ve}iu}-1}^j\dd u\\
	&= \frac1{\sqrt{2\pi}}
	\int_{-\infty}^\infty e^{-\frac12u^2}
	\lpa{\sqrt{\ve} iu}^j
	\lpa{1+\tfrac12\sqrt{\ve}\,iu+\cdots}^j\dd u.
\end{align*}
Thus 
\begin{equation*}
	\sum_{0\le l\le j}\binom{j}{l}
	(-1)^{j-l}e^{-\frac12\ve l^2}
	\sim \begin{cases}
		\frac{1}{\sqrt{\pi}}(-2\ve)^{\frac12j}
		\Gamma\lpa{\frac12(j+1)},
		&\text{if $j$ is even};\\
		\frac{j}{2\sqrt{\pi}}(-2\ve)^{\frac12(j+1)}
		\Gamma\lpa{\frac12(j+2)}, & \text{if $j$ is odd}.
	\end{cases}
\end{equation*}
It follows that
\begin{align*}
	\binom{k}{j}(-\Lambda)^j
	D_{n,k}(j) = O\Lpa{\Lpa{\frac{\rho}
	{k}}^{\cl{\frac12j}}\lambda^j}
	= O\Lpa{\frac{(\log n)^{2\cl{\frac12j}}}
	{n^{\cl{\frac12j}}}\lambda^j}
	\qquad(j=0,1,\dots).
\end{align*}
\end{itemize}
This means that to achieve an error of order $n^{-m}$ for approximating $S(n,k)$, we need to use at least $2m+1$ terms in the expansion \eqref{E:FD}. In particular, since
\[
    D_{n,k}(1)=e^{\frac nk}\lpa{1-\tfrac1k}^n-1,\qquad
    D_{n,k}(2)=1-2e^{\frac nk}\lpa{1-\tfrac1k}^n
	+e^{\frac{2n}k}\lpa{1-\tfrac2k}^n,
\]
we have
\begin{equation}\label{E:FD2}
	\begin{split}
	\frac{S(n,k)}{\lpa{1-\frac{\lambda}k}^k}
	&= 1+\frac{k\lambda}{k-\lambda}
	\Lpa{1-e^{\frac nk}\lpa{1-\tfrac1k}^n}\\
	&\quad +\frac{k(k-1)\lambda^2}
	{2(k-\lambda)^2}
	\Lpa{1-2e^{\frac nk}\lpa{1-\tfrac1k}^n
	+e^{\frac{2n}k}\lpa{1-\tfrac2k}^n}
	+O\Lpa{\frac{(\log n)^4}{n^2}\lambda^2(1+\lambda^2)}.
	\end{split}
\end{equation}
Although the terms in the expansion \eqref{E:FD} become more intricate, their numerical performance remains effective, particularly when $k$ lies near the central range (e.g., the mode of $\Stirling{n}{k}$ or the mean of the Stirling distribution). Unlike the original sum definition \eqref{E:Snk-sum}, whose terms are simpler but suffer from exponential cancellations, the expansion \eqref{E:FD} avoids such issues and improves numerical precision as more terms are included.

\begin{figure}[!ht]
\begin{center}
\includegraphics[width=12cm]{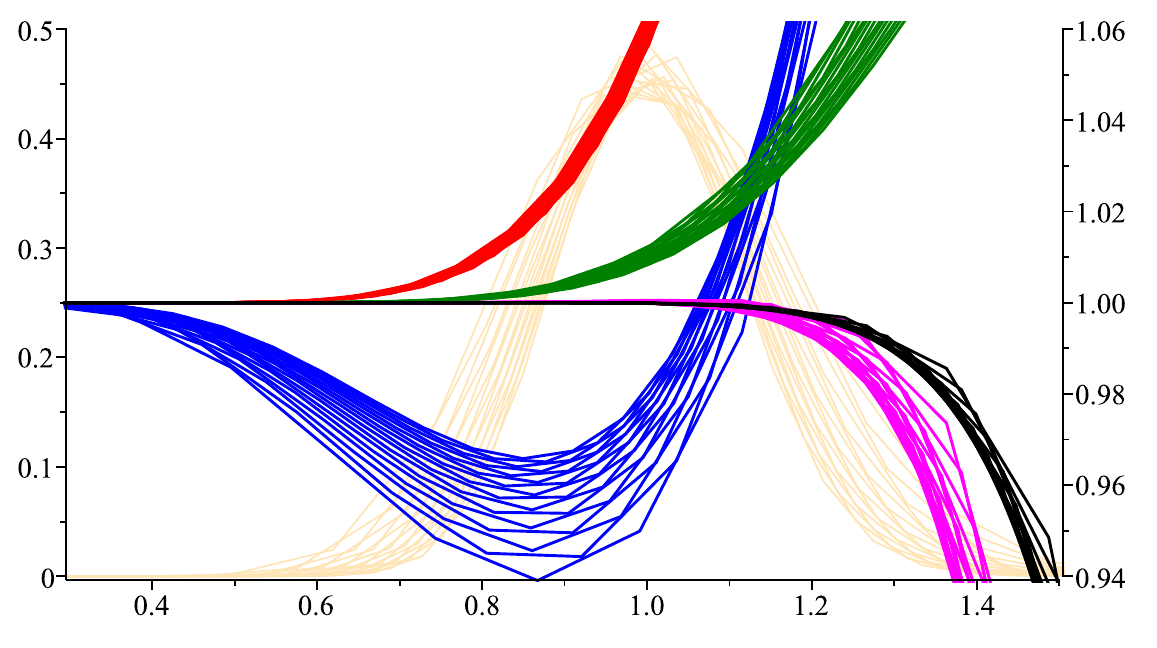}\\[0.5ex]
{\footnotesize
\textcolor{blue}{\raisebox{0.45ex}{\rule{1.6em}{1.0pt}}\, $s=1$} \quad
\textcolor{red}{\raisebox{0.45ex}{\rule{1.6em}{1.0pt}}\, $s=2$} \quad
\textcolor{green!50!black}{\raisebox{0.45ex}{\rule{1.6em}{1.0pt}}\, $s=3$} \quad
\textcolor{magenta}{\raisebox{0.45ex}{\rule{1.6em}{1.0pt}}\, $s=4$} \quad
\textcolor{black}{\raisebox{0.45ex}{\rule{1.6em}{1.0pt}}\, $s=5$} \quad
\textcolor{orange!60}{\raisebox{0.45ex}{\rule{1.6em}{1.0pt}}\, densities}}
\end{center}
 \caption{\emph{Illustrating the asymptotic expansion from equation \eqref{E:FD}, this figure presents the Stirling numbers of the second kind for \(n \in \{20,22,\cdots,50\}\). The colored curves, referencing the right $y$-axis and plotted on a normalized $x$-axis variable centered on the asymptotic mean \(\frac{n}{W(n)}-1\), represent the ratio of normalized Stirling numbers \(S(n, k)\) to this expansion. The colors correspond to approximation orders \(s\): blue (\(s=1\)), red (\(s=2\)), green (\(s=3\)), magenta (\(s=4\)), and black (\(s=5\)). Additionally, the yellow curves, referencing the left $y$-axis, are normalized histograms showing the smooth convergence of the distribution shape toward a normal bell curve.}}\label{F:fin_dif}
\end{figure}

Figure~\ref{F:fin_dif} illustrates the numerical performance of the asymptotic expansion \eqref{E:FD}. The plots there exhibit the characteristic behavior of asymptotic expansions: although higher-order terms can yield greater accuracy near the mode, the expansion becomes increasingly unstable and ultimately divergent beyond the region of uniformity---as seen on the left portion of the ($s=3$) curve (green).

\section{A Poisson-Charlier expansion}
\label{S:PC}

The asymptotic expansion \eqref{E:FD} we derived in the last section encapsulates rich information about the quantity $D_{n,k}(j)$ (see \eqref{E:Dnkj}), which is simply the $j$th backward difference of the function $g(x)$ (see \eqref{E:fg}) evaluated at $x=j$. In this section, we explore a rearrangement based on the Taylor expansion of $g(x)$. This approach leads to a representation involving \emph{Charlier polynomials}, which are, up to scaling factors, the Taylor coefficients of $e^{-z}\lpa{1+\frac za}^b$ in $z$ for constants $a$ and $b$. Such polynomials (with minor sign changes) have been widely used in the context of de-Poissonization and its algorithmic applications; see \cite{Hwang2010} and the references therein for more information.

\subsection{Taylor expansion of $(e^x(1-x))^n$}

We begin with the Taylor expansion
\begin{align} \label{E:Charlier}
	(e^x(1 - x))^n 
	= \sum_{m \ge 0} \frac{\tau_m(n)}{m!}\, x^m 
	= 1 - \frac{n}{2}\,x^2 - \frac{n}{3}\,x^3 
	+ \frac{n(n - 2)}{8}\,x^4 + \cdots,
\end{align}
or 
\[
	g(x) 
	= \sum_{m \ge 0} \frac{\tau_m(n)}{m!} 
	\Lpa{\frac{x}{k}}^m.
\]

\begin{lemma}
\label{L:tau-rr} The polynomials $\tau_m(n)$ satisfy the recurrence
\begin{align}\label{E:tau-rr}
	\tau_m(n) 
	= (m - 1)( \tau_{m - 1}(n) - n \tau_{m - 2}(n)),
	\qquad(m\ge2),
\end{align}
with initial conditions \(\tau_0(n) = 1\) and \(\tau_1(n) = 0 \).		
\end{lemma}
The recurrence \eqref{E:tau-rr} implies, by induction, that \( \tau_m(n) \) is a polynomial in \( n \) of degree \( \lfloor m/2 \rfloor \).
\begin{proof}
By definition $\tau_m(n)=m![x^m]e^{nx}(1-x)^n$ and by the relation $(1 - x)^n = n![z^n] e^{(1 - x)z}$, we obtain
\[
	\tau_m(n)
	= m!n! [x^mz^n] \lpa{e^{nx} e^{(1 - x)z}} 
	= n! [z^n] \lpa{m![x^m] e^{x(n - z)} e^z} 
	= n![z^n] (n - z)^m e^z.
\]
Differentiating the function \( f(z) = (n - z)^{m-1} e^z \) and multiplying by \( z = n-(n-z) \), we obtain
\begin{align*}
    &z\frac{d}{dz}\lpa{(n-z)^{m-1} e^z}\\
	&\qquad =-(m-1)n(n-z)^{m-2} e^z
	+(n+m-1)(n-z)^{m-1}e^z-(n-z)^{m} e^z.
\end{align*}
Extracting the coefficients on both sides yields
\begin{align*}
  n\tau_{m-1}(n)
	= -(m-1)n\tau_{m-2}(n) + (n+m-1)\tau_{m-1}(n)-\tau_{m}(n).
\end{align*}
Then the recurrence \eqref{E:tau-rr} follows from rearranging terms.
\end{proof}

\subsection{The Poisson-Charlier expansion, I: Identity}

We now expand $S(n,k)$ using the polynomials $\tau_m$ as follows. 
\begin{prop} \label{P:PC}
The Stirling partition numbers satisfy the identity
\begin{align}\label{E:PC}
	S(n,k)
	= \sum_{m\ge0}\frac{\tau_m(n)}{m!k^m}\,Q_{n,k}(m),
\end{align}	
for $1\le k\le n$, where
\begin{equation}\label{E:Qmnk}
	Q_{n,k}(m) 
	= \sum_{0\le j \le k}\binom{k}{j}
	(-1)^j j^m e^{-\frac nkj}.
\end{equation}
\end{prop}
The expansion is similar to \eqref{E:st2-lbnz} but with terms arranged differently. In addition, the true value of the expansion \eqref{E:PC} lies in its asymptotic nature; see \eqref{E:PC-an} below. 
\begin{proof}
The expansion \eqref{E:PC} is obtained by substituting the expansion \eqref{E:Charlier} into \eqref{E:Snk-sum}, and by 
interchanging the sums.
\end{proof}
\begin{lemma}
The quantity $Q_{n,k}(m)$ satisfies
\begin{align}\label{E:Qnkm}
  Q_{n,k}(m)= \Lpa{1-\frac{\lambda}k}^{k}
	\sum_{0\le l\le \min\{m,k\}}\Stirling{m}{l}
	\binom{k}{l}l!(-\Lambda)^l,
\end{align}
where $\Lambda$ is defined in \eqref{E:Lambda}.
\end{lemma}
\begin{proof}
By the standard identity $j^m = \sum_{0\le l\le m} 
\Stirling{m}{l}j^{\underline{l}}$, we obtain
\begin{align*}
	Q_{n,k}(m) 
	&= \sum_{0\le l\le m}\Stirling{m}{l}
	\sum_{0\le j\le k}\binom{k}{j}(-1)^j
	j^{\underline{l}}e^{-\frac nkj},
\end{align*}
for $m\ge0$. The inner sum in the above identity can be expressed as
\[
	\sum_{0\le j\le k}\binom{k}{j}(-1)^j
	j^{\underline{l}}e^{-\frac nkj}
	= x^l\frac{d^l}{dx^l}(1-x)^k
	\Biggl|_{x=e^{-\frac{n}{k}}}
	=\binom{k}{l}l!(-1)^l e^{-\frac nkl}
	\lpa{1-e^{-\frac nk}}^{k-l},
\]
which then yields \eqref{E:Qnkm}.
\end{proof}

Although both \eqref{E:PC} and \eqref{E:Qnkm} are recursive through their dependence on $\Stirling{m}{l}$, the asymptotic usefulness of \eqref{E:PC} will become clear in the asymptotic form \eqref{E:PC-an} below; by contrast, \eqref{E:Qnkm} involves only $\Stirling{m}{l}$ with bounded $(m,l)$.

\subsection{Taylor remainder of $(e^x(1-x))^n$}

\begin{lemma}\label{tau_series_estimate}
For all \( 0\le x \le 1 \) and \( m_0 \ge 1 \), the following estimate holds:
\begin{align}\label{E:tau-effective}
	\left| \bigl(e^x(1 - x)\bigr)^n 
	- \sum_{0 \le m < m_0} \frac{\tau_m(n)}{m!}x^m \right| 
	\le m_0 2^{m_0}(x\sqrt{n})^{m_0}.
\end{align}
\end{lemma}
\begin{proof}
Our proof relies on the coefficient bound 
\begin{equation}\label{E:tau-bound}
	\left|\frac{\tau_m(n)}{m!}\right| \le n^{\frac m2},
\end{equation}
which follows from the same proof as used in \eqref{E:coeff-g} by using the formal expansion
\[
	\sum_{m \ge 0} \frac{\tau_m(n)}{m!}x^m
	=\bigl(e^x(1 - x)\bigr)^n
	=\exp\left(-n\sum_{r\ge 2}\frac{x^r}{r}\right),
\] 
and then dominating its coefficients by those of \(\frac1{1-\sqrt{n}x}\). Since the function \(\bigl(e^x(1 - x) \bigr)^n\) is entire, the difference between the function and its truncated Taylor series is exactly the tail sum for all \(x\). We divide the analysis into two cases.

Consider first the case where \(x\sqrt{n} < \frac{1}{2}\). Bounding the remainder directly using \eqref{E:tau-bound} yields
\[
\begin{split}
	\left| \bigl(e^x(1 - x)\bigr)^n 
	- \sum_{0 \le m < m_0} 
	\frac{\tau_m(n)}{m!}x^m \right| 
	&\le \sum_{m \ge m_0} (x\sqrt{n})^m 
	= \frac{(x\sqrt{n})^{m_0}}{1 - x\sqrt{n}}\\
	&< 2(x\sqrt{n})^{m_0}
	\le m_0 2^{m_0}(x\sqrt{n})^{m_0},
\end{split}
\]
where the last inequality holds for all \(m_0 \ge 1\). This establishes \eqref{E:tau-effective}.

Now assume \(x\sqrt{n} \ge \frac{1}{2}\). In this regime, we separate the function and the truncated sum using the triangle inequality. Since \(\tau_0(n) = 1\) and \(0 \le e^x(1 - x) \le 1\) for \(x \in [0,1]\), we have \(\bigl|\bigl(e^x(1 - x)\bigr)^n - 1\bigr| \le 1\). Applying \eqref{E:tau-bound} to the remaining terms gives
\[
\begin{split}
	\left| \bigl(e^x(1 - x)\bigr)^n 
	- \sum_{0 \le m < m_0} \frac{\tau_m(n)}{m!}x^m \right| 
	&\le \bigl|\bigl(e^x(1 - x)\bigr)^n - 1\bigr| 
	+ \sum_{1 \le m < m_0} \left| 
	\frac{\tau_m(n)}{m!} \right| x^m \\
	&\le 1 + \sum_{1 \le m < m_0} (x\sqrt{n})^m \\
	&= \sum_{0 \le m < m_0} (x\sqrt{n})^m.
\end{split}
\]
We bound each term in this finite sum by the maximum possible value. Since \(x\sqrt{n} \ge \frac{1}{2}\) implies \(1 \le 2^{m_0}(x\sqrt{n})^{m_0}\), we obtain
\[
	\sum_{0 \le m < m_0} (x\sqrt{n})^m 
	\le m_0 \max\bigl\{1, (x\sqrt{n})^{m_0}\bigr\}
	\le m_0 2^{m_0}(x\sqrt{n})^{m_0}.
\]
This completes the proof.
\end{proof}

\subsection{Asymptotic nature of the Poisson-Charlier expansion, I}

We now justify the asymptotic nature of the Poisson-Charlier expansion \eqref{E:PC}. 
\begin{thm}\label{Poi_Char_expansion}
If 
\begin{align*}
	\frac{n}{W(n)}\le k\le \frac{2n}{\log n+6},
	\eqtext{or}
	1 \le \lambda \le \frac{2}{e^3}\cdot
	\frac{\sqrt{n}}{\log n+6},
\end{align*}
then
\begin{equation} \label{E:PC-an}
\begin{split}
	S(n,k) = \left( 1-\frac{\lambda}{k}\right)^{k} 
	\left(1+\sum_{2\le m<2s_1}\frac{\tau_m(n)}{m!k^m}
	\,\bar Q_{n,k}(m)
	+O\left(\left(\frac{\lambda\log n}{\sqrt{n}}\right)^{2s_1}
	\right)\right),
\end{split}
\end{equation}
for any fixed \(s_1\ge 1\), where $\bar Q_{n,k}(m) := Q_{n,k}(m) (1-e^{-\frac{n}{k}})^{-k}$ satisfies 
\[
	\bar Q_{n,k}(m)
	=\sum_{0\le j\le m}\binom{k}{j}\Stirling{m}{j}j!
	(-\Lambda)^j.
\]
\end{thm}
Unlike Theorem~\ref{T:finite-diff-exp}, here we group the terms in pairs because $\tau_m(n)$ is a polynomial in $n$ of degree $\tr{\frac m2}$.
\begin{proof}
We apply Theorem~\ref{T:finite-diff-exp} with $s=2s_1ele$:
\[
\begin{split}
	S(n,k)= \left( 1-\frac{\lambda}k\right)^{k} 
	\Bigl(\sum_{0\le j<2s_1}\binom{k}{j}(-\Lambda)^j
	D_{n,k}(j)+O\Lpa{\Lpa{ 
	\frac{\lambda\log n}{\sqrt{n}}}^{2s_1}}\Bigr). 
\end{split}
\]
By Lemma~\ref{tau_series_estimate}, we can evaluate $D_{n,k}(j)$ as follows. 
\begin{align*}
	D_{n,k}(j)&=\nabla_{\!x}^j 
	\left(e^{\frac nkx}\Lpa{1-\frac{x}{k}}^n\right)
	\Bigl|_{x=j}
	=\sum_{0\le m<2s_1}\frac{\tau_m(n)}{m!k^m}
	\nabla_{\!x}^j x^m \Bigl|_{x=j} 
	+O\Lpa{\frac{n^{s_1}}{k^{2s_1}}}.
\end{align*}
Note that $\nabla_{\!x}^j x^m\bigl|_{x=j} =j! \Stirling{m}{j}$. We then have
\[
\begin{split}
	\sum_{0\le j<2s_1}\binom{k}{j}(-\Lambda)^jD_{n,k}(j)
	&=\sum_{0\le m<2s_1}\frac{\tau_m(n)}{m!k^m}
	\sum_{0\le j<2s_1}\binom{k}{j}(-\Lambda)^j
	j! \Stirling{m}{j}\\
	&\qquad+O\left(\left(\frac{n}{k^2}\right)^{s_1}
	\sum_{0\le j<2s_1}\frac{(k\Lambda)^j}{j!}\right).
\end{split}
\]
The dominant term here equals that of \eqref{E:PC-an} by \eqref{E:Qnkm}. Since $\lambda\ge1$, the $O$-term is also commensurate with the error term in \eqref{E:PC-an}.
\end{proof}

Note that the error term in \eqref{E:PC-an} should be $O(\lambda (\log n)^{2s_1}n^{-s_1})$ when $\lambda<1$ since $\bar Q_{n,k}(m)=O(\lambda)$ for $m\ge1$.

\subsection{Asymptotic nature of the Poisson-Charlier expansion, II}
For $m=1,2,3$, we have
\begin{align*}
	\bar{Q}_{n,k}(1) &= -\frac{\lambda}{1-\frac{\lambda}k},\quad
	\bar{Q}_{n,k}(2) = \frac{\lambda(\lambda-1)}
	{\lpa{1-\frac{\lambda}k}^2}, \quad
	\bar{Q}_{n,k}(3) = -\frac{\lambda(\lambda^2-3\lambda+1)
	+\frac{\lambda^2}{k}}{\lpa{1-\frac{\lambda}k}^3}.
\end{align*}
For large $k$, $\bar{Q}_{n,k}(m)\sim\sum_{1\le j\le m}\Stirling{m}{j}(-\lambda)^j$.

The first few terms of the expansion~\eqref{E:PC-an} are given as follows:
\begin{equation}\label{E:PC-two-terms}
	\frac{S(n,k)}{\lpa{1-\frac{\lambda}{k}}^k}
	=1-\frac{n\lambda(\lambda-1)}{2(k-\lambda)^2}
	\,+ \frac{n\lambda}{3(k-\lambda)^3}
	\Lpa{\lambda^2-3\lambda+1+\frac{\lambda}k}+ \cdots.
\end{equation}
The growth order of the second-order term is less transparent; in special cases, we have
\[
	\frac{S(n,k)}{\lpa{1-\frac{\lambda}{k}}^k}
	= \begin{cases}
		\lpa{1+O\lpa{n^{-1}((\log n)^{2c}+(\log n)^{c+1})}}, &
		\text{if }k=\frac{n}{\log n - c\log\log n},
		c\in\mathbb{R},\\
		\lpa{1+O\lpa{n^{-c}(n^{1-c}+\log n)}}, &
		\text{if }k=\frac{n}{c\log n},
		c>\frac12,\\
		\lpa{1+O\lpa{c^{-1}}}, &
		\text{if }k=\frac{2n}{\log n+\log c}, c\to\infty.
	\end{cases}
\]

The numerical fits are satisfactory when $k$ remains less than $\frac{2n}{\log n}$; see Figure~\ref{F:Poi_Char} for graphical renderings. 
\begin{figure}[!ht]
	\begin{center}
	\includegraphics[width=10cm]{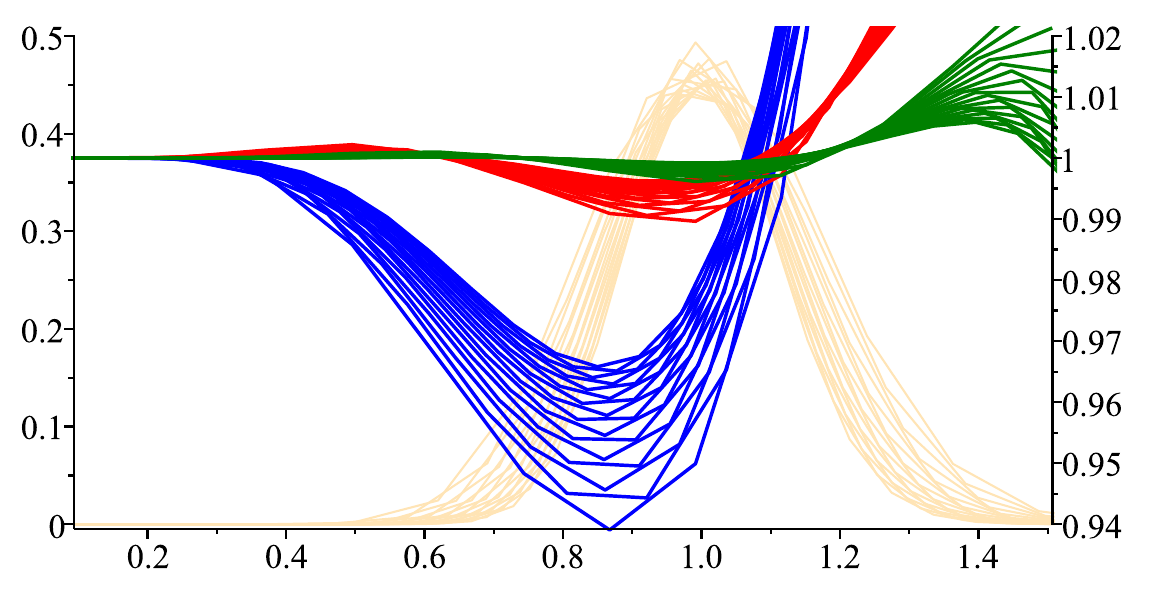}\\[0.5ex]
	{\footnotesize
	\textcolor{blue}{\raisebox{0.45ex}{\rule{1.6em}{1.0pt}}\, $s_1=1$} \quad
	\textcolor{red}{\raisebox{0.45ex}{\rule{1.6em}{1.0pt}}\, $s_1=2$} \quad
	\textcolor{green!50!black}{\raisebox{0.45ex}{\rule{1.6em}{1.0pt}}\, $s_1=3$} \quad
	\textcolor{orange!60!}{\raisebox{0.45ex}{\rule{1.6em}{1.0pt}}\, densities}}
	\end{center}
	\caption{\emph{Ratios of the right-hand side of \eqref{E:PC-an} and $\Stirling{n}{k}$ for $n=20,22,\dots,50$ (plotted against $\frac{n}{W(n)}-1$): blue curves correspond to $s_1=1$, red ones to $s_1=2$, and green ones to $s_1=3$. The yellow lines correspond to the densities of $\Stirling{n}{k}$ for each $n$, which show particularly the convergence of the histograms of $\Stirling{n}{k}$ towards normal for $k$ near the mode of the distribution.}}\label{F:Poi_Char}
\end{figure}

\subsection{Numerical comparisons}

The formulas \eqref{E:FD} in Theorem \ref{T:finite-diff-exp} and \eqref{E:PC-an} in Theorem \ref{Poi_Char_expansion} employ finite difference and Poisson-Charlier polynomial approaches, respectively, with numerical values shown in Figures \ref{F:fin_dif} and \ref{F:Poi_Char}. We abbreviate these as FD and PC. Since Theorem \ref{Poi_Char_expansion} extends Theorem \ref{T:finite-diff-exp}, their structures are similar. Using the first $s-1$ terms of FD and the first $2s_1-1$ terms of PC yields error terms of the same order; see Figure~\ref{contrast_FD_PC}. More figures below show that FD performs better below the mean while PC is superior above it, with PC's variability relative to FD in creasing dramatically as $k$ grows.

\begin{figure}[!ht]
\begin{center}
\begin{tabular}{cc}
\includegraphics[height=4.5cm]{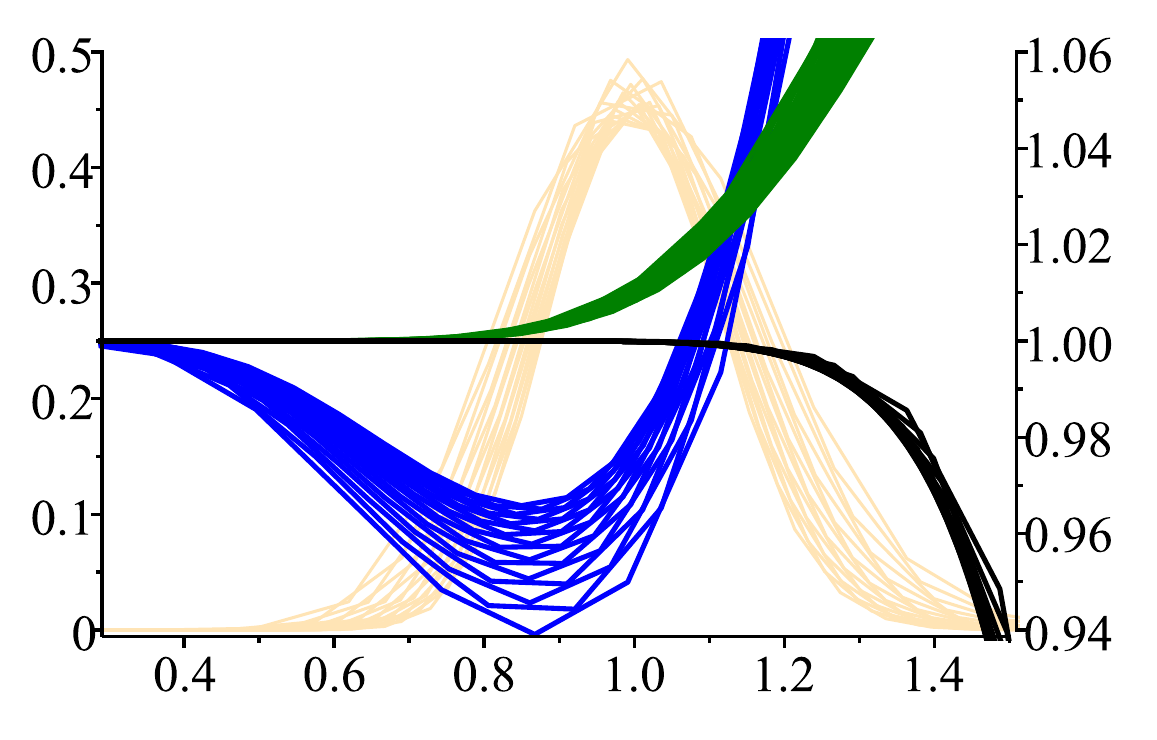}
& \includegraphics[height=4.5cm]{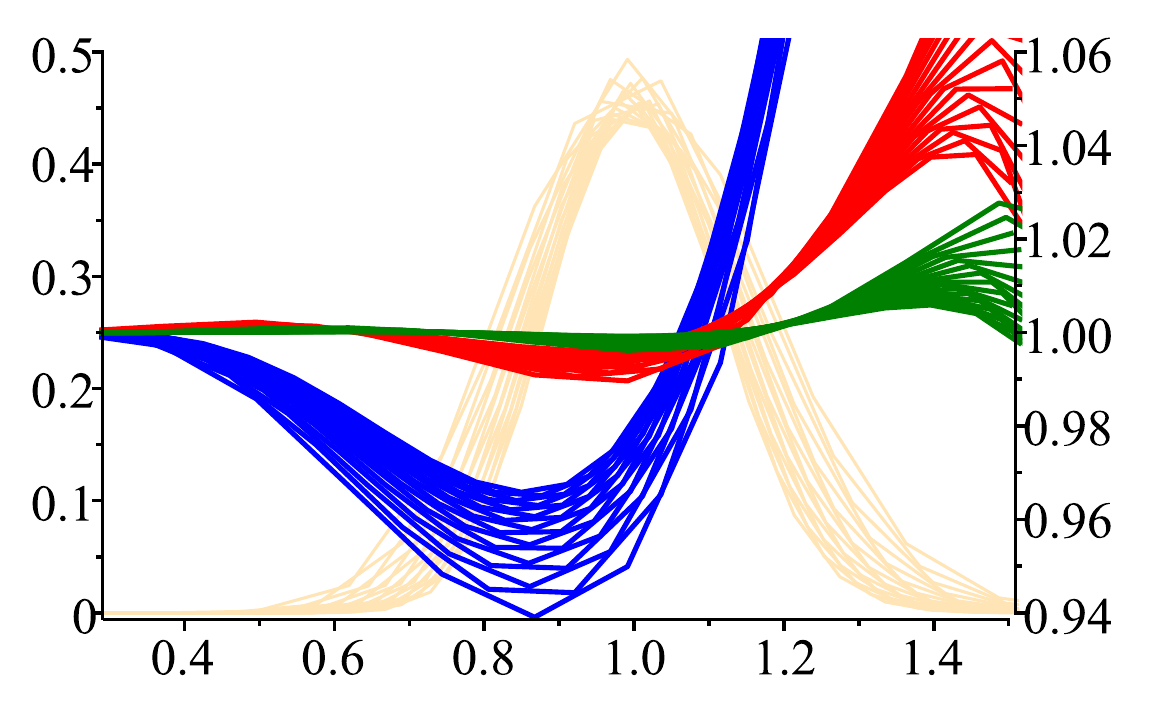}\\
FD in \eqref{E:FD} with&
PC in \eqref{E:PC-an} with\\
{\footnotesize
\textcolor{blue}{\raisebox{0.45ex}{\rule{1.6em}{1.0pt}}\, $s=1$}\,
\textcolor{green!50!black}{\raisebox{0.45ex}{\rule{1.6em}{1.0pt}}\, $s=3$}\,
\textcolor{black}{\raisebox{0.45ex}{\rule{1.6em}{1.0pt}}\, $s=5$}\,
\textcolor{orange!60}{\raisebox{0.45ex}{\rule{1.6em}{1.0pt}}\, densities}}&
{\footnotesize
\textcolor{blue}{\raisebox{0.45ex}{\rule{1.6em}{1.0pt}}\, $s_1=1$}\,
\textcolor{red}{\raisebox{0.45ex}{\rule{1.6em}{1.0pt}}\, $s_1=2$}\,
\textcolor{green!50!black}{\raisebox{0.45ex}{\rule{1.6em}{1.0pt}}\, $s_1=3$}\,
\textcolor{orange!60}{\raisebox{0.45ex}{\rule{1.6em}{1.0pt}}\, densities}}
\end{tabular}
\end{center}
\caption{\emph{Ratios of $\Stirling{n}{k}$ to their asymptotic estimates for $n=20, 22,\ldots,50$: FD with $s=1,3,5$ (left) and PC with $s_1=1,2,3$ (right).}}
\label{contrast_FD_PC}
\end{figure}

\begin{figure}[!ht]
  \centering
  \begin{subfigure}[t]{0.48\textwidth}
    \centering
    \includegraphics[width=7cm]{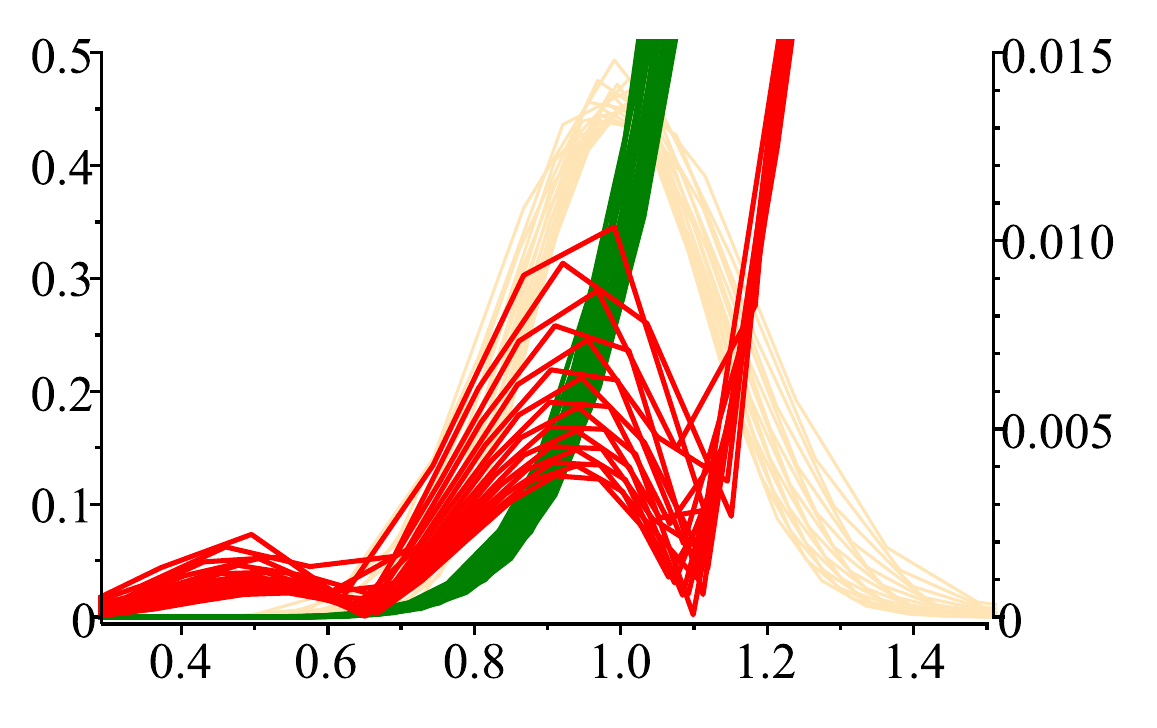}\\[-1ex]
    {\footnotesize
	\textcolor{green!50!black}{\raisebox{0.45ex}{\rule{1.6em}{1.0pt}}\,FD with $s=3$}
	\,vs\,
    \textcolor{red}{\raisebox{0.45ex}{\rule{1.6em}{1.0pt}}\,PC with $s_1=2$}}
  \end{subfigure}
  \hfill
  \begin{subfigure}[t]{0.48\textwidth}
    \centering
    \includegraphics[width=7cm]{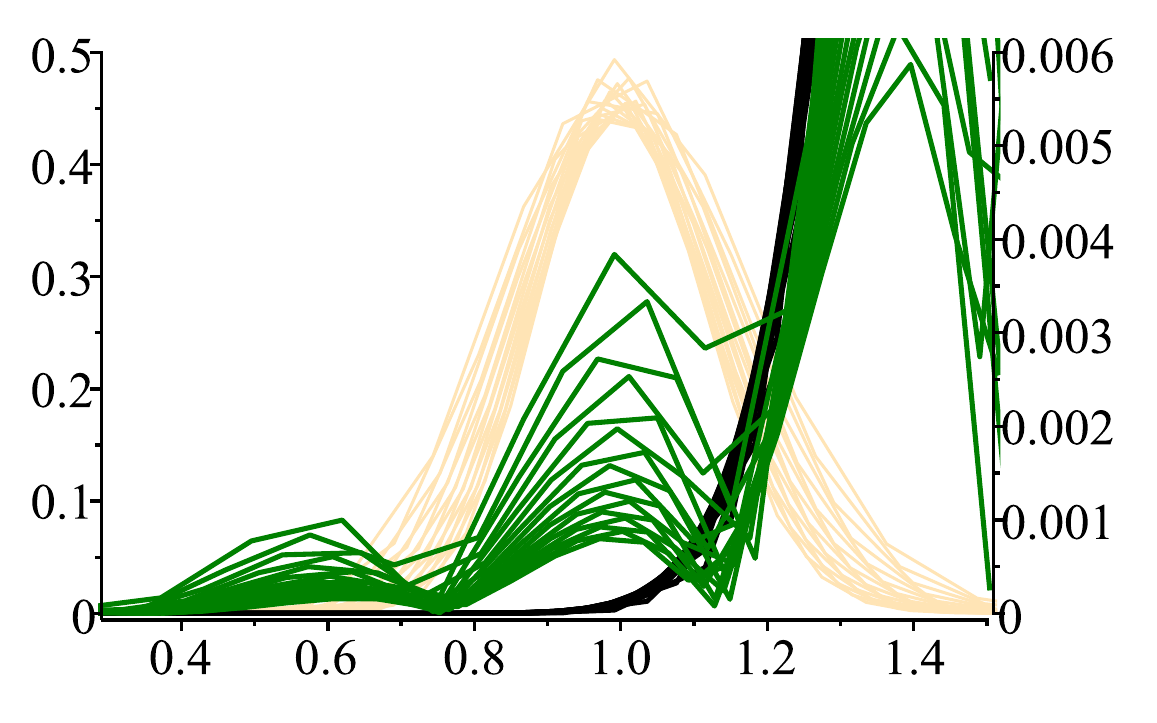}\\[-1.5ex]
    {\footnotesize
	\textcolor{black}{\raisebox{0.45ex}{\rule{1.6em}{1.0pt}}\,FD with $s=5$}
	\,vs\,
    \textcolor{green!50!black}{\raisebox{0.45ex}{\rule{1.6em}{1.0pt}}\,PC with $s_1=3$}}
  \end{subfigure}  
  \caption{\emph{Comparison of FD in \eqref{E:FD} and PC in \eqref{E:PC-an} for $n=20, 22,\ldots,50$.}}\label{F:comparison_combined}
\end{figure}

\section{Expansions for $k\le \frac{mn}{\log n}$ based on Laplace's saddle-point approximation}

For a better numerical comparison later, before discussing other elementary asymptotic expansions for $S(n,k)$ in \eqref{E:Lap3}, we begin with Laplace's saddle-point expansion for $S(n,k)$, which, although not proved by elementary approaches, provides a more uniform approximation.

The history of asymptotic approximations to Stirling numbers begins with Laplace's pioneering and far-reaching expansions that he derived in the 1780s \cite{Laplace1785, Laplace1786}, and later modified in his monumental book \enquote{Th\'eorie analytique des probabilit\'es}, published in three versions from 1812 to 1820 \cite{Laplace1820}. See our companion paper for a more detailed account. 

Laplace's pioneering and far-reaching saddle-point expansion (see \cite{Laplace1785, Laplace1820, David1962}), established by a formal approach using his saddle-point method\footnote{Note that Laplace's original approach in his 1785 paper \cite{Laplace1785} is based on the manipulation of real integrals and the foundation of complex analysis had to wait for about four decades to be established by Cauchy; see \cite{Bottazzini2013}. The formal nature of Laplace's analysis was later made rigorous by Cauchy \cite{Cauchy1844} whose approach is mostly complex analysis, although eventually he converts all complex integrals to real ones.}, is of the form:
\begin{equation}\label{E:Lap-SPE}
	\Stirling{n}{k} \frac{k!}{n!}
	= \frac{R^{-n}(e^R-1)^k}{\sqrt{2 \pi V(R)}} 
	\Lpa{1+\frac{d_1(R)}{V(R)}+\frac{d_2(R)}{V(R)^2}+\cdots},
\end{equation}
for some explicitly computable $d_1(R)$ and $d_2(R)$ (not needed in this paper), where $R>0$ solves the equation
\begin{align}\label{E:bj}
	\frac{1-e^{-R}}{R} = \frac{1}{\rho_*}, 
	\eqtext{and}
	V(R) := (n+1)(R+1-\rho_*),\eqtext{where}
	\rho_* := \frac{n+1}k.
\end{align}

Essentially, the same expansion was later derived by Bleick and Wang \cite{Bleick1974} by applying the saddle-point method to a complex integral along a vertical path. Such an integration path was ascribed by Wegner in \cite{Wegner2012} to Curt Meyer: \enquote{\emph{In 1966 C. Meyer had the idea to use an integration along a vertical line to represent the Stirling numbers of the second kind. The integral
representation (2.1) of this paper is based on this fruitful idea and is essential for the main results in the present paper and the earlier publication.}}
 
A similar approach (using instead integrals on a circle) was previously adopted by several authors; see, e.g., Arfwedson \cite{Arfwedson1951}, Moser and Wyman \cite{Moser1958}, Good \cite{Good1961}, and Ivchenko \cite{Ivchenko1965}. Briefly, the main difference in the resulting expansion is that the saddle-point equation becomes $\frac{1-e^{-r}}{r} = \frac{k}{n}$.

From Laplace's classical expansion \eqref{E:Lap-SPE}, we can then derive more explicit asymptotic approximations for $\Stirling{n}{k}$ when $k$ is of order $\frac{mn}{\log n}$ for any $m>0$. Strangely, we do not find the resulting expansions \eqref{E:mn-logn} and \eqref{E:R-to-rho} in the literature.

First, it is easy to prove the existence and uniqueness of the solution $R$ of the equation \eqref{E:bj} when $1\le k\le n$ ($R =0$ when $k=n+1$); see \cite{Moser1958, Bleick1974}. Moreover, $R\to\infty$ when $k=o(n)$ and $R\to0$ when $k\to n^{-}$. Indeed, Laplace himself already derived a useful expansion of $R$ in terms of the ratio of $n+1$ and $k$.

\begin{lemma} \cite[p.~197]{Laplace1820} The solution $R>0$ 
	to the equation \eqref{E:bj} satisfies 
\begin{align}\label{eq:rho}
	R = \rho_*-\sum_{j\ge1}\frac{j^{j-1}}{j!}\,
	\rho_*^j e^{-j\rho_*},
\end{align}
where $\rho_*$ is defined in \eqref{E:bj}. The series is absolutely convergent when $1\le k\le n+1$. 
\end{lemma}
\begin{proof}
Since $R=0$ when $k=n+1$, we assume that $1\le k\le n$. The saddle-point equation \eqref{E:bj} can be written as 
\begin{align}\label{E:rho-lag}
	R = \rho_*(1-e^{-R})\eqtext{or}
  e^{-R} = e^{-\rho_*+\rho_* e^{-R}}.
\end{align}
By Lagrange inversion formula \cite[\S 3.8]{Comtet1974}, we get the 
expansion 
\[
	e^{-R} = \sum_{j\ge1}\frac{j^{j-1}}{j!}\,
	\rho_*^{j-1} e^{-j\rho_*}.
\]
Then equation \eqref{eq:rho} follows from the first equation in \eqref{E:rho-lag}.
\end{proof}
\begin{cor} For $1\le k\le n$, the inequalities
\begin{align}
	\rho_*-1\le R\le \rho_*
\end{align}
hold.
\end{cor}

\begin{prop} If $1\le k=o(n)$, then 
\begin{align}\label{E:stir2-k-son}
	S(n,k)
	= \lpa{1-e^{-R}}^{k-n}e^{-ne^{-R}}
	\lpa{1+O\lpa{Re^{-R} + n^{-1}}},
\end{align}
where $R>0$ solves the equation \eqref{E:bj}.
\end{prop}
\begin{proof}
First, by \eqref{E:Lap-SPE} and Stirling's formula,
\[
	\frac{n!R^{-n}(e^R-1)^k}
	{k^n\sqrt{2\pi V(R)}}
	= \sqrt{\frac{n+1}{V(R)}}\,
	\lpa{1-e^{-R}}^{k-n}e^{kR-n-1}
	\lpa{1+O\lpa{n^{-1}}}.
\]
Then, by \eqref{E:Lap-SPE}, \eqref{E:rho-lag} and the saddle-point equation \eqref{E:bj}, we have 
\[
	V(R) = \frac{kR(1-(R+1)e^{-R})}{(1-e^{-R})^2},
\]
and then
\[
	\sqrt{\frac{n+1}{V(R)}}
	= \sqrt{\frac{n+1}{kR}\cdot \frac{(1-e^{-R})^2}
	{1-(1+R)e^{-R}}}
	= \sqrt{\frac{1-e^{-R}}
	{1-(1+R)e^{-R}}}
	= 1+O\lpa{Re^{-R}}.
\]
Finally, 
\[
	e^{kR-n-1} = \exp\Lpa{-\frac{kRe^{-R}}{1-e^{-R}}}
	= e^{-(n+1)e^{-R}}.
\]
Thus \eqref{E:stir2-k-son} follows. 
\end{proof}

\begin{cor} If $k \le \frac{mn}{\log n+\eta_n}$ for any sequence 
	$\eta_n\to\infty$, $m\ge2$, then 
\begin{align}\label{E:mn-logn}
	S(n,k)
	= \lpa{1-e^{-R}}^{k} 
	\exp\llpa{n\sum_{2\le l<m}
	\frac{e^{-lR}}{l}}
	\lpa{1+O\lpa{Re^{-R}+ ne^{-mR}+n^{-1}}}.
\end{align}
\end{cor}
\begin{proof}
If $\frac nk-\frac1m\log n\to\infty$, then $ne^{-mR}=o(1)$.
\end{proof}
The expression \eqref{E:mn-logn} with $m=2$ first appeared in Laplace's book \cite[p.~199]{Laplace1820}.

We now express \eqref{E:mn-logn} in terms of $\rho_*$, which results in a slightly smaller range in $k$.
\begin{cor} If $k \le \frac{mn}{\log n+(m-2)\log\log n+\eta_n}$ 
for any sequence $\eta_n\to\infty$, $m\ge2$, or 
\[
	\frac{n}{k}-\frac1m\lpa{\log n+(m-2)\log\log n}\to\infty,
\]
for $m=2,3,\dots$, then 
\begin{equation} \label{E:R-to-rho}
	\begin{split}
		S(n,k)
		&= \lpa{1-e^{-\rho_*}}^{k}
		\exp\Lpa{-\frac{(n+1) e^{-2\rho_*}}{2}
		\sum_{0\le h\le m-3}J_h(\rho_*)e^{-h\rho_*}}\\
		&\qquad
		\times\lpa{1+O\lpa{\rho_* e^{-\rho_*}+k\rho_*^{m-1}e^{-m\rho_*}+n^{-1}}}.
	\end{split}
\end{equation}
where $J_h$ is a polynomial of degree $h$ defined by 
\[
	J_h(z):=\frac2{(h+2)^2z}
	[x^{h+1}]\frac{e^{(h+2)z x}-1}{(1-x)^2}
    =2\sum_{0\le l\le h}
	\frac{(h+2)^{l-1}(h+1-l)}{(l+1)!}\,z^l.
\]

\end{cor}
\begin{remark} 
In particular, $J_0(z)=1$, $J_1(z)=\frac{3z+4}3$, and $J_2(z)= \frac{8z^2+12z+9}6$. This implies that:
\begin{align*}
	\frac{S(n,k)}{(1-e^{-\rho_*})^k} \sim 
	\begin{cases}
		1, 
		& \text{if }k\le \frac{2n}{\log n+\eta_n},\\
		\exp\lpa{-\frac{n+1}2\,e^{-2\rho_*}},
		& \text{if }k\le \frac{3n}{\log n+\log\log n+\eta_n},\\
		\exp\lpa{-\frac{n+1}2\,e^{-2\rho_*}
		-\frac{n+1}6(3\rho_*+4)e^{-3\rho_*}},
		& \text{if }k\le \frac{4n}{\log n+2\log\log n+\eta_n},
	\end{cases}
\end{align*}
for any $\eta_n\to\infty$.
\end{remark}

\begin{proof}
By Lagrange inversion formula and the relation $R=\rho_*(1-e^{-R})$, 
we have
\[
	\log\frac{1-e^{-R}}{1-e^{-\rho_*}}
	= \sum_{h\ge2}d_h e^{-h\rho_*},
\]
where
\begin{align*}
	d_h &= \frac1h[z^{h-1}]
	e^{h\rho_* z}\frac{\mathrm{d}}{\mathrm{d} z}
	\log\frac{1-z}{1-ze^{-\rho_* z}}
	= -\frac1h\sum_{1\le l<h}\frac{(h\rho_*)^l}{l!}.
\end{align*}
Similarly, 
\[
	\sum_{2\le l<m}\frac{e^{-lR}}{l}
	= \sum_{h\ge2}e_{m,h}e^{-h\rho_*},
 \eqtext{where}
	e_{m,h}
	= \frac1h\sum_{2\le l\le \min\{m-1,h\}}
	\frac{(h\rho_*)^{h-l}}{(h-l)!}.
\]
It follows that
\begin{align*}
	&k\log\frac{1-e^{-R}}{1-e^{-\rho_*}}+n\sum_{2\le l<m}
	\frac{e^{-lR}}{l}\\
	&\qquad= -k\sum_{2\le h<m}\frac{e^{-h\rho_*}}{h}
	\sum_{1\le l<h}\frac{h^{l-1}(h-l)}{l!}\,\rho_*^l\\
	&\qquad\qquad -k\sum_{h\ge m}\frac{e^{-h\rho_*}}{h}
	\llpa{\sum_{h-m+2\le l\le h-1}\frac{h^{l-1}\rho_*^l}{l!}(h-l)
	+\sum_{1\le l\le h-m+1}
	\frac{(h\rho_*)^{l}}{l!}}.
\end{align*}
The second double-sum is of order $O(k\rho_*^{m-1}e^{-m\rho_*})$ when 
$\rho_*\to\infty$. This proves \eqref{E:R-to-rho}.
\end{proof}

\section{A modern synthesis of classical asymptotic expansions}

This section compiles the main known asymptotic approximations for $\Stirling{n}{k}$ that involve only simple elementary functions and hold uniformly for $k$ in the range 
\[
    \frac{(1-\ve)n}{\log n}\le k\le \frac{(2-\ve)n}{\log n}, 
\]
which we refer to as \emph{the central range} because it encompasses the LLT regime for the Stirling partition numbers. We show that all known approximations in the central range can be derived in a unified, purely formal manner. Building on these existing formulas, we introduce several new or modified versions that improve accuracy; their numerical performance is evaluated in Section~\ref{S:comparison}. The derivations of these new expansions follow the same approach as the Poisson--Charlier expansion \eqref{E:PC-an} and are deferred to Appendices~\ref{S:App-B}, \ref{S:App-C}, and~\ref{S:App-D}. For the numerical comparison in Section~\ref{S:comparison}, at most the first two terms of each expansion are used, whose justifications follow directly from either \eqref{E:FD2} or \eqref{E:PC-an}.

\subsection{Four types of expansions}

Beyond the less explicit saddle-point approximations, the more explicit asymptotic expressions for $S(n,k)$ in the central range typically fall into one of the two dominant forms: 
\begin{align}\label{E:bino-and-exp}
	\text{Binomial:}\;\;\Lpa{1-\frac{\lambda}k}^k,
	\eqtext{or} \text{Exponential:}\;\; e^{-\lambda},
\end{align}
depending on whether the binomial coefficient $\binom{k}{j}$ in \eqref{E:Snk-sum} is retained in its original form or approximated by $\frac{k^j}{j!}$.

When these are further combined with different approximations to the arithmetic factor $1-jt$, specifically, by a geometric progression of the form $e^{-jt}$ (exponential) or $(1-t)^j$ (binomial), we obtain four distinct asymptotic patterns. These are summarized in the tree diagram, with the parameters $\lambda := ke^{-\frac nk}$ and $\lambda_b := k(1-\frac 1k)^n$.
\begin{equation}\label{E:Lap3}
\begin{minipage}{0.9\linewidth}
\centering
\begin{tikzpicture}[grow=down, scale=0.7, transform shape,
level 1/.style={sibling distance=250pt, level distance=50pt, 
        nodes={rounded corners=5pt, thick, align=center, 
				fill=yellow!10, text=black, draw=blue}},
level 2/.style={sibling distance=120pt, level distance=50pt,
        nodes={rounded corners=5pt, align=center}},
level 3/.style={sibling distance=80pt, level distance=50pt,
        nodes={rounded corners=5pt, align=center}},
level 4/.style={sibling distance=80pt, level distance=50pt,
        nodes={rounded corners=5pt, align=center}},
edge from parent/.style={draw=black, shorten >=2pt, thick, 
shorten <=2pt}, every node/.style={rectangle},
nodeopt/.style={inner sep=7pt}
]

\node[nodeopt, inner sep=8pt, thick, draw=blue, fill=yellow!10] 	
{$S(n,k) := \sum\limits_{0\le j\le k}\binom{k}{j}(-1)^j(1-\frac jk)^n$}
  child {node[nodeopt, inner sep=7pt] 
	  {Binomial \\ $\binom{k}{j}=\binom{k}{j}$}
    child {node[nodeopt] {$1-\frac jk \sim e^{-\frac jk}$}
      child {node[nodeopt] {Bino-Exp: \\
			  $\lpa{1-\frac{\lambda}{k}}^k$}
			  child {node[nodeopt] {  
				Eqs.~\eqref{E:FD} \& \eqref{E:PC-an}}}}
    }
    child {node[nodeopt] {$1-\frac jk \sim (1-\frac1k)^j$}
      child {node[nodeopt] 
			  {Bino-Bino: \\
				$\lpa{1-\frac{\lambda_b}k}^k$}
        child {node[nodeopt, yshift=-10pt] {de Moivre \cite{deMoivre1712}\\
				Laplace \cite{Laplace1820}\\ 
				David-Barton \cite{David1962}}}
      }
    }
  }
  child {node[nodeopt, inner sep=7pt] 
  {Exponential\\
	$\binom{k}{j}\sim \frac{k^j}{j!}$}
	    child {node[nodeopt] {$1-\frac jk \sim e^{-\frac jk}$} 
	      child {node[nodeopt] {Exp-Exp: \\ $e^{-\lambda}$}
				  child {node[nodeopt] {Laplace \cite{Laplace1785}\\ 
					Cayley \cite{Cayley1888}}}}
	    }
	    child {node[nodeopt] {$1-\frac jk \sim (1-\frac1k)^j$}
	      child {node[nodeopt] 
				  {Exp-Bino: \\ $e^{-\lambda_b}$}
	        child {node[nodeopt] {Cayley \cite{Cayley1888}}}
	      }
	    }
	  };
\end{tikzpicture}	
\end{minipage}
\end{equation}

\medskip

Surprisingly, these elementary approximations have remained little known since their first appearance in the 18th and 19th centuries, and the binomial forms discussed here (Eqs.~\eqref{E:FD} and \eqref{E:PC-an}) appear not to have been treated in earlier work.

We show how these four asymptotic forms in \eqref{E:Lap3} can be understood and constructed in a unified framework, which is different from Laplace's original derivation via \eqref{E:Lap-SPE} or Cayley's formal arguments used in \cite{Cayley1888}. For simplicity of presentation and page length, we only outline the underlying formal ideas. 

The use of $e^{-jt}$ or $(1-t)^j$ to approximate $1-jt$ is not unique, and in theory any functions $\phi(t)$ analytic at the origin with the property that 
\[
	\phi(t)^j = 1-jt + O(t^2),\qquad (t\sim0)
\]
may be adopted. For simplicity and historical reasons, we restrict our discussion below to the two cases $\phi(t) = e^{-t}$ and $\phi(t) = 1-t$.

\subsection{Expansions of binomial type: de Moivre and Laplace}
We show in this subsection how de Moivre and Laplace derive the binomial-type expansions for $S(n,k)$.

\subsubsection{Replacing $1-jt$ with $(1-t)^j$}
The first approximation to alternating sums of the form \eqref{E:Snk-sum} already appears, as early as 1712, in de Moivre’s memoir on the measurement of chance (see \cite{deMoivre1712} and the English translation in \cite{Hald1984}). The crucial idea is to replace the factor $1-jt$ with $(1-t)^j$, yielding the approximation 
\begin{equation}\label{E:bino-exp}
	\begin{split}
		\sum_{0\le j\le k'}\binom{k'}{j}(-1)^j
		\Lpa{1-\frac jk}^n 
		&\approx \sum_{0\le j\le k'}\binom{k'}{j}(-1)^j
		\Lpa{1-\frac 1k}^{jn}\\
		&= \Lpa{1-\Lpa{1-\frac1k}^n}^{k'}
		= \Lpa{1-\frac{\lambda_b}k}^{k'}.
	\end{split}
\end{equation}
Here de Moivre claims that this approximation is useful when $k'\ll k$. In the preface to his later published book \cite[Preface]{deMoivre1756}, de Moivre attributed this \enquote{artifice} of \enquote{changing an arithmetic progression into a geometric one} to Edmond Halley (1656–1742), calling it \enquote{\emph{a very remarkable method of solution}.}

Laplace later applied the same substitution in a more general setting, taking $k'=k$, to estimate the probability of collecting all $k$ distinct tickets after $n$ stages, where at each stage $s\ge1$ tickets are randomly selected. He began with the exact expression
\begin{align}
	\sum_{0\le j\le k}
	\binom{k}{j}(-1)^j
	\llpa{\frac{\binom{k-j}{s}}{\binom{k}{s}}}^n
	&= \sum_{0\le j\le k}
	\binom{k}{j}(-1)^j\prod_{0\le i<s}
	\left(1-\frac{j}{k-i}\right)^n\nonumber \\
	&\approx \sum_{0\le j\le k}
	\binom{k}{j}(-1)^j\Lpa{1-\frac sk}^{jn}\nonumber\\
	&= \Lpa{1-\Lpa{1-\frac sk}^n}^{k}, 
	\label{E:Lap-lotto}
\end{align}
valid when $k$ is large. Here Laplace extended de Moivre's restriction of $k'\ll k$ in \eqref{E:bino-exp} to $k'=k$.

Then Laplace used this approximation with $k=90$ and $s=5$ to estimate the smallest $n$ such that the probability of obtaining a complete collection exceeds one-half. From this, he obtained the estimate $n\approx 85.204$ \cite[p.~203]{Laplace1820}. Today we can compute the exact probabilities numerically:
\begin{align*}
  \sum_{0\le j\le k}
	\binom{k}{j}(-1)^j
	\llpa{\frac{\binom{k-j}{s}}{\binom{k}{s}}}^n
    \approx \begin{cases}
	0.48909\,90163, & \text{if } (n,k,s) = (85,90,5),\\
	0.50930\,98536, & \text{if } (n,k,s) = (86,90,5).
    \end{cases}
\end{align*}
Laplace described the procedure \eqref{E:Lap-lotto} as \enquote{\emph{an extremely simple and very accurate method for obtaining the value of $n$ (un moyen fort simple et tr\`es approché d'obtenir la valeur de $n$)}.}

\subsubsection{Usefulness of the approximation}
The above Bino-Bino type of approximation was later critically examined in David and Barton's book \cite[Ch.~16]{David1962}, where they refer to \eqref{E:Lap-lotto} with $s=1$ as de Moivre's approximation. After a numerical comparison with several other expansions, they offered a sharply contrasting assessment \cite[p.~317]{David1962}:
\begin{quote}
\enquote{\emph{De Moivre's approximation is not at all useful. While the approximations of Laplace become better with increasing $n$, that put forward by de Moivre will actually become worse, and it is suggested that this approximation should never be used.}}	
\end{quote}

This marked discrepancy in judgements is largely due to the distinction between \emph{pointwise accuracy} and \emph{uniform closeness}. De Moivre and Laplace used their approximation to identify the threshold value of $n$ for a fixed $k$ such that the probability exceeds $\frac12$, and the numerical efficiency was satisfactory for their needs. In contrast, David and Barton conducted a numerical comparison of eight different asymptotic expansions (including saddle-point approximations) for $n=20$ and varying $k$ between $1$ and $15$, focusing particularly on the uniformity of the approximations across all $k$ in this range. 

Approximations like the one above, such as \eqref{E:bino-exp}, are valid primarily in the range $k\le \frac{(2-\ve)n}{\log n}$, making them inferior in uniform accuracy to saddle-point methods. However, saddle-point approximations tend to involve cumbersome and unwieldy computation when solving equations like $S(n,k)=\frac12$. In such cases, simpler approximations like \eqref{E:Lap-lotto} remain valuable for their ease of use and effective performance in estimating numerical thresholds.
 
Methodologically, while the justification of the formal expansion \eqref{E:Lap-lotto} (and its extended version to an expansion) requires different approaches and deeper analysis (as we synthesized in Section~\ref{sec:numerical}), its back-of-the-envelope nature makes it the right choice in giving a first-order estimate when no other finer ones are available. On the other hand, to apply the saddle-point method (as Laplace already worked out formally in \cite{Laplace1785}) with $s>1$, one is naturally led to work on multidimensional complex integrals which complicates the analysis. Alternatively, Laplace switched to a different idea in \cite{Laplace1820} by asymptotically reducing the multivariate finite-difference problem to an expansion of univariate finite-difference one.

Finally, it is worthwhile to mention that Bernstein used induction in his book \cite[p.~75]{Bernstein1934} to derive, by a conditional argument, the upper bound
\[
	S(n,k) < \Lpa{1-\Lpa{1-\frac1k}^n}^k.
\]
Bernstein deduced this bound from the probabilistic interpretation of $S(n,k)$ in the context of the occupancy problem. Let $n$ distinct balls be distributed independently and uniformly among $k$ distinct bins. If $E_i$ denotes the event that bin $i$ is non-empty for $i=1, \dots, k$, then the probability that all bins are occupied is given by
\[
  P\left( \bigcap_{1\le i\le k} E_i \right) = \frac{k! }{k^n}\Stirling{n}{k}=S(n,k).
\]
Using the chain rule, we decompose this intersection into conditional probabilities:
\[
  P\left( \bigcap_{1\le i\le k} E_i \right) = P(E_1) \prod_{2\le m\le k} P(E_m \mid E_1 \cap \dots \cap E_{m-1}).
\]
Intuitively, conditioning on the event that the first $m-1$ bins are occupied implies that at least $m-1$ balls have been 'expended' to satisfy those conditions, thereby stochastically reducing the number of balls available to occupy bin $m$. Since $P(E_i) = 1 - (1 - 1/k)^n$ for each $i$, it follows that
\begin{equation}
\label{ineq:neg-asssociation}
 P(E_m \mid E_1 \cap \dots \cap E_{m-1}) < P(E_m).  
\end{equation}
yielding the desired inequality
\[
  {S(n,k)} < \prod_{1\le i\le k} P(E_i) = \left( 1 - \left( 1 - \frac{1}{k} \right)^n \right)^k.
\]
The inequality \eqref{ineq:neg-asssociation} can be made rigorous by invoking negative association theory; see e. g. \cite{doubhashi_devdatt_ranjan_1998}.

\subsection{A generic formal construction of the Bino-Bino expansion}

We now show how de Moivre's and Laplace's original ideas for deriving the asymptotic approximations \eqref{E:bino-exp} and \eqref{E:Lap-lotto} can be extended to obtain a full asymptotic expansion. 

For simplicity, we consider only the case when $s=1$ or $S(n,k)$. Let $t := k^{-1}$, $n=\frac{\rho}t$ and $\lambda_b := k(1-\frac1k)^n=t^{-1}(1-t)^{\frac\rho t}$. Then $(1-\frac{j}{k})^n=(1-jt)^{\frac\rho t}$. Now expand (with respect to $t$) the ratio 
\begin{align*}
	\frac{(1-\frac jk)^n}{(1-\frac1k)^{jn}}
	&= (1-jt)^{\frac\rho t}(1-t)^{-\frac{j\rho}t}
	:= \sum_{m\ge0}c_{m}^{[\text{bb}]}(j) t^m\\
	&= 1-\frac{\rho j(j-1)t}{2}
	+\frac{\rho j(3\rho j^3 - 2(3\rho+4)j^2+
	3\rho j+8)t^2}{24}+\cdots,
\end{align*}
where $c_{m}^{[\text{bb}]}(j)$ is a polynomial in $j$ of degree $2m$. With this expansion, we then obtain formally the approximation (with $t=k^{-1}$)
\begin{align}
	S(n,k) &\sim \sum_{m\ge0}k^{-m}
	\sum_{0\le j\le k}\binom{k}{j}\Lpa{-\frac{\lambda_b}k}^j
	c_{m}^{[\text{bb}]}(j) \nonumber \\
	&= \Lpa{1-\frac{\lambda_b}k}^k
	\biggl(1-\frac{\rho \lambda_b^2(k-1)}{2(k-\lambda_b)^2}
	\label{E:bb}\\
	&\qquad+\frac{\rho\lambda_b^2(k-1)}{24(k-\lambda_b)^4}
	\left(\begin{array}{l}
		\lpa{3\rho(\lambda_b^2-4\lambda_b+2)+8(\lambda_b-3)}k^2\\
		-\lambda_b(3\rho+8)(\lambda_b-4)k-8\lambda_b^2
	\end{array}
	\right)+\cdots\biggr).\nonumber
\end{align}
It can be shown that 
\[
  \sum_{0\le j\le k}\binom{k}{j}\Lpa{-\frac{\lambda_b}k}^j
	c_{m}^{[\text{bb}]}(j) 
	= O\Lpa{\Lpa{1-\frac{\lambda_b}k}^{k-m}\rho^m \lambda_b^{2m}}
  \qquad(m=0,1,\dots).
\]
Thus the (outer) sum in \eqref{E:bb} is expected to be an asymptotic expansion as long as $\rho\lambda_b^2=o(k)$ or $ne^{-2n/k}=o(1)$ for any $\ve>0$.

\subsubsection{Replacing $1-jt$ with $e^{-jt}$}
Another natural choice as we used above \eqref{E:fg} for approximating the factor $1-jt$ is $e^{-jt}$. The construction of such Bino-Exp type is not unique, as already visible from the two expansions \eqref{E:FD} and \eqref{E:PC-an} that we analyzed above. They have not appeared explicitly in the literature as far as we know although they are implicit in Laplace's analysis in the form $(1-e^{-R})^k$ \cite[p.~199]{Laplace1820}, as resulting from his saddle-point expansion, where $R>0$ solves the equation $(n+1)(1-e^{-R})=kR$. 

In addition to the above two expansions \eqref{E:FD} and \eqref{E:PC-an}, we describe yet another one which follows more or less the idea we used above to construct \eqref{E:bb}. By the same notations with $\lambda := ke^{-\frac nk}=t^{-1}e^{-\rho}$, we can expand the ratio (with respect to $t$) 
\begin{align*}
	\Lpa{1-\frac jk}^ne^{\frac{jn}k}
	&= (1-jt)^{\frac \rho t}e^{j\rho}
	:= \sum_{m\ge0}c_{m}^{[\text{be}]}(j) t^m
	= 1-\frac{\rho j^2t}{2}
	+\frac{\rho j^3(3\rho j - 8)t^2}{24}+\cdots,
\end{align*}
where $c_{m}^{[\text{be}]}(j)$ is a polynomial in $j$ of degree $2m$. With this definition, we then obtain the formal expansion (with $t=k^{-1}$)
\begin{align}
	S(n,k) &\sim \sum_{m\ge0}k^{-m}
	\sum_{0\le j\le k}\binom{k}{j}\lpa{-\frac{\lambda}k}^j
	c_{m}^{[\text{be}]}(j)\nonumber \\
	&= \Lpa{1-\frac{\lambda}k}^k
	\biggl(1-\frac{\rho\lambda(\lambda-1)k}{2(k-\lambda)^2}
	\label{E:be}\\
	&\quad+\frac{\rho{\lambda}}{24(k-\lambda)^4}
	\left(\begin{array}{l}
		\lpa{3(\lambda^3 - 6\lambda^2+7\lambda-1)\rho
		+ 8(\lambda^2 - 3\lambda +1)}k^2\\
		+4\lambda\lpa{3(\lambda-1)\rho - 2\lambda(\lambda-3)}k
		- (3\rho + 8)\lambda^2
	\end{array}
	\right)+\cdots\biggr).\nonumber
\end{align}
This expansion is to be compared with \eqref{E:PC-two-terms}: the leading term and the first correction in equation \eqref{E:be} agree with those in equation \eqref{E:PC-two-terms}, while the next correction differs. It can be shown that 
\[
  \sum_{0\le j\le k}\binom{k}{j}\Lpa{-\frac{\lambda}k}^j
	c_{m}^{[\text{be}]}(j) 
	= O\Lpa{\Lpa{1-\frac{\lambda}k}^{k-m}\rho^m {\lambda}^{2m}}
  \qquad(m=0,1,\dots).
\]
Thus the (outer) sum in \eqref{E:be} is expected to be an asymptotic expansion as long as $\rho{\lambda}^2=o(k)$ or $ne^{-2n/k}=o(1)$ for any $\ve>0$.

\subsection{Expansions of exponential type: Laplace and Cayley}
The expansions above retain the binomial coefficients $\binom{k}{j}$ in the sum \eqref{E:Snk-sum}, which, when further approximated by exponential factors of the form $\frac{k^j}{j!}$, yields expansions with leading terms of the form $e^{-\lambda}$. The first such approximations appeared in Laplace's memoirs \cite[\S~XXVI]{Laplace1785} and \cite[\S~XLIV]{Laplace1786} as a simple consequence of his saddle-point expansion \eqref{E:Lap-SPE}. Other variants are listed in the following table.
\medskip
\begin{center}
\begin{tabular}{c|cccc}
	Reference & \makecell{Laplace 
	\cite{Laplace1785,Laplace1786}} 
	& \makecell{Laplace \cite{Laplace1820}}
	& \makecell{Cayley \cite{Cayley1888}}
	& \makecell{Menon \cite{Menon1973}} \\ 
	Year & 1785--1786 & 1820 & 1888 & 1973 \\ \hline
	$S(n,k)\sim$ & $e^{-ke^{-\frac nk}}$ 
	& $e^{-ke^{-\frac{n+1}k}}$ 
	& $e^{-ke^{-\frac nk}}$
	& $e^{-ke^{-\frac nk+\frac1{2k}-\frac1{12k^2}}}$
\end{tabular}	
\end{center}
\medskip

Motivated by these forms, it is natural to consider the generic form:
\begin{equation}\label{E:ee-gen}
	S(n,k) \sim e^{-\lambda(\alpha)},\eqtext{with}
	\lambda(\alpha) := k\exp\Lpa{-\frac nk -\frac{\alpha}k}
	= ke^{-\rho-\alpha t},
\end{equation}
for some constant $\alpha\in\mathbb{R}$. Then we begin with the expansion
\begin{equation}\label{E:bnep}
	\begin{split}
	\binom{k}{j}\Lpa{1-\frac jk}^n
	&= \frac{\lambda(\alpha)^j}{j!}
	(1-jt)^{\frac \rho t}e^{j(\rho+\alpha t)}
	\prod_{1\le l<j}(1-lt)\\
	&:= \frac{\lambda(\alpha)^j}{j!}
	\sum_{m\ge0}c_{m}^{[\text{ee}]}(j) t^m,
	\end{split}
\end{equation}
where $c_{m}^{[\text{ee}]}(j)$ is a polynomial in $j$ of degree $2m$. The first few terms are given by 
\begin{align*}	
	\sum_{m\ge0}c_{m}^{[\text{ee}]}(j) t^m
	= 1-\frac{jt}{2}\lpa{(\rho+1)j-2\alpha-1}
	+\frac{jt^2}{24}
	\left(\begin{array}{l}
		3(\rho+1)^2j^3\\
		-2(6(\rho+1)\alpha+7\rho+5)j^2\\
		+3(4\alpha^2+4\alpha+3)j-2
	\end{array}\right)+\cdots.
\end{align*} 
With these polynomials, we then obtain the formal expansion
\begin{align}
	S(n,k) &\sim \sum_{m\ge0}k^{-m}
	\sum_{j\ge0}\frac{(-\lambda(\alpha))^j}{j!}\,
	c_{m}^{[\text{ee}]}(j)\nonumber\\
	&= e^{-\lambda(\alpha)}\Lpa{1-
	\frac{\lambda(\alpha)}{2k}\lpa{(\lambda(\alpha)-1)\rho
	+\lambda(\alpha)+2\alpha}
	+\cdots}.\label{E:ee0}
\end{align}
In particular, the first error term satisfies 
\begin{align}\label{E:varying-d}
	-\frac{1}{2k}\times
	\begin{cases}
		\lambda\lpa{(\lambda-1)\rho+\lambda}, 
		&\text{if }\alpha=0 \text{ (Laplace, Cayley)};\\
		\lambda(1)\lpa{(\lambda(1)-1)\rho+\lambda(1)+2}, 
		&\text{if }\alpha=1 \text{ (Laplace)};\\
		\lambda(-\frac12)(\lambda(-\frac12)-1)(\rho+1), 
		&\text{if }\alpha=-\frac12 \text{ (Menon)}.
	\end{cases}
\end{align}
Menon \cite{Menon1973} obtained a more refined expansion than Laplace's (with $\lambda_e' := ke^{-\frac nk+\frac1{2k}-\frac1{12k^2}}$)
\begin{align}\label{E:menon-2}
	S(n,k) = e^{-\lambda_e'}
	\Lpa{1-\frac{\lambda_e'(\lambda_e'-1)}{k^2}
	\Lpa{\frac{n+k}2-\frac14}+\cdots},
\end{align}
which corresponds, up to an error of order $k^{-2}$, to $\alpha=-\frac12$; see \S~\ref{S:error-reduced} for an extension.

\medskip

\noindent\textbf{Error-reduced expansions.} While other values of $\alpha$ in \eqref{E:varying-d} may be selected, none of these seems to be \emph{optimal} for varying $k$ in the sense of making the error term as small as possible. Instead, if we take (recursively)
\begin{align}\label{E:d}
	\alpha
	=\frac{\rho}2-\frac{(\rho+1)}2\,\lambda(\alpha),
\end{align}
then the first error term after $1$ in \eqref{E:ee0} (involving $\alpha$) becomes zero, which provides better numerical efficiency. With this choice of $\alpha$, the expansion \eqref{E:ee0} becomes
\begin{align}\label{E:ee}
	S(n,k) &= e^{-\lee}\Lpa{1-
	\frac{\lee}{24k^2}
	\lpa{3\rho^2\lee(\lee - 2) 
	-2\rho(\lee^2-6\lee+4)-\lee^2}
	+\cdots}.
\end{align}
Here $\lee$ is given \emph{recursively} by, in view of \eqref{E:ee-gen} and \eqref{E:d}, 
\[
	\lee
	= k\exp\Lpa{-\frac{n}{k}-
	\frac{\rho}{2k}+\frac{(\rho+1)}{2k}\,\lee},
\]
which is solved as
\begin{align}\label{E:La}
	\lee = \frac{2k}{\rho+1}\,
	T\Lpa{\frac{(\rho+1)e^{-\rho-\frac\rho{2k}}}2},
\end{align}
where $T(z)=-W(-z)$ denotes the Cayley tree function and satisfies the equation $T(z)=ze^{T(z)}$ with the Taylor expansion $T(z) = \sum_{j\ge1}\frac{j^{j-1}}{j!}z^j$.
              
For large $\rho = \frac nk>\frac12\log n$, the special form \eqref{E:La} has the Taylor expansion
\begin{equation}\label{E:lambda-star-k}
	\begin{split}
	\lee
	&= \sum_{h\ge1}t^{h-1}
	\sum_{1\le m\le h}\frac{m^{h-1}
	(\rho+1)^{m-1}\lambda^m(-\rho)^{h-m}}
	{m!(h-m)!2^{h-1}}\\
	&= \lambda +\frac{\lambda(\lambda-1)\rho+\lambda^2}
	{2k}+\frac{\lambda}{8k^2}
	\lpa{(\lambda-1)\rho+\lambda}
	\lpa{(3\lambda-1)\rho+3\lambda}+\cdots.
	\end{split}
\end{equation}
We see that the major error term \eqref{E:varying-d} is now \enquote{incorporated} into $\lee$ (as the second-order term in the right-hand side of \eqref{E:lambda-star-k}). Such an \enquote{error shift} or \enquote{error reduction} technique is completely general (at least formally) and can be applied to other expansions; see Section~\ref{S:error-reduced}.

While the solution of $\lee$ in \eqref{E:La} looks more complicated than the original $\lambda$, its numerical evaluation is rather straightforward in most modern symbolic systems; on the other hand, one may use the expansion \eqref{E:lambda-star-k} for a similar numerical purpose if Lambert $W$-function is not available. Additionally, the error terms in the two expansions \eqref{E:ee0} and \eqref{E:ee} show further advantages of \eqref{E:ee} because of wider range of uniformity. Note that the terms inside the large parentheses in \eqref{E:ee} are of order $1+O(e^{-\eta})$ when $k\le \frac{3n}{\log n+\log\log n+\eta}$, which approaches $1$ as $\eta$ increases. 

\subsubsection{Replacing $1-jt$ with $(1-t)^{-j}$}
In the course of re-deriving Laplace's exponential approximation $S(n,k)\sim e^{-\lambda}$, Cayley \cite{Cayley1888} first argues that $S(n,k)$ is close to $e^{-\lambda_b}$, where $\lambda_b := k(1-\frac1k)^n$, and then approximates such an expansion in terms of $\lambda$. We construct the full asymptotic expansion using $\lambda_b$ by the same approach used above, beginning with
\[
	(1-jt)^{\frac\rho t}
	(1-t)^{-\frac{j\rho}t}\prod_{1\le l<j}(1-lt)
	= \sum_{m\ge0}c_m^{[eb]}(j) t^m.
\]
Then from this we derive the expansion
\begin{align}
	S(n,k) &\sim \sum_{m\ge0}k^{-m}
	\sum_{j\ge0}\frac{(-\lambda_b)^j}{j!}\,
	c_{m}^{[\text{eb}]}(j)\nonumber\\
	&= e^{-\lambda_b}\Lpa{1-
	\frac{(\rho+1)\lambda_b^2}{2k}
	+\frac{\lambda_b^2}{24k^2}
	\lpa{3(\rho+1)^2\lambda_b^2-4(3\rho^2+4\rho+2)\lambda_b
	+6\rho(\rho-2)}
	+\cdots}\label{E:eb}.
\end{align}

\section{Numerical discussions}
\label{S:comparison} 

In this section, we first present graphical comparisons of the various approximations to $S(n,k)$ discussed above. We then analyze their principal differences within the central range, with particular emphasis on explaining the discrepancies observed in Figure~\ref{F:compare_all}.

\subsection{Graphical renderings}

We assess the accuracy of the six types of asymptotic expansions for $S(n,k)$ introduced in the preceding sections, using the following notation:
\begin{align}\label{E:lll}
	\lambda := ke^{-\frac nk}, \quad
	\lambda_b:= k(1-\tfrac1k)^n\eqtext{and}
	\lambda_e' := ke^{-\frac nk+\frac1{2k}-\frac1{12k^2}},
\end{align}
and the expansions:
\begin{equation}\label{E:six}
	\begin{cases}
		\text{de Moivre \eqref{E:bb}:}&D(n,k)
		\sim\lpa{1-\frac{\lambda_b}k}^k
		\lpa{1-\frac{\rho \lambda_b^2(k-1)}{2(k-\lambda_b)^2}}\\
		\text{Finite difference \eqref{E:FD2}:}& F(n,k)
		\sim \lpa{1-\frac{\lambda}{k}}^k
		\lpa{1-k\Lambda D_{n,k}(1)+\binom{k}{2}
		\Lambda^2D_{n,k}(2)}\\
		\text{Poisson-Charlier \eqref{E:PC-two-terms}:} 
		&P(n,k) \sim \lpa{1-\frac{\lambda}{k}}^k
		\lpa{1-\frac{\rho k\lambda(\lambda-1)}
		{2(k-\lambda)^2}}\\
		\text{Laplace \eqref{E:ee0}:}&L_\alpha(n,k)
		\sim e^{-\lambda(\alpha)}\lpa{1-
		\frac{\lambda(\alpha)}{2k}\lpa{(\lambda(\alpha)-1)\rho
		+\lambda(\alpha)+2\alpha}}\\
		\text{Cayley \eqref{E:eb}:} &C(n,k) 
		\sim e^{-\lambda_b}\lpa{1-
		\frac{(\rho+1)\lambda_b^2}{2k}}\\
		\text{Menon \eqref{E:menon-2}:} &M(n,k) \sim
		e^{-\lambda_e'}\lpa{1-\frac{(2(\rho+1)k-1)
		\lambda_e'(\lambda_e'-1)}{4k^2}}.
	\end{cases}
\end{equation}
The finite difference version ($F$) and the Poisson-Charlier ($P$) have the same leading term, but with different errors. 

Our numerical comparisons are based on the absolute error measure:
\begin{equation}\label{E:error-value}
	\Delta_f(n,k) :=\left|\frac{S(n,k)}{f(n,k)}-1\right|,
\end{equation}
where $f(n,k)$ denotes one of the six approximations in \eqref{E:six}. Their smallness is summarized and illustrated through Figure~\ref{F:compare_all},

 \begin{figure}[!ht]
 \begin{center}
 \begin{minipage}[t]{0.49\textwidth}
 \begin{overpic}[width=\textwidth]{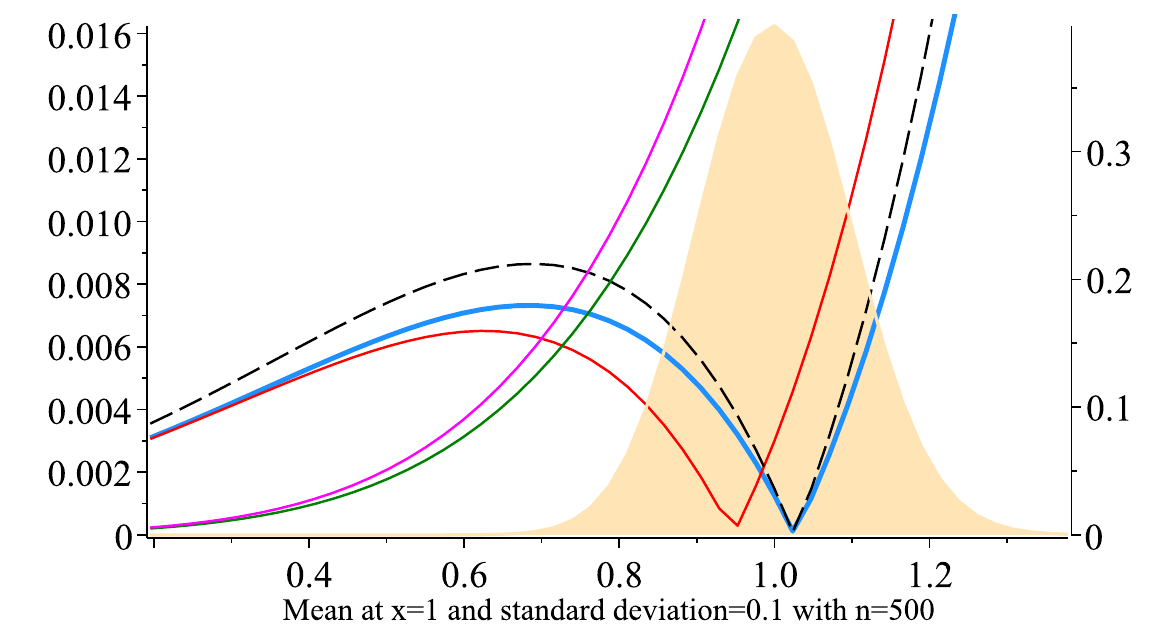}
 \put(79,33){\color{blue}{$P\& F$}}
 \put(37,33){\color{blue}{$M$}}
 \put(51,16){\color{blue}{$L_{0}$}}
 \put(52,42){\color{blue}{$C$}}
 \put(58,35){\color{blue}{$D$}}
 \end{overpic}
 \end{minipage}
 \hfill
 \begin{minipage}[t]{0.49\textwidth}
 \begin{overpic}[width=\textwidth]{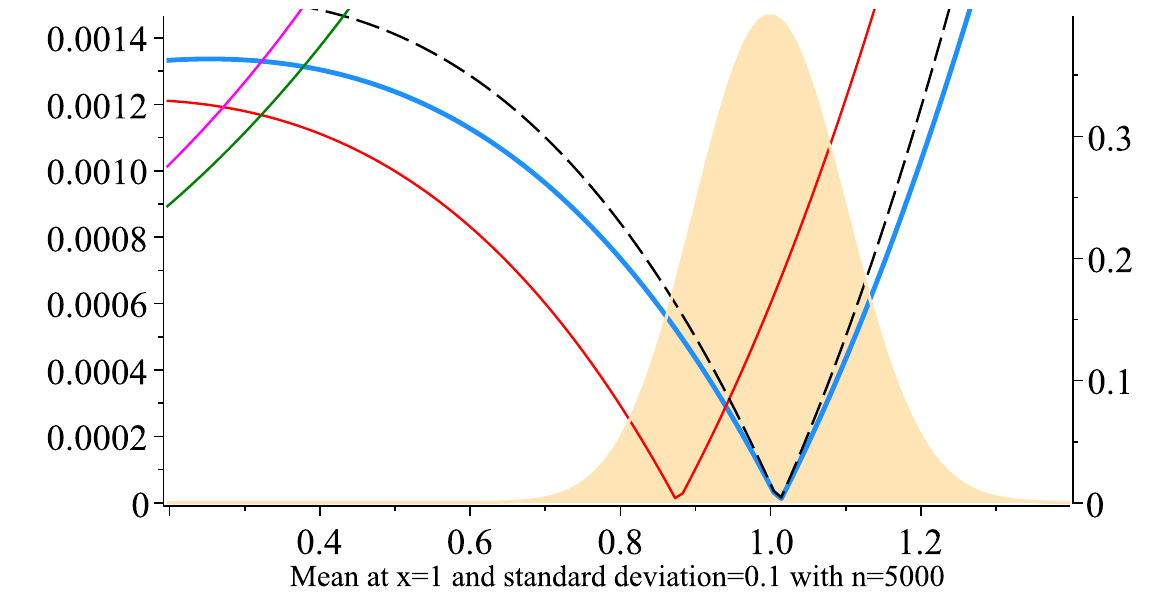}
 \put(35,38){\color{blue}{$P\& F$}}
 \put(38,47){\color{blue}{$M$}}
 \put(33,31){\color{blue}{$L_{0}$}}
 \put(20,49){\color{blue}{$C$}}
 \put(16,32){\color{blue}{$D$}}
 \end{overpic}
 \end{minipage}
 
\caption{\emph{Closeness $\Delta_f$ of the different approximations to $S(n,k)$, as defined in \eqref{E:six}: Cayley ($C$), de Moivre ($D$), Laplace ($L$), Menon ($M$), and our Poisson--Charlier ($P$). The plots correspond to $n=500$ (left) and $n=5000$ (right). The left vertical axis shows the scale of the error terms $\Delta_f$, while the right vertical axis represents the distribution of $\frac{1}{B_n}\Stirling{n}{k}$, shown in soft yellow for ease of comparison. The mean is located at $x=1$, and each increment of $0.1$ on the horizontal axis corresponds to one standard deviation.}} \label{F:compare_all}
\end{center}
\end{figure}

\subsection{Asymptotics of $\Delta_f(n,k)$ for $k$ near the mean}

To clarify why the errors introduced by the two binomial approximations: Bino-Bino $D(n,k)$ and Exp-Bino $C(n,k)$ (i.e., those based on approximating $1-jt$ by $(1-t)^j$), behave so differently from the others, we observe the following pattern: their errors are smaller than those of all other approximations when $k$ is smaller than and away from the mean of the distribution, but they grow much more rapidly and fluctuate more strongly as $k$ approaches the mean. We begin by examining the case when
\[
  k\approx k_c := \frac{n}{\log n-c\log\log n},
	\qquad(c\in\mathbb{R}).
\]
In this case, 
\[
	\lambda_b \sim \lambda\Lpa{1-\frac{n}{2k_c^2}},
	\eqtext{and}
	\lambda_e' \sim \lambda\Lpa{1+\frac{1}{2k_c}},
\]
and the errors in the approximations \eqref{E:six} have the orders:
\begin{align}
	n\Delta_f(n,k)
	\asymp (\log n)^{c}\times
	\begin{cases}
		(\log n)^c, & \text{if }f\in\{D,C\}\\
		|\log n-c\log\log n|, & \text{if }f=F\\
		|(\log n)^c-\log n+c\log\log n|, 
		& \text{if }f\in\{P,L_\alpha,M\}.
	\end{cases}
\end{align}
In particular, when $c=1$, 
\begin{align}
	n\Delta_f(n,k)
	\asymp \log n\times
	\begin{cases}
		\log n, & \text{if }f\in\{D,F,C\}\\
		\log\log n, & \text{if }f\in\{P,L_\alpha,M\}.
	\end{cases}
\end{align}
Thus, for $c<1$ the binomial approximations are more accurate, whereas for $c>1$ the finite-difference expansion $F$ yields superior performance. Near the central range, the exponential approximations achieve the smallest errors.

\subsection{The cusps in Figure~\ref{F:compare_all}}

Two distinct behaviors of $\Delta_f$ are observed in Figure~\ref{F:compare_all} for $k$ near the central range. The errors either increase monotonically (as in the cases of $\Delta_D$ and $\Delta_C$), or exhibit a sign change (as for $\Delta_P$, $\Delta_F$, $\Delta_{L_\alpha}$, and $\Delta_M$), which gives rise to the observed \enquote{cusps} in their absolute values. These phenomena are direct consequences of the following asymptotic approximations:
\begin{equation}
	\begin{split}
	&n\Lpa{\frac{S(n,k)}{f(n,k)}-1}\\
	&\quad\sim \begin{cases}
		-\frac12(\log n)^{2c},
		& \text{if }f \in\{D,C\};\\
 		-\frac12(\log n)^c\lpa{(\log n)^c-\log n+c\log\log n}
    ,& \text{if }f \in\{P,F,L_\alpha,M\}.
	\end{cases}
	\end{split}	
\end{equation}

\subsection{Intersections between $D$ and others}
We begin by rewriting all the error terms in \eqref{E:six} in terms of $\lambda$:
\begin{equation}\label{E:six-2}
	\begin{array}{llclcl}
		\text{de Moivre \eqref{E:bb}}&D(n,k)
		&\sim&\lpa{1-\frac{\lambda_b}k}^k&\times&
		\lpa{1-\frac{\rho\lambda^2}{2k}}\\
		\begin{array}{@{}l@{}}
		\text{Finite difference \eqref{E:FD2}}\\
		\text{\& Poisson-Charlier \eqref{E:PC-two-terms}}
		\end{array} & F(n,k) &\sim& 
		\lpa{1-\frac{\lambda}{k}}^k&\times&
		\lpa{1-\frac{\rho\lambda^2}{2k}
		+\frac{\rho\lambda}{2k}}\\[1ex]
		\text{Laplace \eqref{E:ee0}}&L_\alpha(n,k)
		&\sim& e^{-\lambda(\alpha)} &\times& \lpa{1-
		\frac{\rho\lambda^2}{2k}
		+\frac{\rho\lambda}{2k}
		-\frac{\lambda^2}{2k}
		-\frac{\alpha\lambda}{k}}\\
		\text{Cayley \eqref{E:eb}} &C(n,k) 
	  &\sim& e^{-\lambda_b} &\times& \lpa{1-
		\frac{\rho\lambda^2}{2k}\;\,
		\phantom{+\frac{\rho\lambda}{2k}}
		-\frac{\lambda^2}{2k}}\\
		\text{Menon \eqref{E:menon-2}:} & M(n,k)
		&\sim& e^{-\lambda_e'} &\times& \lpa{1-
		\frac{\rho\lambda^2}{2k}
		+\frac{\rho\lambda}{2k}
		-\frac{\lambda^2}{2k}
		+\frac{\lambda}{2k}}.
	\end{array}
\end{equation}
Since our definition of $\Delta_f$ is based on absolute values, any intersection of two curves observed in Figure~\ref{F:compare_all} occurs asymptotically either when their error terms are equal or when they differ only by a sign. Among these possibilities, we are interested in the smaller such intersection point. For instance, the intersection of the curves corresponding to $D$ and $F$ in Figure~\ref{F:compare_all} occurs asymptotically at the point where
\[
	\frac{\rho\lambda^2}{2k}
	\sim -\frac{\rho\lambda^2}{2k}
	+\frac{\rho\lambda}{2k},\eqtext{or}
	\lambda\sim\frac12, \eqtext{or}k\sim\frac n{W(2n)}.
\]
Similarly, for the intersection of $D$ and $L_\alpha$ on Figure~\ref{F:compare_all}, we have
\[
	\lambda\sim\frac{\rho-2\alpha}{2\rho+1},
	\eqtext{or} \rho = W(2n) 
	+ \frac{4\alpha+1}{2W(2n)}+\cdots.
\]
Finally, the two curves $D$ and $M$ intersect asymptotically at 
\[
	\lambda\sim\frac{\rho+1}{2\rho+1}
	\eqtext{or} \rho = W(2n)
	- \frac{1}{2W(2n)}+\cdots.
\]
Since $W(2n)=\log n -\log\log n + \log 2+ o(1)$, we see that these intersections occur before the mean of the distribution (asymptotically at $\frac{n}{W(n)}-1$), as is visible from Figure~\ref{F:compare_all}.

\section{Refined error-reduced expansions}
\label{S:error-reduced}

Motivated by the expansion \eqref{E:ee} and its use of a free parameter for error reduction, we consider further expansions obtained by the same procedure.

\subsection{Refined binomial expansions}

As in the derivation of \eqref{E:ee}, define
\[
	\lambda_b(d) = k\left(1-\frac{1}{k}\right)^{n+d}.
\]
The value of $d$ that eliminates the leading error term in \eqref{E:bb} is given by (with $x=\lambda_b(d)$)
\[
	d = -\frac{\rho(k-1)x}
	{2(k-x)}.
\]
This choice then leads to the recursive equation satisfied by the optimal $x=\lambda_b(d)$:
\[
	x = k\Lpa{1-\frac1k}^{n-\frac{x(k-1)n}
	{2k(k-x)}},
\]
with $x>0$. As no simple exact solution is available for such an equation for $x$, we use the approximate equation 
\[
	x = k\Lpa{1-\frac1k}^{n-\frac{xn}{2k}},
\]
yielding the solution
\[
	\lbb = \frac{2k}{n\log\frac1{1-\frac1k}}
	\,T\llpa{\frac n2\Lpa{1-\frac1k}^n\log\frac1{1-\frac1k}}.
\]
With this $\lbb$, we then have 
\begin{equation}\label{E:stir-lambda-bb}
  S(n,k)
  = \Lpa{1-\frac{\lbb} k}^k
	\llpa{1-
	\frac{\rho\lbb^2(3\rho{\lbb} 
	- 6\rho - 8{\lbb} + 18)}{24(k-{\lbb})^2}
	+\cdots}
\end{equation}
Asymptotically, $\lbb$ satisfies, with $\lambda_b = k(1-\frac1k)^n$,
\begin{align*}
	\lbb &= \sum_{m\ge0}\frac{t^m}{m!}
	\sum_{0\le j\le m}\stirling{m}{j}
	(j+1)^{j-1}\lambda_b^{j+1}\Lpa{\frac\rho2}^j
	= \lambda_b +\frac{\rho\lambda_b^2}{2k}
	+\frac{\rho\lambda_b^2(3\rho\lambda_b+2)}{8k^2}+\cdots,
\end{align*}
where $\stirling{m}{j}$ denotes the signless Stirling numbers of the first kind. In terms of $\lambda$, this yields 
\[
	\lbb = \lambda +\frac{\rho\lambda(\lambda-1)}
	{2k}+\frac{3\rho^2\lambda(3\lambda-1)(\lambda-1)
	+2\rho\lambda(3\lambda-4)}{24k^2}
	+\cdots.
\]

The same procedure applies to extend the Bino-Exp expansion \eqref{E:be}, and we obtain 
\begin{equation}\label{E:stir-lambda-be}
	S(n,k) = \Lpa{1-\frac{\lbe}k}^k
	\llpa{1-\frac{3\rho^2\lbe^2(\lbe - 2) - 
	8\rho\lbe(\lbe^2 - 3\lbe + 1)}
	{24(k-\lbe)^2}+\cdots},
\end{equation}
where $\lbe>0$ solves the equation
\[
	x = k e^{-\frac{n}{k}+\frac{(x-1)n}{2k^2}}.
\]
The solution is given by
\[
	\lbe = \frac{2k^2}{n}
	T\Lpa{\frac{n}{2k}e^{-\frac nk-\frac{n}{2k^2}}},
\]
which has the same form as \eqref{E:La}. Thus, by \eqref{E:lambda-star-k}, we have 
\begin{align*}
	\lbe 
	&= \sum_{h\ge1}t^{h-1}\Lpa{\frac{\rho}2}^{h-1}
	\sum_{1\le m\le h}\frac{(-1)^{h-m}\lambda^m m^{h-1}}
	{m!(h-m)!}\\
	&= \lambda+\frac{\rho\lambda(\lambda-1)}{2k}
	+\frac{\rho^2\lambda(3\lambda-1)(\lambda-1)}
	{8k^2} +\cdots.
\end{align*}

\subsection{Refined exponential expansions}

Refining the Exp-Bino expansion \eqref{E:eb} by the same error-reduction procedure, we obtain 
\begin{equation}\label{E:stir-lambda-eb}
	S(n,k) = e^{-\leb}
	\llpa{1-\frac{\leb^2}{24(k-\leb)^2}
	\lpa{3\rho^2(\leb - 2) -2\rho(\leb-3)
	-\leb-6}+\cdots},
\end{equation}
where $\leb>0$ solves the equation
\[
	x = k\Lpa{1 - \frac1k}^{n
	- \frac{\rho+1}2x},
\]
with the solution
\[
	\leb = \frac{2k}{(n+k)\log\frac1{1-\frac1k}}
	T\Lpa{\frac{n+k}2\Lpa{1-\frac1k}^{n}
	\log\frac1{1-\frac1k}}.
\] 
When $k\le \frac{2n}{\log n}$, we have 
\begin{align*}
	\leb &= \sum_{m\ge0}\frac{t^m}{m!}
	\sum_{0\le j\le m}\stirling{m}{j}
	(j+1)^{j-1}\lambda_b^{j+1}\Lpa{\frac{\rho+1}2}^j\\
	&= \lambda_b+\frac{(\rho+1)\lambda_b^2}{2k}
	+\frac{(\rho+1)\lambda_b^2(3\rho\lambda_b+3\lambda_b+2)}
	{8k^2} +\cdots.
\end{align*}

Following Menon~\cite{Menon1973}'s expansion \eqref{E:menon-2} with
\[
  \lambda_e'
	:= k\exp\Lpa{-\frac nk+\frac1{2k}-\frac1{12k^2}},
\]
we can extend the same error-reduction technique by considering $\hat \lambda = ke^{-\rho-\alpha t-\beta t^2}$ and then identifying the optimal choices for $(\alpha,\beta)$: $\alpha$ is given in \eqref{E:d} and 
\[
	\beta = -\frac{3\hat\lambda(\hat\lambda -2)\rho^2
	-2({\hat\lambda}^2-6\hat\lambda+4)
	\rho-\hat{\lambda}^2}{24}.
\]
While we cannot solve the resulting equation for 
\begin{align}\label{E:hat-lambda}
	\hat\lambda = \lambda\exp\Lpa{- 
	\frac{\rho-(\rho+1)\hat\lambda}{2k}
	-\frac{8\rho+6\rho(\rho-2)\hat\lambda-
	(\rho-1)(3\rho+1){\hat\lambda}^2}{24k^2}},
\end{align}
its asymptotic and numerical values can be readily computed:
\begin{align*}
	\hat\lambda 
	&= \lambda +\frac{\lambda((\lambda-1)\rho+\lambda)}{2k}
	+\frac{\lambda\left\{3(4\lambda^2-6\lambda+1)\rho^2+8(2\lambda^2-1)\rho+8\lambda^2\right\}}{24k^2}
	+\cdots 
\end{align*}
Note that the equation \eqref{E:hat-lambda} is of the form 
\begin{align*}
	\hat\lambda&=y e^{d_1\hat\lambda+d_2{\hat\lambda}^2}
	\eqtext{or}
	m[y^m]\hat\lambda = [t^{m-1}]
	e^{m(d_1t+d_2t^2)},
\end{align*}
where 
\[
	(y,d_1,d_2) := \Lpa{\lambda 
	e^{-\rho(\frac1{2k}+\frac1{3k^2})},
	\frac{\rho+1}{2k}
	-\frac{\rho(\rho-2)}{4k^2},
	\frac{(\rho-1)(3\rho+1)}{24k^2}}.
\]
Thus by Lagrange inversion formula:
\begin{align*}
	\hat\lambda 
	= \sum_{m\ge1}\frac{y^m}m
	\sum_{0\le j\le\tr{\frac12(m-1)}}
  \frac{m^{m-1-j}d_2^jd_1^{m-1-2j}}{j!(m-1-2j)!},
\end{align*}
which is expressible in terms of Hermite polynomials.

These choices then give
\begin{equation}\label{E:stir-lambda-menon}
  S(n,k)
  = e^{-\hat{\lambda}}
	\Lpa{1+O\lpa{k^{-3} \rho^3 (\hat{\lambda}
	+{\hat{\lambda}^{4}})}},
\end{equation}
where the $O$-term is $o(1)$ in the range $k\le \frac{4n}{\log n+(2+\ve)\log\log n}$.

\subsection{Numerical efficiency}

How do these refined expansions compare numerically? Since all four parameters $\lambda_{xy}$ with $x,y\in\{b,e\}$ are asymptotically equivalent to $\lambda$, we can express every first error term in terms of $\lambda$ alone, as shown in Table~\ref{tab:first_error_term}.
\begin{table}[htbp]
\centering
\renewcommand*{\arraystretch}{1.5}	
\caption{First error terms for the four types of expansions} 
\begin{tabular}{clcccc}
\hline
Type & First error term times $-\frac{24k^2}{\lambda}$ \\ \hline
Exp-Exp & $3\rho^2\lambda(\lambda - 2) -2\rho(\lambda^2-6\lambda+4) -\lambda^2$ \\
Exp-Bino & $ 3\rho^2\lambda(\lambda - 2) -2\rho(\lambda^2-3\lambda) -\lambda(\lambda+6)$\\
Bino-Bino & $3\rho^2\lambda(\lambda - 2) +2\rho(2\lambda^2+3\lambda)$\\
Bino-Exp & $3\rho^2\lambda(\lambda - 2) + 2\rho(2\lambda^2 + 6\lambda - 4)$\\ 
\hline
\end{tabular}
\label{tab:first_error_term}
\end{table}

Taking $n = 1000$ as a representative case, Figure~\ref{fig:four_images} displays the error functions \eqref{E:error-value} for both the original and modified estimates across four distribution scenarios: Exp-Exp \eqref{E:ee0} vs.~\eqref{E:ee}, Bino-Exp \eqref{E:be} vs.~\eqref{E:stir-lambda-be}, Bino-Bino \eqref{E:bb} vs.~\eqref{E:stir-lambda-bb}, and Exp-Bino \eqref{E:eb} vs.~\eqref{E:stir-lambda-eb}. To highlight the significant reduction in error, each subplot uses dual-axis scaling: the left $y$-axis corresponds to the original estimate (solid lines), while the right $y$-axis corresponds to the modified version (dashed lines).

\begin{figure}[!ht]
  \centering
  \begin{subfigure}[t]{0.45\textwidth}
    \centering
    \includegraphics[width=\textwidth]{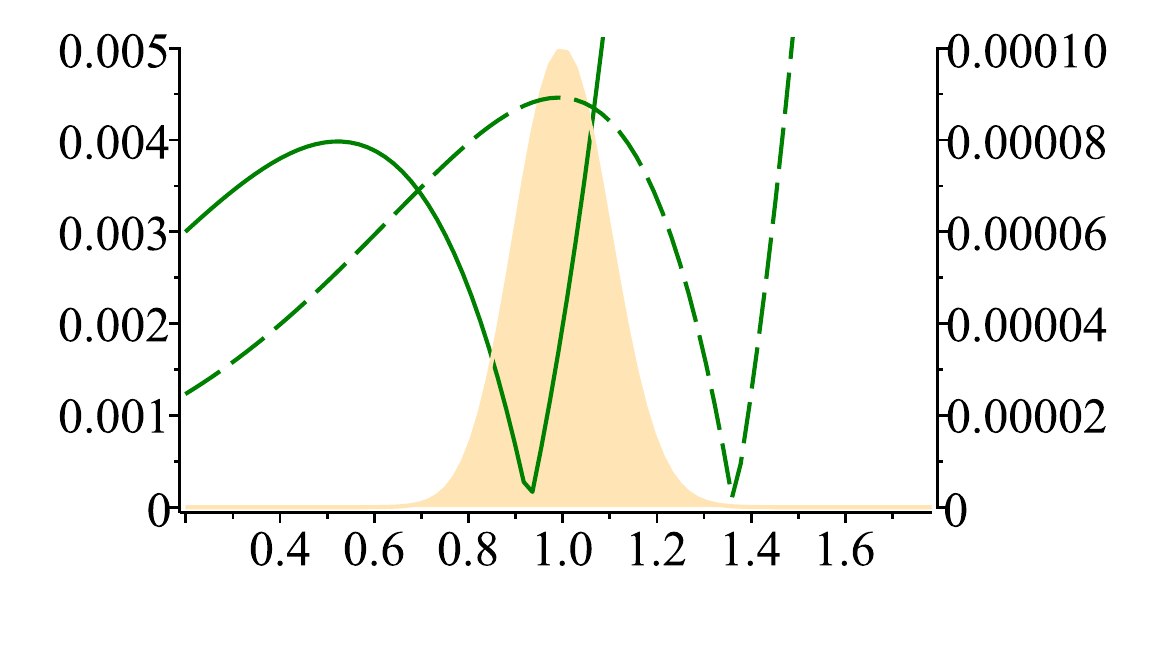}\\[-1.5ex]
    {\footnotesize
	\textcolor{green!50!black}{\raisebox{0.45ex}{\rule{1.6em}{1.0pt}}\, \eqref{E:ee0}} \,with $\alpha=0$ \,vs\,
	\textcolor{green!50!black}{\raisebox{0.45ex}{\rule{0.4em}{1pt}\hspace{0.2em}\rule{0.4em}{1pt}\hspace{0.2em}\rule{0.4em}{1pt}}\, \eqref{E:ee}}}
    \caption{Exp-Exp}
  \end{subfigure}
  \hfill
  \begin{subfigure}[t]{0.45\textwidth}
    \centering
    \includegraphics[width=\textwidth]{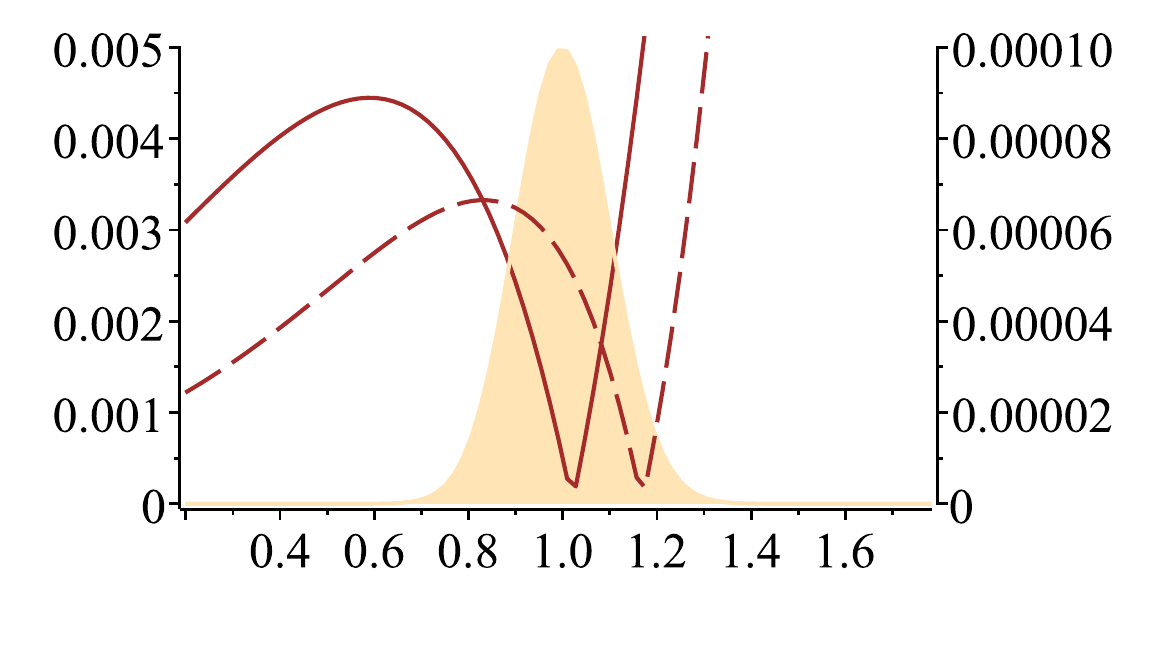}\\[-1.5ex]
    {\footnotesize
	\textcolor{brown}{\raisebox{0.45ex}{\rule{1.6em}{1.0pt}}\, \eqref{E:be}}
    \,vs\,
	\textcolor{brown}{\raisebox{0.45ex}{\rule{0.4em}{1pt}\hspace{0.2em}\rule{0.4em}{1pt}\hspace{0.2em}\rule{0.4em}{1pt}}\, \eqref{E:stir-lambda-be}}}
    \caption{Bino-Exp}
  \end{subfigure}

  \vspace{10pt}

  \begin{subfigure}[t]{0.45\textwidth}
    \centering
    \includegraphics[width=\textwidth]{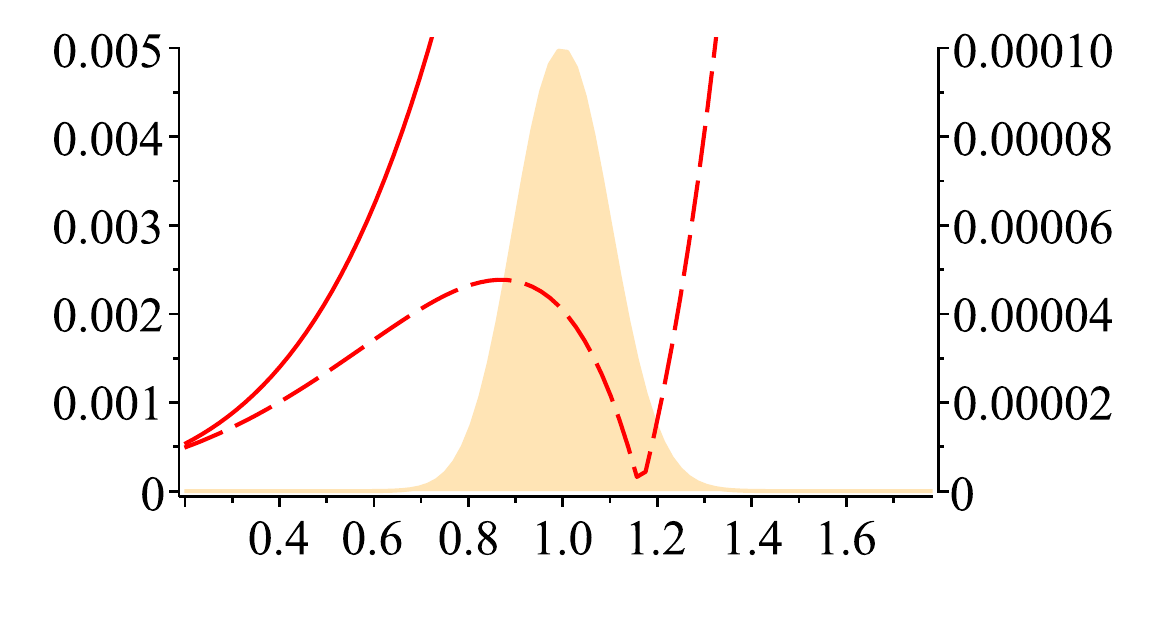}\\[-1.5ex]
    {\footnotesize
	\textcolor{red}{\raisebox{0.45ex}{\rule{1.6em}{1.0pt}}\, \eqref{E:bb}}
    \,vs\,
	\textcolor{red}{\raisebox{0.45ex}{\rule{0.4em}{1pt}\hspace{0.2em}\rule{0.4em}{1pt}\hspace{0.2em}\rule{0.4em}{1pt}}\, \eqref{E:stir-lambda-bb}}}
    \caption{Bino-Bino}
  \end{subfigure}
  \hfill
  \begin{subfigure}[t]{0.45\textwidth}
    \centering
    \includegraphics[width=\textwidth]{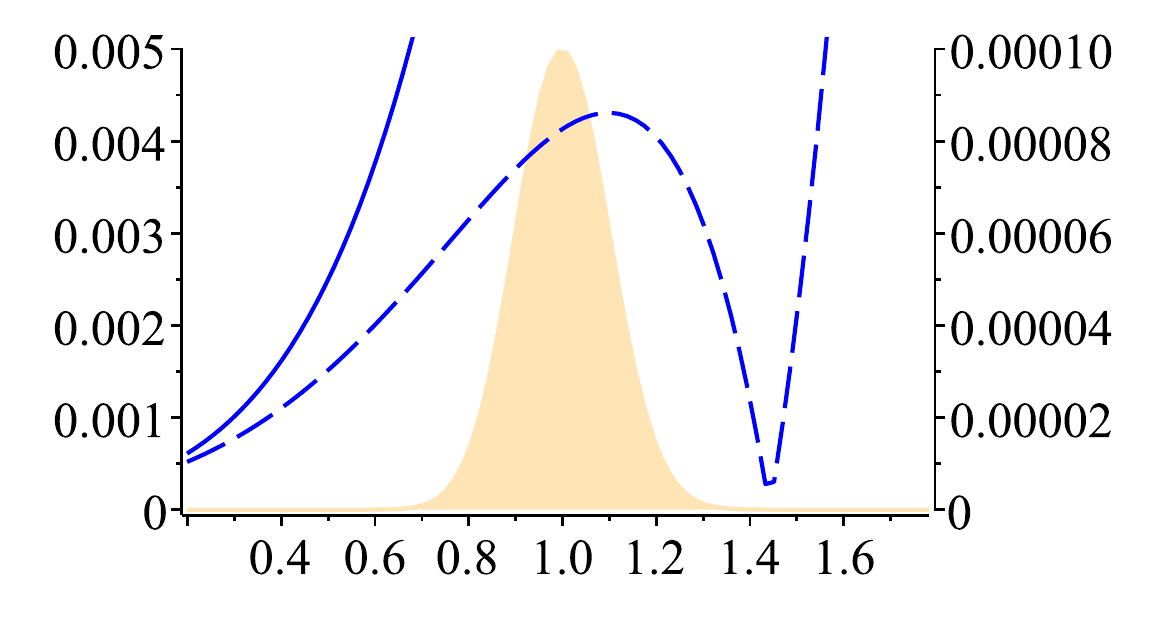}\\[-1.5ex]
    {\footnotesize
	\textcolor{blue}{\raisebox{0.45ex}{\rule{1.6em}{1.0pt}}\, \eqref{E:eb}}
    \,vs\,
	\textcolor{blue}{\raisebox{0.45ex}{\rule{0.4em}{1pt}\hspace{0.2em}\rule{0.4em}{1pt}\hspace{0.2em}\rule{0.4em}{1pt}}\, \eqref{E:stir-lambda-eb}}}
    \caption{Exp-Bino}
  \end{subfigure}
  
  \caption{All four subplots share the same $x$-axis scale (mean at $x=1$, $\mathrm{SD}=0.1$, $n=1000$). The left $y$-axis corresponds to the original error (solid lines) and the right $y$-axis to the modified error (dashed lines). The background soft yellow curve is a normalized histogram of the distribution $\frac{1}{B_n}\Stirling{n}{k}$, included for reference only; its values are not tied to either $y$-axis.
}\label{fig:four_images}
\end{figure}

Figure~\ref{fig:numerical-table} (with $n = 1000$) provides a final verification by comparing theoretical expectations with numerical results. The left panel isolates the modified estimates from Figure~\ref{fig:four_images} (specifically \eqref{E:ee}, \eqref{E:stir-lambda-be}, \eqref{E:stir-lambda-bb}, and \eqref{E:stir-lambda-eb}), while the right panel plots the actual numerical errors computed from the values in Table~\ref{tab:first_error_term}. The close agreement between the two panels confirms that our modified estimates accurately capture the true error behavior across all four distribution models.

\begin{figure}[!ht]
  \centering
  \begin{subfigure}[t]{0.45\textwidth}
    \centering
    \includegraphics[width=\textwidth]{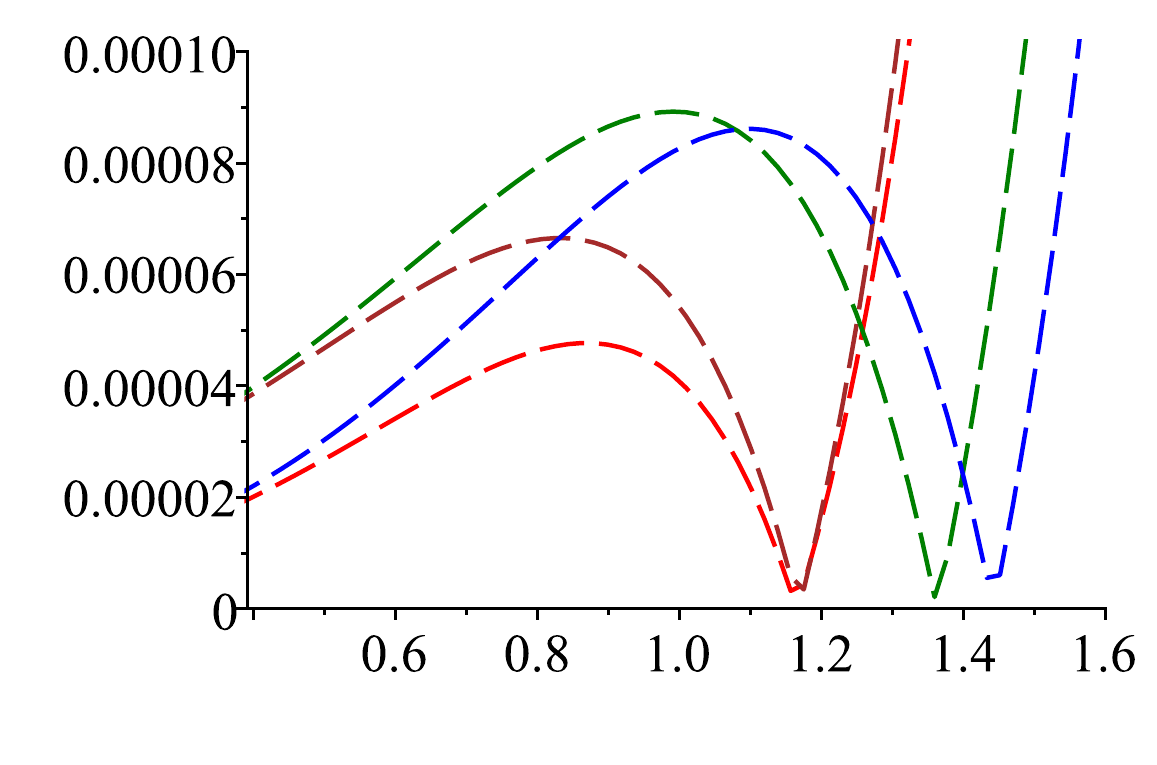}\\[-1.5ex]
    {\footnotesize
	\textcolor{green!50!black}{\raisebox{0.45ex}{\rule{0.4em}{1pt}\hspace{0.2em}\rule{0.4em}{1pt}\hspace{0.2em}\rule{0.4em}{1pt}}\,E-E \eqref{E:ee}}\hspace*{20pt}
    \textcolor{brown}{\raisebox{0.45ex}{\rule{0.4em}{1pt}\hspace{0.2em}\rule{0.4em}{1pt}\hspace{0.2em}\rule{0.4em}{1pt}}\,B-E \eqref{E:stir-lambda-be}}\\
    \textcolor{red}{\raisebox{0.45ex}{\rule{0.4em}{1pt}\hspace{0.2em}\rule{0.4em}{1pt}\hspace{0.2em}\rule{0.4em}{1pt}}\,B-B \eqref{E:stir-lambda-bb}}\hspace*{20pt}
    \textcolor{blue}{\raisebox{0.45ex}{\rule{0.4em}{1pt}\hspace{0.2em}\rule{0.4em}{1pt}\hspace{0.2em}\rule{0.4em}{1pt}}\,E-B \eqref{E:stir-lambda-eb}}}
    \caption{Comparison of the four modified error terms in Figure \ref{fig:four_images}.}
  \end{subfigure}
  \hfill
  \begin{subfigure}[t]{0.45\textwidth}
    \centering
    \includegraphics[width=\textwidth]{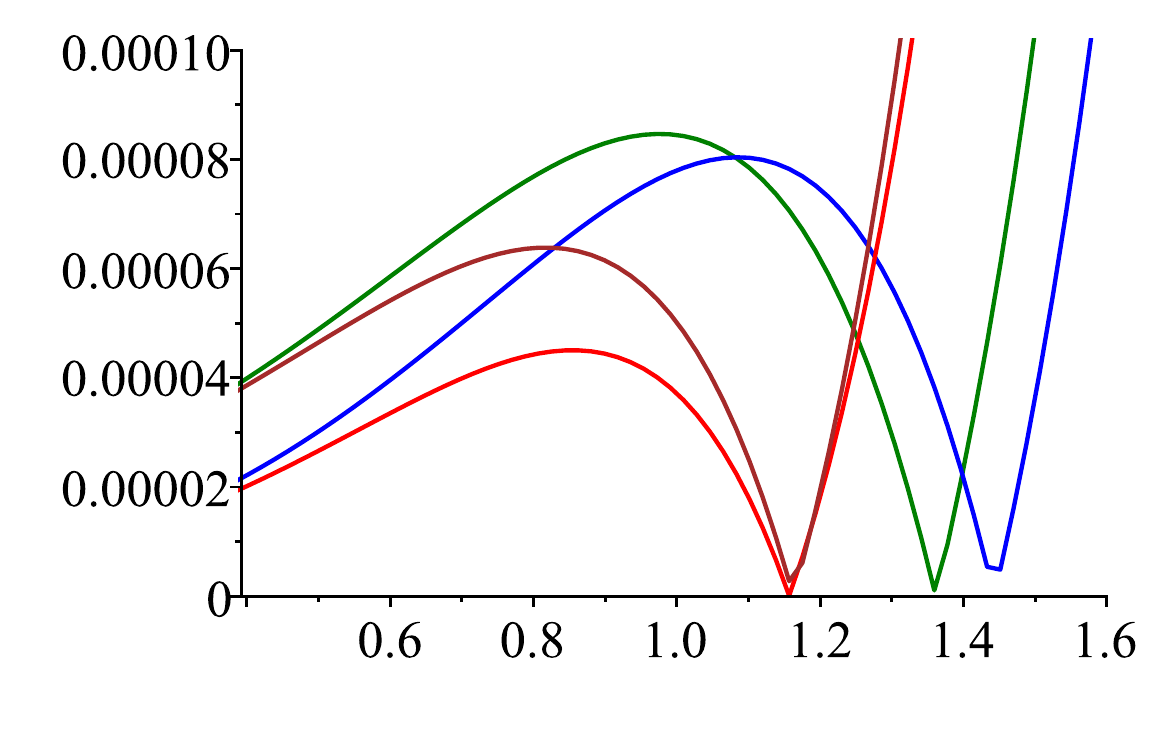}\\[-1.5ex]
    {\footnotesize
	\textcolor{green!50!black}{\raisebox{0.45ex}{\rule{1.6em}{1.0pt}}\,E-E}\hspace*{40pt}
    \textcolor{brown}{\raisebox{0.45ex}{\rule{1.6em}{1.0pt}}\,B-E}\\
    \textcolor{red}{\raisebox{0.45ex}{\rule{1.6em}{1.0pt}}\,B-B}\hspace*{40pt}
    \textcolor{blue}{\raisebox{0.45ex}{\rule{1.6em}{1.0pt}}\,E-B}}
    \caption{Comparison of the four leading error terms in Table \ref{tab:first_error_term}.}
  \end{subfigure}  
  \caption{With $n=1000$, the modified errors from Figure~\ref{fig:four_images} numerically coincide with the leading error terms in Table~\ref{tab:first_error_term}. The axis scaling and color-to-type mapping are consistent with Figure~\ref{fig:four_images}.}\label{fig:numerical-table}
\end{figure}

Figure~\ref{F:menon}, also with $n = 1000$, compares the error functions for the original Menon estimate \eqref{E:menon-2} (solid line, left $y$-axis) and its modified counterpart \eqref{E:stir-lambda-menon} (dashed line, right $y$-axis). The dual-axis scaling again reveals that the modification reduces errors by several orders of magnitude across the relevant range.

\begin{figure}[!ht]
\begin{center}
\includegraphics[width=0.6\textwidth]{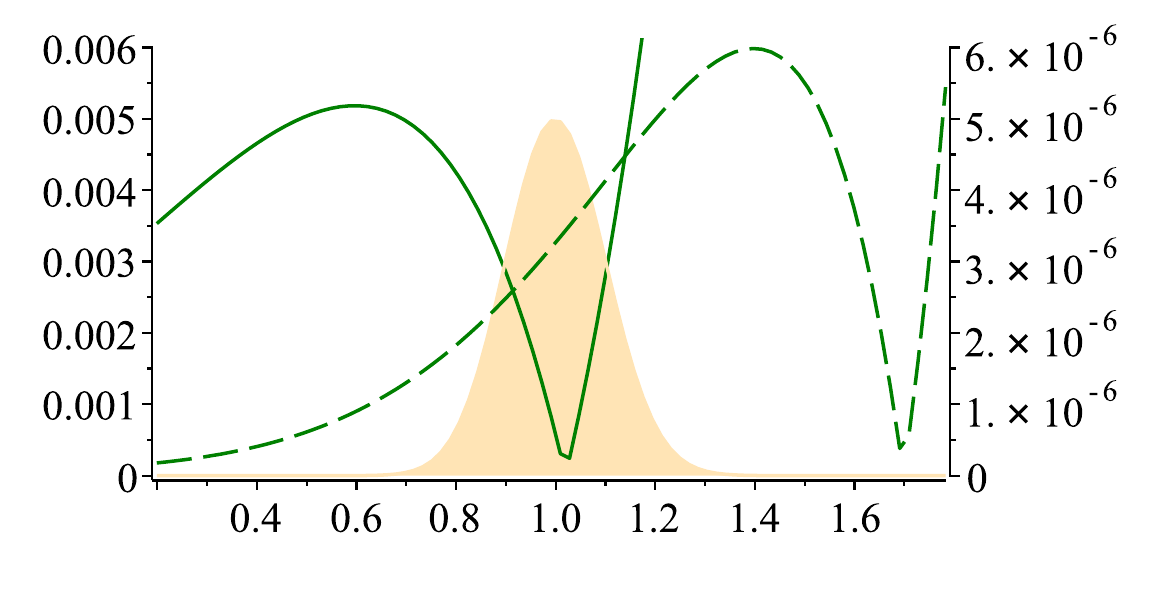}\\[-1.5ex]
{\footnotesize
	\textcolor{green!50!black}{\raisebox{0.45ex}{\rule{1.6em}{1.0pt}}\, \eqref{E:menon-2}}
    \quad vs \quad
	\textcolor{green!50!black}{\raisebox{0.45ex}{\rule{0.4em}{1pt}\hspace{0.2em}\rule{0.4em}{1pt}\hspace{0.2em}\rule{0.4em}{1pt}}\, \eqref{E:stir-lambda-menon}}}
\end{center}
\caption{\emph{Original vs.\ modified Menon error terms (mean $x=1$, $\mathrm{SD}=0.1$, $n=1000$). The left $y$-axis corresponds to the original error (solid) and the right $y$-axis to the modified error (dashed). The background soft yellow curve is a normalized histogram of $\frac{1}{B_n}\Stirling{n}{k}$, included for reference only; its values are not tied to either $y$-axis.}}\label{F:menon}
\end{figure}

\section*{Acknowledgements}

The authors used AI-assisted tools for language polishing, LaTeX editing, and consistency checks. All mathematical and historical claims were independently verified by the authors, who bear sole responsibility for the content of the paper.

\appendix

\numberwithin{equation}{section}

\section{Finer approximations for $\mathbb{E}(X_n)$ and $\mathbb{V}(X_n)$}
\label{S:App1}

Our crude analysis above does not provide optimal error terms in both mean and variance approximations \eqref{E:mu-var2}. By other analytic approaches, we can derive the finer asymptotic expansions:
\begin{align*}
	\mathbb{E}(X_n)
	&= \frac{n}{\omega_n}-1+\frac{\omega_n}{2(\omega_n+1)^2}
	+\frac{\omega_n^2(2\omega_n^3 + 8\omega_n^2 + 11\omega_n + 20)}
	{24(\omega_n + 1)^5n}+O\Lpa{\frac1{n^2}},\\
	\mathbb{V}(X_n)
	&= \frac{n}{\omega_n(\omega_n+1)}
	-1 +\frac{\omega_n(\omega_n - 1)}
	{2(\omega_n + 1)^4} -\frac{\omega_n^2(2\omega_n^3
	+10\omega_n^2-27\omega_n+40)}
	{24(\omega_n+1)^7n}+O\Lpa{\frac{1}{(n\,\omega_n)^2}}.
\end{align*}
These can be derived by at least three different approaches. 
\begin{itemize}
\item \emph{Direct saddle-point method}: one begins with the generating functions 
\[
	\sum_{n\ge0}\frac{z^n}{n!}
	\sum_{0\le k\le n}\Stirling{n}{k}e^{ks}
    = e^{e^s(e^z-1)}
	= e^{e^z-1}\lpa{1+(e^z-1)s + \tfrac12 e^z(e^z-1)s^2
	+ \cdots},
\]
and then apply the saddle-point method to derive asymptotic expansions for each of the  coefficients $n![z^n](e^z-1) e^{e^z-1}$ and $n![z^n]e^z (e^z-1)e^{e^z-1}$, and then normalize by the asymptotic expansion \eqref{E:Bn} of the Bell numbers.

\item \emph{Bell numbers asymptotics}: alternatively, in such special cases, we have the simpler relations in terms of Bell numbers:
\[
	\left\{
	\begin{split}
		n![z^n](e^z-1)e^{e^z-1}
		&= B_{n+1}-B_n,\\
		n![z^n]e^z(e^z-1)e^{e^z-1}
		&= B_{n+2}-2B_{n+1},
	\end{split}\right.
\]
so that the mean and the variance satisfy 
\[
	\mathbb{E}(X_n)
	= \frac{B_{n +1}}{B_n}-1,
	\eqtext{and}
	\mathbb{V}(X_n)
	= \frac{B_{n +2}}{B_n}
	- \Lpa{\frac{B_{n +1}}{B_n}}^2-1,
\]
respectively. We next apply \eqref{E:Bn} with \(n\) replaced by \(n+\ell\). Writing \(\omega_m:=W(m)\) for the principal Lambert \(W\) function (so \(W(x) e^{W(x)}=x\)), we expand \(W(n+\ell)\) about \(n\) via Taylor’s theorem. Using the standard derivatives
\[
W'(x)=\frac{W(x)}{x(1+W(x))},\qquad
W''(x)=-\frac{W(x)^2\,(W(x)+2)}{x^2(1+W(x))^3},
\]
we obtain, for \(\ell=O(1)\),
\[
\omega_{n+\ell}
= \omega_n
+ \frac{\ell\,\omega_n}{(\omega_n+1)\,n}
-\frac{\ell^2\,\omega_n^2\,(\omega_n+2)}{2\,(\omega_n+1)^3\,n^2}
+ O\left(\frac{|\ell|^3}{n^3}\right).
\]
The ensuing substitutions into \eqref{E:Bn} are straightforward but algebraically lengthy; they can be carried out routinely (e.g.\ with symbolic computation). See also \cite{Canfield1995, Czabarka2011} for closely related expansions.

\item \emph{Bivariate asymptotics and Quasi-powers framework}: It is also possible to derive first a uniform asymptotic approximation by saddle-point method to the (Touchard) polynomials (see \cite{Elbert2001}):
\[
	\sum_{1\le k\le n}\Stirling{n}{k}v^k
	= n![z^n]e^{v(e^z-1)},
\]
when $v=e^s$ lies in a neighborhood of unity, and then expand locally the resulting approximation at $s=0$, similar to the calculations of moments under the Quasi-powers framework; see \cite{Canfield1975, Canfield1977, Hwang1994} for related ideas.
\end{itemize}

\section{Justification of the Bino-Bino expansion}
\label{S:App-B}

Define polynomials \(c_m(w,j)\) by
\begin{equation}\label{eq:1.1}
\frac{(1-jt)^{w/t}}{(1-t)^{jw/t}}
=
\sum_{m\ge 0} c_m(w,j)\,t^m .
\end{equation}
Equivalently, the left side equals \(\exp\bigl(-wj\,d_j(t)\bigr)\), where
\[
d_j(t):=\sum_{r\ge 1}\frac{j^r-1}{r+1}\,t^r
\qquad (j\ge 1,\ 0\le t<1/j).
\]
Since every monomial in \(c_m(w,j)\) has the form \(w^u j^v\) with \(0\le u\le m\) and \(0\le v\le m+u\le 2m\), we have
\(\deg_w c_m\le m\) and \(\deg_j c_m\le 2m\).

\begin{lemma}\label{lem:bb-inequality}
Fix \(N\ge 0\). For \(w\ge 1\), integer \(j\ge 0\), and \(0\le jt\le \tfrac{1}{2}\),
\begin{equation}\label{eq:1.6}
\left|
\frac{(1-jt)^{w/t}}{(1-t)^{jw/t}}
-
\sum_{0\le m \le N} c_m(w,j)\,t^m
\right|
\le C_N\,(j^2wt)^{N+1},
\end{equation}
where \(C_N\le 2^{N+1}e^{1/2}\).
\end{lemma}
\begin{proof}
The case \(j=0\) is trivial; assume \(j\ge 1\). Since \(j^r-1\le j^r\),
\[
0\le d_j(t)\le \sum_{r\ge 1}\frac{(jt)^r}{r+1}
\le \frac{jt}{2(1-jt)}\le jt.
\]
Setting \(x=wj\,d_j(t)\) in the standard Taylor remainder for \(e^{-x}\) gives
\begin{equation}\label{eq:1.8}
\left|
\frac{(1-jt)^{w/t}}{(1-t)^{jw/t}}
-
\sum_{0\le s\le N}\frac{(-1)^s}{s!}(wj)^s d_j(t)^s
\right|
\le
\frac{(j^2wt)^{N+1}}{(N+1)!}.
\end{equation}

Write \(d_j(t)^s=\sum_{m\ge s}b_{s,m}\,t^m\). The coefficients of \(d_j(t)\) are bounded by those of \(\frac{jt}{2(1-jt)}\), so the standard tail bound for a power series
\(\phi(x)=\sum_{\ell\ge0}\phi_\ell x^\ell\) with nonnegative coefficients converging at \(x=R\),
\[
\left|\phi(x)-\sum_{0\le \ell\le N}\phi_\ell x^\ell\right|
\le \left(\frac xR\right)^{N+1}\phi(R),
\]
applied to \(\phi(x)=x^s(1-x)^{-s}\) with \(x=jt\), \(R=\frac{1}{2}\), gives
\[
\left|
d_j(t)^s-\sum_{s\le m\le N}b_{s,m}\,t^m
\right|
\le 2^{N+1-s}(jt)^{N+1}.
\]
Substituting into \eqref{eq:1.8}, the additional truncation error is at most
\[
\sum_{1\le s\le N}\frac{(wj)^s}{s!}\,2^{N+1-s}(jt)^{N+1}
\le (e^{1/2}-1)2^{N+1}\,(j^2wt)^{N+1},
\]
where we used \(w^s j^{N+1+s}\le (j^2w)^{N+1}\) for \(1\le s\le N\). Combining with \eqref{eq:1.8} proves the lemma.
\end{proof}

For nonnegative integers \(x\le k/2\), the lemma gives (with \(t=\frac1k\) and \(w=\rho=\frac nk\))
\begin{equation}\label{eq:lem}
\frac{(1-\frac{x}{k})^n}{(1-\frac{1}{k})^{nx}}
=
\sum_{0\le m\le N}\frac{c_m(\rho,x)}{k^m}
+ O\left(\left(\frac{nx^2}{k^2}\right)^{\!N+1}\right).
\end{equation}

\begin{thm}
Assume
\begin{align}\label{E:kk3}
\frac{n}{\log n}\le k\le \frac{2n}{\log n+6},
\eqtext{or in terms of $\lambda$}
\frac{1}{\log n}\le \lambda\le \frac{2}{e^3}\cdot \frac{\sqrt n}{\log n+6},
\end{align}
where \(\lambda:=k e^{-n/k}\). Then, for every fixed \(N\ge 0\),
\begin{equation}\label{eq:main-binobino1}
S(n,k)
=
\sum_{0\le j\le k}(-1)^j\binom{k}{j}\left(1-\frac{1}{k}\right)^{nj}
\left(\sum_{0\le m\le N}\frac{c_m(\rho,j)}{k^m}\right)
+ R_N(n,k),
\end{equation}
with
\[
R_N(n,k)
= O\left(\left(1-\frac{\lambda}{k}\right)^k
\left(\frac{\lambda\log n}{\sqrt n}\right)^{2N+1}\right),
\]
where the implied constant depends only on \(N\).
\end{thm}

\begin{proof}
Recall \(g(x):=e^{nx/k}(1-x/k)^n\) (see \eqref{E:fg}). By Theorem~\ref{T:finite-diff-exp},
\begin{equation}\label{eq:trunc}
S(n,k)
=
(-1)^k\sum_{0\le j\le 2N}\binom{k}{j}
\Bigl(\nabla_x^{k-j}\Lpa{\frac{\lambda}k}^x\Bigr)\Big|_{x=k}\,
\bigl(\nabla_x^j g(x)\bigr)\big|_{x=j}
+ O\Lpa{\Lpa{1-\frac{\lambda}{k}}^k\Lpa{\frac{\lambda\log n}{\sqrt n}}^{2N+1}}.
\end{equation}
Set \(G_N(x):=\sum_{0\le m\le N}c_m(\rho,x)\,k^{-m}\). By \eqref{eq:lem}, for integers \(0\le x\le 2N\),
\[
g(x)=g(1)^x\,G_N(x)+E_N(x),
\eqtext{and}
E_N(x)=g(1)^x\,O\bigl((nx^2/k^2)^{N+1}\bigr).
\]
Since \(0\le g(1)\le 1\), \(x\le 2N\), and
\(\nabla_x^{k-j}u^x\big|_{x=k}=(-1)^{k-j}(1-u)^k\bigl(\frac{u}{1-u}\bigr)^j\), the contribution of \(E_N\) to \eqref{eq:trunc} is
\(O\bigl((1-\frac{\lambda}k)^k(n/k^2)^{N+1}(1+\lambda)^{2N}\bigr)\).
Under \eqref{E:kk3}, we have \(n/k^2\asymp (\log n)^2/n\) and
\(1/\log n\le \lambda=O(\sqrt n/\log n)\). Hence
\[
\left(\frac{n}{k^2}\right)^{N+1}(1+\lambda)^{2N}
= O\Lpa{\Lpa{\frac{\lambda\log n}{\sqrt n}}^{2N+1}},
\]
so this contribution is absorbed into the remainder in \eqref{eq:trunc}.

\medskip
\noindent\textbf{Main term.}
Write \(\Sigma_{2N}\) for the sum in \eqref{eq:trunc} with \(g(x)\) replaced by \(g(1)^xG_N(x)\), and decompose \(\Sigma_{2N}=\Sigma_{\le k}-E_N^{[2]}\), where \(\Sigma_{\le k}\) extends the sum to \(j\le k\) and \(E_N^{[2]}\) is the tail \(2N+1\le j\le k\). By Leibniz's formula \eqref{E:Leibniz},
\[
\Sigma_{\le k}
=(-1)^k\nabla_x^k\bigl(\lpa{\tfrac{\lambda}k}^xg(1)^xG_N(x)\bigr)\big|_{x=k}
=(-1)^k\nabla_x^k\bigl(\bigl(1-\tfrac{1}{k}\bigr)^{nx}G_N(x)\bigr)\big|_{x=k},
\]
which expands to the right side of \eqref{eq:main-binobino1}.

\medskip
\noindent\textbf{Tail bound.}
Every monomial in \(c_m(\rho,x)k^{-m}\) has the form \(n^ux^vk^{-m-u}\) with \(v\le m+u\) and \(u\le m\), so \(\nabla_x^s G_N(x)\equiv 0\) for \(s>2N\) and, for \(0\le s\le 2N\),
\[
\nabla_x^sG_N(x)\big|_{x=s}
= O\Bigl(\sum_{\lceil s/2\rceil\le m\le N}\Lpa{\frac n{k^2}}^m\Bigr)
= O\bigl(n^{-s/2}(\log n)^{2N}\bigr).
\] 
Also,
\[
\bigl|\nabla_x^r g(1)^x\bigr|\big|_{x=j}
= g(1)^{\,j-r}|1-g(1)|^r
= O\lpa{n^{-r}(\log n)^{2r}},
\]
since \(1-g(1)=O(n/k^2)=O((\log n)^2/n)\).
By Leibniz's rule,
\[
\bigl|\nabla_x^j\bigl(g(1)^xG_N(x)\bigr)\bigr|\big|_{x=j}
= O\bigl((\log n)^{2N}\,(4/\sqrt{n})^{\,j}\bigr)
\qquad(j\ge 2N+1).
\]
Since \(\frac{k\cdot\frac{\lambda}k}{1-\frac{\lambda}k}\asymp\lambda\), substituting into the bound for \(E_N^{[2]}\) gives
\[
|E_N^{[2]}|
= O\left(\Lpa{1-\frac\lambda k}^k(\log n)^{2N}\sum_{j\ge 2N+1}\frac{1}{j!}\Bigl(\frac{4\lambda}{\sqrt n}\Bigr)^{\!j}\right)
= O\left(\left(1-\frac{\lambda}{k}\right)^k\left(\frac{\lambda\log n}{\sqrt n}\right)^{2N+1}\right),
\]
completing the proof.
\end{proof}

Since \(c_m^{[\mathrm{bb}]}(j)=c_m(\rho,j)\), the theorem is precisely the Bino-Bino expansion \eqref{E:bb} in the range \eqref{E:kk3}.

\section{Bino--Exp case (outline)}
\label{S:App-C}

\begin{lemma}[Bino--Exp truncation bound]\label{lem:be-analogue}
Fix \(N\ge 0\). For \(w\ge 1\), integer \(x\ge 0\), and \(0\le xt\le \frac{1}{2}\),
\begin{equation}\label{eq:be-tail}
\left|
(1-xt)^{w/t}e^{xw}
-
\sum_{0\le m\le N} c_m^{[\mathrm{be}]}(x,w)\,t^m
\right|
\le C_N\,(x^2wt)^{N+1},
\end{equation}
where \(C_N\le 2^{N+1}e^{1/2}\).
\end{lemma}

\begin{proof}
Write
\[
(1-xt)^{w/t}e^{xw}
=
\exp\bigl(-w\,d_x^{[\mathrm{be}]}(t)\bigr),
\eqtext{where}
d_x^{[\mathrm{be}]}(t):=\sum_{r\ge 1}\frac{x^{r+1}}{r+1}\,t^r.
\]
If \(xt\le \frac{1}{2}\), then
\[
0\le d_x^{[\mathrm{be}]}(t)
\le \frac{1}{2}\sum_{r\ge 1}x^{r+1}t^r
=
\frac{x^2t}{2(1-xt)}
\le x^2t.
\]
Thus Taylor's theorem gives
\[
\left|
e^{-w d_x^{[\mathrm{be}]}(t)}
-
\sum_{0\le r\le N}\frac{(-1)^r}{r!}w^r\bigl(d_x^{[\mathrm{be}]}(t)\bigr)^r
\right|
\le
\frac{(x^2wt)^{N+1}}{(N+1)!}.
\]
Also by the coefficient-wise bound
\[
d_x^{[\mathrm{be}]}(t)\preceq \frac{x^2t}{2(1-xt)},
\]
we see that the same positive-coefficient tail estimate used in Lemma~\ref{lem:bb-inequality}, with \(R=\frac{1}{2}\), shows that truncating each power \(\bigl(d_x^{[\mathrm{be}]}(t)\bigr)^r\) at degree \(N\) contributes at most
\[
2^{N+1}(e^{1/2}-1)\,(x^2wt)^{N+1}.
\]
Combining the two bounds proves \eqref{eq:be-tail}.
\end{proof}

\begin{thm}[Bino--Exp analogue of Theorem~1.2]\label{thm:be-thm12}
Assume \eqref{E:kk3}. Then, for every fixed \(N\ge 0\),
\begin{equation}\label{eq:be-main}
S(n,k)
=
\sum_{0\le j\le k}(-1)^j\binom{k}{j}\Lpa{\frac{\lambda}k}^j
\left(
\sum_{0\le m\le N}\frac{c_m^{[\mathrm{be}]}(j,\rho)}{k^m}
\right)
+
R_N^{[\mathrm{be}]}(n,k),
\end{equation}
where
\begin{equation}\label{eq:be-rem}
R_N^{[\mathrm{be}]}(n,k)
=
O_N\left(
\Lpa{1-\frac{\lambda}{k}}^k
\left(\frac{\lambda\log n}{\sqrt n}\right)^{2N+1}
\right).
\end{equation}
\end{thm}

\begin{proof}
Start from the finite-difference truncation \eqref{eq:trunc}, always with \(g(x):=e^{\rho x}\left(1-\frac{x}{k}\right)^n\). For \(0\le x\le 2N\) and \(n\) large, \(\frac xk\le \frac{1}{2}\), so Lemma~\ref{lem:be-analogue} with \(t=\frac{1}{k}\) and \(w=\rho\) yields
\[
g(x)=G_N^{[\mathrm{be}]}(x)+E(x),
\]
where
\[
G_N^{[\mathrm{be}]}(x):=\sum_{0\le m\le N}\frac{c_m^{[\mathrm{be}]}(x,\rho)}{k^m},
\qquad
E(x)=O\left(\left(\frac{x^2n}{k^2}\right)^{N+1}\right).
\]
Hence
\[
\bigl|\nabla_x^jE(x)\bigr|_{x=j}
\le
2^j\max_{0\le u\le j}|E(u)|
=O\left(
\left(\frac{n}{k^2}\right)^{N+1}\right)
\qquad(0\le j\le 2N),
\]
so replacing \(g\) by \(G_N^{[\mathrm{be}]}\) in \eqref{eq:trunc} contributes at most
\[
\Lpa{1-\frac{\lambda}{k}}^k\left(\frac{n}{k^2}\right)^{N+1}(1+\lambda)^{2N}.
\]
Under \eqref{E:kk3}, this is
\[
O_N\left(
\Lpa{1-\frac{\lambda}{k}}^k
\left(\frac{\lambda\log n}{\sqrt n}\right)^{2N+1}
\right).
\]

Now \(G_N^{[\mathrm{be}]}\) is a polynomial in \(x\) of degree at most \(2N\), so
\[
\nabla_x^jG_N^{[\mathrm{be}]}\equiv 0
\qquad (j>2N).
\]
Thus the truncated sum extends exactly from \(0\le j\le 2N\) to \(0\le j\le k\), and Leibniz's formula gives
\[
(-1)^k\nabla_x^k\Bigl(\Lpa{\frac{\lambda}k}^xG_N^{[\mathrm{be}]}(x)\Bigr)\Big|_{x=k}
=
\sum_{0\le j\le k}(-1)^j\binom{k}{j}\Lpa{\frac{\lambda}k}^jG_N^{[\mathrm{be}]}(j),
\]
which is precisely \eqref{eq:be-main}. In contrast with the Bino--Bino case, there is no large-\(j\) tail, because the approximant is already a polynomial.
\end{proof}

\section{Justification of the Exp--Exp case (outline)}
\label{S:App-D}

We compare the Exp--Exp truncation with the Charlier--Poisson truncation.

\medskip
\noindent\textbf{Auxiliary polynomials.}
Define polynomials \(a_m(j)\) by
\begin{equation}\label{eq:def-a-polynomials}
\prod_{0\le l<j}(1-lt)=\sum_{m\ge 0}a_m(j)\,t^m,
\qquad j\in \mathbb{Z}_{\ge 0}.
\end{equation}
Then \(a_m(x)\) is a polynomial in \(x\) of degree \(2m\), with \(a_0\equiv 1\) and \(a_m(0)=\cdots=a_m(m)=0\) for \(m\ge 1\). Since (with $t=\frac{1}{k}$)
\[
\binom{k}{j}
=
\frac{k^j}{j!}\prod_{0\le l<j}(1-lt),
\]
we have
\begin{equation}\label{eq:binom-am}
\binom{k}{j}\Lpa{\frac{\lambda}k}^j
=
\frac{\lambda^j}{j!}\sum_{m\ge 0}a_m(j)\,t^m,
\end{equation}
and this identity remains valid for \(j>k\), because the product in \eqref{eq:def-a-polynomials} then vanishes.

\medskip
\noindent\textbf{The Exp--Exp coefficients.}
Recall the expansion \eqref{E:Charlier}. Then by the definition of $\lambda$
\[
\binom{k}{j}\left(1-\frac{j}{k}\right)^n
=
\binom{k}{j}\Lpa{\frac{\lambda}k}^j\cdot \Bigl(e^{j/k}\Lpa{1-\frac jk}\Bigr)^n
=
\frac{\lambda^j}{j!}
\sum_{m,l\ge 0}
a_m(j)\,\frac{\tau_l(n)}{l!}\cdot\frac{j^l}{k^{m+l}}.
\]
Collecting powers of \(k^{-1}\), we obtain
\begin{equation}\label{eq:def-cee-via-a-tau}
\binom{k}{j}\Lpa{1-\frac{j}{k}}^n
=
\frac{\lambda^j}{j!}\sum_{s\ge 0}\frac{c_s^{[\mathrm{ee}]}(n,j)}{k^s},
\eqtext{where}
c_s^{[\mathrm{ee}]}(n,j)
=
\sum_{\substack{m+l=s\\ m,l\ge 0}}
a_m(j)\,\frac{\tau_l(n)}{l!}\,j^l.
\end{equation}

The truncated Exp--Exp main term is
\[
S_{ee}^{(N)}(n,k)
:=
\sum_{0\le s\le N}\frac{1}{k^s}
\sum_{j\ge 0}(-1)^j\frac{\lambda^j}{j!}\,c_s^{[\mathrm{ee}]}(n,j).
\]

\medskip
\noindent\textbf{Comparison with Poisson--Charlier expansion.}
For \(m\ge 0\), define $h_m(y) := [t^m](1-yt)^{1/t}$. For odd \(N\), Theorem~\ref{Poi_Char_expansion} gives
\begin{equation}\label{eq:s-cp-restatement}
S(n,k)=S_{CP}^{(N)}(n,k)
+O\left(
\Lpa{1-\frac{\lambda}k}^k
\left(\frac{\lambda\log n}{\sqrt n}\right)^{N+1}
\right),
\end{equation}
where
\begin{equation}\label{eq:s_cp}
S_{CP}^{(N)}(n,k)
=
\sum_{0\le l\le N}\frac{\tau_l(n)}{l!}
\sum_{0\le j\le k}(-1)^j\binom{k}{j}\Lpa{\frac{\lambda}k}^j\left(\frac{j}{k}\right)^l.
\end{equation}

\begin{prop}\label{prop:cp-ee-identity}
For every integer \(N\ge 0\),
\begin{equation}\label{eq:cp-ee}
S_{CP}^{(N)}(n,k)-S_{ee}^{(N)}(n,k)
=
\sum_{0\le l\le N}\frac{\tau_l(n)}{l!}\,(-1)^l
\frac{d^l}{dn^l}
f_{N-l}\left(\lambda,\frac{1}{k}\right),
\end{equation}
where $f_M(y,t):=(1-yt)^{1/t}-\sum_{0\le m\le M}h_m(y)t^m$.
\end{prop}

\begin{proof}
By \eqref{eq:binom-am}, equation \eqref{eq:s_cp} becomes
\[
S_{CP}^{(N)}(n,k)
=
\sum_{0\le l\le N}\frac{\tau_l(n)}{l!}
\sum_{m\ge 0}t^m
\sum_{j\ge 0}(-1)^j\frac{\lambda^j}{j!}a_m(j)(jt)^l.
\]
On the other hand, \eqref{eq:def-cee-via-a-tau} gives
\[
S_{ee}^{(N)}(n,k)
=
\sum_{0\le l\le N}\frac{\tau_l(n)}{l!}
\sum_{0\le m\le N-l}t^m
\sum_{j\ge 0}(-1)^j\frac{\lambda^j}{j!}a_m(j)(jt)^l.
\]
Subtracting, only the terms \(m>N-l\) remain. Now
\[
(jt)^l\lambda^j
=
(-1)^l\frac{d^l}{dn^l}\lambda^j,
\]
because \(\frac{d}{dn}\lambda^j=-jt\,\lambda^j\) with \(k\) fixed. Also,
\[
h_m(y)=\sum_{j\ge 0}(-1)^j\frac{y^j}{j!}a_m(j),
\]
so termwise differentiation yields
\[
\sum_{j\ge 0}(-1)^j\frac{\lambda^j}{j!}a_m(j)(jt)^l
=
(-1)^l\frac{d^l}{dn^l}h_m(\lambda).
\]
Summing over \(m>N-l\) gives \(f_{N-l}(\lambda,1/k)\), proving \eqref{eq:cp-ee}.
\end{proof}

\medskip
\noindent\textbf{Derivative bounds.}
The tail estimate in Lemma~\ref{lem:be-analogue}, with \(w=1\) and division by \(e^y\), gives
\[
|f_M(y,t)|\le C_M e^{-y}(y^2t)^{M+1}
\qquad
(y\ge 0,\ yt\le \tfrac{1}{2}).
\]
We use the following two consequences.

\begin{lemma}\label{lem:fy-derivatives}
Fix integers \(M,r\ge 0\). There exists \(C_{M,r}>0\) such that, whenever \(0<t\le 1\), \(y\ge 0\), and \(yt\le \frac{1}{2}\),
\[
\left|\partial_y^r f_M(y,t)\right|
\le
C_{M,r}\,e^{-y}\,t^{M+1}(1+y^{2M+2}).
\]
\end{lemma}

\begin{prop}\label{prop:n-derivative-f-lsharp}
Fix integers \(N\ge 0\) and \(0\le l\le N\). For \(k\ge 1\) and \(q\le \frac{1}{2}\),
\[
\left|
\frac{d^l}{dn^l}f_{N-l}\left(\lambda,\frac{1}{k}\right)
\right|
\le
C_{N,l}\,e^{-\lambda}\,k^{-(N+1)}\bigl(1+\lambda^{2N-l+2}\bigr).
\]
\end{prop}

\begin{thm}\label{thm:ee-main}
Fix an odd integer \(N\ge 0\). If
\[
\frac{n}{W(n)}\le k\le \frac{2n}{\log n+6},
\qquad\text{equivalently}\qquad
1\le \lambda\le \frac{2}{e^3}\cdot \frac{\sqrt n}{\log n+6},
\]
then
\[
S(n,k)
=
S_{ee}^{(N)}(n,k)
+
O_N\left(
e^{-\lambda}
\left(\frac{\lambda\log n}{\sqrt n}\right)^{N+1}
\right).
\]
\end{thm}

\begin{proof}
By \eqref{eq:s-cp-restatement} and Proposition~\ref{prop:cp-ee-identity},
\[
S(n,k)-S_{ee}^{(N)}(n,k)
=
O\left(
e^{-\lambda}
\left(\frac{\lambda\log n}{\sqrt n}\right)^{N+1}
\right)
+
\sum_{0\le l\le N}\frac{\tau_l(n)}{l!}\,(-1)^l
\frac{d^l}{dn^l}f_{N-l}\left(\lambda,\frac{1}{k}\right).
\]
Hence, by Proposition~\ref{prop:n-derivative-f-lsharp},
\[
\left|S_{CP}^{(N)}(n,k)-S_{ee}^{(N)}(n,k)\right|
=O\left(
e^{-\lambda}
\sum_{0\le l\le N}
|\tau_l(n)|\,k^{-(N+1)}\bigl(1+\lambda^{2N-l+2}\bigr)\right).
\]
Using the upper bound \(|\tau_l(n)|=O(n^{l/2})\) (see \eqref{E:tau-bound}) and \(k\asymp n/\log n\), this is bounded above by 
\[
e^{-\lambda}
\sum_{0\le l\le N}
(\log n)^{N+1}n^{-(N+1-l/2)}
\bigl(1+\lambda^{2N-l+2}\bigr).
\]
Since \(\lambda\ge 1\) and \(l\le N\),
\[
n^{-(N+1-l/2)}
=O\left(
n^{-(N+1)/2}\lambda^{N+1}\right),
\]
while
\[
n^{-(N+1-l/2)}\lambda^{2N-l+2}
=
n^{-(N+1)/2}\lambda^{N+1}\cdot
n^{-(N+1-l)/2}\lambda^{N-l+1}
=O\left(
n^{-(N+1)/2}\lambda^{N+1}\right),
\]
because \(\lambda=O(\sqrt n/\log n)=O(\sqrt n)\). Therefore
\[
S_{CP}^{(N)}(n,k)-S_{ee}^{(N)}(n,k)
=
O_N\left(
e^{-\lambda}
\left(\frac{\lambda\log n}{\sqrt n}\right)^{N+1}
\right).
\]
Combining this with \eqref{eq:s-cp-restatement} proves the theorem.
\end{proof}

\bibliographystyle{abbrv}
\bibliography{stir2-refs-2026}
\end{document}